\documentclass[reqno]{amsart}

\usepackage{amssymb,amsmath,amsthm,xcolor,enumerate,comment,enumitem,bm,longtable,hyperref,cleveref}
\usepackage[colorinlistoftodos]{todonotes}
\usepackage[normalem]{ulem}
\usepackage[utf8]{inputenc}

\usepackage{tikz-cd}
\usetikzlibrary{positioning}

\usepackage[margin=1in]{geometry}

\allowdisplaybreaks

\newtheorem{thrm}{Theorem}[section]
\newtheorem{cor}[thrm]{Corollary}
\newtheorem{lem}[thrm]{Lemma}
\newtheorem{prop}[thrm]{Proposition}

\theoremstyle{definition}
\newtheorem{defn}[thrm]{Definition}
\newtheorem{exm}[thrm]{Example}

\newtheorem{rem}[thrm]{Remark}

\crefrangeformat{equation}{#3(#1)#4--#5(#2)#6}

\crefname{thrm}{Theorem}{Theorems}
\crefname{lem}{Lemma}{Lemmas}
\crefname{cor}{Corollary}{Corollaries}
\crefname{prop}{Proposition}{Propositions}
\crefname{defn}{Definition}{Definitions}
\crefname{exm}{Example}{Examples}
\crefname{rem}{Remark}{Remarks}
\crefname{section}{Section}{Sections}
\crefname{equation}{\unskip}{\unskip}
\crefname{enumi}{\unskip}{\unskip}
\crefname{subsection}{Subsection}{Subsections}

\makeatletter
\newcommand{\mylabel}[2]{#2\def\@currentlabel{#2}\label{#1}}
\makeatother

\renewcommand{\iff}{\Leftrightarrow}

\newcommand{\impl}{\Rightarrow}

\newcommand{\id}{\mathrm{id}}

\newcommand{\I}{\mathcal I}

\newcommand{\cG}[1]{\mathcal G{(#1)}}

\newcommand{\End}[1]{\operatorname{\mathrm{End}}{#1}}

\newcommand{\mend}{\operatorname{end}}

\DeclareMathOperator{\im}{im}

\newcommand{\m}{{}^{-1}}

\newcommand{\0}{\theta}
\newcommand{\ve}{\varepsilon}
\newcommand{\e}{\epsilon}

\newcommand{\af}{\alpha}
\newcommand{\bt}{\beta}
\newcommand{\lb}{\lambda}

\newcommand{\gm}{\gamma}
\newcommand{\f}{\varphi}
\newcommand{\s}{\sigma}

\newcommand{\dl}{\delta}
\newcommand{\z}{\zeta}

\newcommand{\B}{\mathcal B}

\newcommand{\tl}{\tilde}

\newcommand{\sst}{\subseteq}

\newcommand{\bd}{\mathbf{d}}
\newcommand{\br}{\mathbf{r}}

\newcommand{\irr}[1]{\operatorname{\mathrm{irr}}{(#1)}}
\begin{document}

\title[$H^3(S,A)$ and crossed module extensions of $A$ by $S$]{Inverse semigroup cohomology\\ and crossed module extensions\\ of semilattices of groups by inverse semigroups}
	\author{Mikhailo Dokuchaev}
	\address{Instituto de Matem\'atica e Estat\'istica, Universidade de S\~ao Paulo,  Rua do Mat\~ao, 1010, S\~ao Paulo, SP,  CEP: 05508--090, Brazil}
	\email{dokucha@gmail.com}
	\author{Mykola Khrypchenko}
	\address{Departamento de Matem\'atica, Universidade Federal de Santa Catarina, Campus Reitor Jo\~ao David Ferreira Lima, Florian\'opolis, SC,  CEP: 88040--900, Brazil}
	\email{nskhripchenko@gmail.com}
	\author{Mayumi Makuta}
	\address{Instituto de Matem\'atica e Estat\'istica, Universidade de S\~ao Paulo,  Rua do Mat\~ao, 1010, S\~ao Paulo, SP,  CEP: 05508--090, Brazil}
	\email{may.makuta@gmail.com} 
	\begin{abstract} 
		We define and study the notion of a crossed module over an inverse semigroup and the corresponding $4$-term exact sequences, called crossed module extensions. For a crossed module $A$ over an $F$-inverse monoid $T$, we show that equivalence classes of admissible crossed module extensions of $A$ by $T$ are in a one-to-one correspondence with the elements of the cohomology group $H^3_\le(T^1,A^1)$.
	\end{abstract}

	\maketitle
	
	\tableofcontents
	
	\section*{Introduction}  

The interpretations of low-dimensional group cohomology which deal with some types of short exact sequences go back to the early works by Schreier~\cite{Schreier26,Schreier25}. In the modern terminology, each $2$-cocycle $\s\in Z^2(G,A)$ defines the \textit{crossed product} extension $A*_\s G$ of $A$ by $G$, and each extension of a $G$-module $A$ by $G$ induces an element of $Z^2(G,A)$ as its \textit{factor set}. This leads to a one-to-one correspondence between $H^2(G,A)$ and the equivalence classes of extensions of $A$ by $G$ (see, for example,~\cite{Maclane,Brown}). There are numerous generalizations of this result in several directions. For instance, R\'edei~\cite{Redei52} introduced and studied the so-called Schreier extensions of monoids (which were later associated with certain monoid co\-ho\-mo\-lo\-gy~\cite{Patchkoria14} in~\cite{Patchkoria18}), Lausch~\cite{Lausch} provided a cohomological framework for the earlier results by Coudron~\cite{Coudron68} and D'Alarcao~\cite{DAlarcao69} on extensions of semilattices of groups by inverse semigroups, Dokuchaev and Khrypchenko~\cite{DK2,DK3} investigated a class of extensions of semilattices of groups by groups which can be described in terms of the second \textit{partial cohomology group}~\cite{DK}.

The third group cohomology $H^3(G,A)$ naturally appears as the set of \textit{obstructions} to extensions of an \textit{abstract kernel}. The classical result states that an abstract kernel admits an extension if and only if the corresponding obstruction is trivial~\cite{Maclane,Brown}. There exist analogs of this fact in the context of extensions of monoids~\cite{MFMPS20}, semilattices of groups by inverse semigroups~\cite{Lausch} and for a class of extensions of semilattices of groups by groups~\cite{DKM}. Another, maybe lesser-known, interpretation of $H^3(G,A)$ (which naturally generalizes to $H^n(G,A)$ for all $n\ge 3$) uses the notion of a \textit{crossed module} over a group. Crossed modules seem to be first introduced in the work by Whitehead~\cite{Whitehead49} on algebraic topology. They also appeared, together with the corresponding $4$-term exact sequences, in the article by Maclane~\cite{MacLane49}, although he did not use the term ``crossed module''. The one-to-one correspondence between the elements of $H^3(G,A)$ and the equivalence classes of $4$-term exact sequences was established by Gerstenhaber~\cite{Gerstenhaber66}. Finally, Huebschmann~\cite{Huebschmann80} and Holt~\cite{Holt79} independently generalized the result to an arbitrary $n\ge 3$. For more details we recommend Maclane's historical note appended to Holt's paper~\cite{Holt79}.

In this work we introduce the notion of a crossed module over an inverse semigroup. Our definition is inspired by the definition of a module over an inverse semigroup in the sense of Lausch~\cite{Lausch} and by the definition of a crossed module over a group in the sense of Whitehead~\cite{Whitehead49} and Maclane~\cite{MacLane49}. In \cref{sec-crossed-mod-ext} we associate, with any crossed $S$-module $A$, a $4$-term exact sequence of inverse semigroups $A \xrightarrow{i} N \xrightarrow{\beta} S \xrightarrow{\pi} T$, which we call a \textit{crossed module extension of $A$ by $T$}. We then show that any such extension induces a $T$-module structure on $A$, and equivalent extensions induce the same $T$-module structure on $A$. In \cref{sec-from-E(T_A)-to-H^3(T^1_A^1)} we construct a map from the set $\mathcal{E}(T,A)$ of equivalence classes of crossed module extensions of a $T$-module $A$ by $T$ to the group $H^3(T^1,A^1)$. We then introduce the so-called {\it admissible} crossed module extensions and show that the set $\mathcal{E}_\le(T,A)$ of equivalence classes of admissible crossed module extensions of $A$ by $T$ is mapped to $H^3_\le(T^1,A^1)$. \cref{sec:H3toE} is the main technical part of the paper. With any cocycle from $Z^3_\le(T^1,A^1)$ we associate a crossed module extension $A \xrightarrow{i} N \xrightarrow{\beta} S \xrightarrow{\pi} T$ of $A$ by $T$, which gives rise to a map from $H^3_\le(T^1,A^1)$ to $\mathcal{E}(T,A)$. Here $S$ is the \textit{$E$-unitary cover}~\cite{McAlisterReilly77} of $T$ through the free group $FG(T)$ and $N$ is a semilattice of groups which can be seen as a direct product of $A$ and $K=\pi\m(E(T))$ in the appropriate category. Our approach is based on that of Maclane~\cite{MacLane49}, but the technical details are much more complicated, since, unlike the group case, the semigroups $S$ and $K$ are not free, so we cannot define homomorphisms on them just sending the generators as we wish. Finally, in \cref{sec-H^3_le<->E_le} we show that $H^3_\le(T^1,A^1)$ is mapped bijectively onto $\mathcal{E}_\le(T,A)$, whenever $T$ is an \textit{$F$-inverse} monoid. This is proved in \cref{H^3_le(T_A)<->E(T_A)}, which is the main result of our work.

	\section{Preliminaries} \label{sec:Prelim}

An \textit{inverse semigroup} is a semigroup $S$ in which every element $s\in S$ admits a unique \textit{inverse}, denoted $s\m$, such that $ss\m s=s$ and $s\m ss\m=s\m$. The set of idempotents $E(S)$ of an inverse semigroup $S$ is a semilattice under the multiplication in $S$. If an inverse semigroup $S$ can be represented as a disjoint union of groups, then the identity elements of these groups are exactly the idempotents of $S$, and the inverses of the group elements coincide with the inverses in $S$. In this case we write $S=\bigsqcup_{e\in E(S)}S_e$, where $S_e=\{s\in S\mid ss\m=s\m s=e\}$ is a \textit{group component} of $S$. Inverse semigroups of this type are called \textit{semilattices of groups}, or sometimes \textit{Clifford semigroups}~\cite[II.2]{Petrich}. These are exactly the inverse semigroups whose idempotents are central~\cite[Theorem II.2.6]{Petrich}. In particular, any commutative inverse semigroup is a \textit{semilattice of abelian groups}.

A homomorphism between two inverse semigroups respects inverses, and a homomorphic image of an inverse semigroup is always an inverse semigroup. A special class of inverse semigroup homomorphisms are \textit{idem\-potent-separating homomorphisms}, i.e. the ones that are injective on idempotents. If $\pi: S \to S'$ is an idempotent-separating epimorphism, then the inverse image of any $e' \in E(S')$ is a subgroup $K_e$ of $S$, where $e$ is the unique idempotent of $S$ such that $\pi(e) = e'$. The collection $\{K_e\}_{e\in E(S)}$ of subgroups of $S$ determines the inverse subsemigroup $K = \bigsqcup_{e\in E(S)} K_e$ of $S$ and hence a semilattice of groups, which is called a \textit{(group) kernel normal system}. Conversely, given an inverse semigroup $S$, a collection $\mathcal K=\{K_e\}_{e\in E(S)}$ of disjoint subgroups of $S$ with identity elements $e\in E(S)$ is a kernel normal system of $S$ exactly when the union $K=\bigsqcup_{e\in E(S)}K_e$ is a subsemigroup of $S$ satisfying $sKs^{-1}\subseteq K$ for all $s\in S$ (see \cite[Theorem 7.54]{Clifford-Preston-2}). There is a characterization of the kernel normal systems of an arbitrary inverse semigroup in \cite[p. 60]{Clifford-Preston-2}.


	Let us now recall some definitions from \cite{DK2,Lausch}.
	\begin{defn}
		Let $A$ be a semilattice of groups. An endomorphism $\f : A \to A$ is called \textit{relatively invertible} if there exist $\bar{\f}\in \End A$ and $e_\f \in E(A)$ satisfying:
		\begin{enumerate}[label=(\roman*)]
			\item $\bar{\f} \circ \f (a) = e_\f a$ and $\f \circ \bar{\f}(a) = \f(e_\f) a$, for any $a \in A$;
			\item $e_\f$ is the identity of $\bar{\f}(A)$ and $\f(e_\f)$ is the identity of $\f(A)$.
		\end{enumerate}
		The set of relatively invertible endomorphisms of $A$ is denoted by $\mend A$.
	\end{defn}
	For a semilattice of groups $A$, we denote by $\I_{ui} (A)$ the inverse semigroup of isomorphisms between principal ideals of $A$. Then we have the following.
	\begin{prop}{\cite[Proposition 3.4]{DK2}}\label{prop:3.4}
		The set $\mend A$ forms an inverse subsemigroup of $\End A$ isomorphic to $\I_{ui} (A)$.
	\end{prop}
	
	\begin{defn}{\cite[Definition 3.8]{DK2}}\label{def:TwistMod}
		Let $S$ be an inverse semigroup. A \textit{twisted $S$-module} is a semilattice of groups $A$ together with a triple $\Lambda = (\alpha, \lambda, f)$, where $\alpha : E(S) \to E(A)$ is an isomorphism, $\lambda$ is a map $S \to \mend A$ and $f$ is a map $S^2 \to A$ (called a \textit{twisting}) such that $f(s,t) \in A_{\alpha(stt^{-1} s^{-1})}$ satisfying:
		\begin{enumerate}[label=(TM\arabic*),leftmargin=2cm]
			\item\label{TM1} $\lambda_e(a) = \alpha(e)a$, for all $e \in E(S)$ and $a \in A$;
			\item\label{TM2} $\lambda_s(\alpha(e)) = \alpha(ses^{-1})$, for all $s \in S$, $e \in E(S)$;
			\item\label{TM3} $\lambda_s \circ \lambda_t(a) = f(s,t)\lambda_{st}(a)f(s,t)^{-1}$, for all $s,t \in S$ and $a \in A$;
			\item\label{TM4} $f(se,e) = \alpha(ses^{-1})$, $f(e,es) = \alpha(ess^{-1})$, for all $s \in S$, $e \in E(S)$;
			\item\label{TM5} $\lambda_s(f(t,u)) f(s,tu) = f(s,t) f(st,u)$, for all $s,t,u \in S$.
		\end{enumerate}
	\end{defn}
	
	\begin{rem}
		In what follows we shall consider twisted $S$-modules with trivial twistings, i.e. those which satisfy $f(s,t) = \alpha(stt^{-1}s^{-1})$, for all $s,t \in S$. Then, as a consequence of \labelcref{TM3}, the map $\lambda$ is in fact a homomorphism. If, moreover, $A$ is commutative then $A$ is called an \textit{$S$-module} and the pair $(\alpha, \lambda)$ an \textit{$S$-module structure on $A$}\footnote{This notion of an $S$-module is in fact equivalent to that in the sense of Lausch~\cite[p. 274]{Lausch}.}.
	\end{rem}

	Given an $S$-module $A$ and $n\ge 1$, denote by $C^n(S^1,A^1)$ the abelian group of functions
	$$
	\left\{f:S^n\to A\mid f(s_1,\dots,s_n)\in A_{\alpha(s_1\dots s_ns\m_n\dots s\m_1)}\right\}
	$$
	under the coordinate-wise multiplication. The groups $C^n(S^1,A^1)$ form a cochain complex under the coboundary homomorphism $\delta^n:C^n(S^1,A^1)\to C^{n+1}(S^1,A^1)$ mapping $f\in C^n(S^1,A^1)$ to $\delta^nf\in C^{n+1}(S^1,A^1)$, where
	\begin{align*}
	(\delta^nf)(s_1,\dots,s_{n+1})&=\lambda_{s_1}(f(s_2,\dots,s_{n+1}))\notag\\
	&\quad\prod_{i=1}^nf(s_1,\dots,s_is_{i+1},\dots,s_{n+1})^{(-1)^i}\notag\\
	&\quad f(s_1,\dots,s_n)^{(-1)^{n+1}}.
	\end{align*}
	Denote $\ker\delta^n$ by $Z^n(S^1,A^1)$, $\im\delta^{n-1}$ by $B^n(S^1,A^1)$ and $Z^n(S^1,A^1)/B^n(S^1,A^1)$ by $H^n(S^1,A^1)$. The elements of $C^n(S^1,A^1)$, $Z^n(S^1,A^1)$ and $B^n(S^1,A^1)$ will be called {\it $n$-cochains}, {\it $n$-cocycles} and {\it $n$-coboundaries} of $S$ with values in $A$, respectively. The group $H^n(S^1,A^1)$ is isomorphic to the Lausch cohomology group $H^n(S,A)$ for all $n\ge 2$ as shown in~\cite[Proposition 2.13]{DK3}.
	
	An $n$-cochain $f\in C^n(S^1,A^1)$, $n\ge 1$, is said to be {\it order-preserving}, if
	\[s_1\le t_1,\dots,s_n\le t_n\impl f(s_1,\dots,s_n)\le f(t_1,\dots,t_n).\]
	Such $n$-cochains form a subgroup of $C^n(S^1,A^1)$, denoted by $C^n_\le(S^1,A^1)$. Since moreover $\dl^n\left(C^n_\le(S^1,A^1)\right)\sst C^{n+1}_\le(S^1,A^1)$, we obtain the cochain complex
	\[C^1_\le(S^1,A^1)\overset{\delta^1}{\to}\dots\overset{\delta^{n-1}}{\to}C^n_\le(S^1,A^1)\overset{\delta^n}{\to}\dots\]
	We add one more term $C^0_\le(S^1,A^1)$ on the left of this complex, whose definition we do not need in this paper, and naturally define the groups of \textit{order-preserving $n$-cocycles} $Z^n_\le(S^1,A^1)$, \textit{$n$-coboundaries} $B^n_\le(S^1,A^1)$ and \textit{$n$-cohomologies} $H^n_\le(S^1,A^1)$ of $S$ with values in $A$. This inverse semigroup cohomology is related to the partial group cohomology~\cite{DK} as proved in~\cite[Theorem 3.16]{DK3} and will also play a crucial role in our paper.
	
	\section{Crossed module extensions of \texorpdfstring{$A$}{A} by \texorpdfstring{$T$}{T}}\label{sec-crossed-mod-ext}
	In this section, we define the notion of a crossed module over an inverse semigroup and construct the related $4$-term sequence of inverse semigroups.

	\begin{defn}\label{def:modcruz}
		Let $S$ be an inverse semigroup. A \textit{crossed $S$-module} is a semilattice of groups $N$ with a triple $(\alpha, \lambda, \beta)$, where $\alpha$ is an isomorphism $E(S) \to E(N)$, $\lambda$ is a homomorphism $S \to \mend N$ and $\beta$ is an idempotent-separating homomorphism $N \to S$ such that $\beta|_{E(N)} = \alpha^{-1}$ and
		\begin{enumerate}[label=(CM\arabic*),leftmargin=2cm]
			\item\label{CM1} $\lambda_e(n) = \alpha(e)n$, for all $e \in E(S)$, $n \in N$;
			\item\label{CM2} $\lambda_s(\alpha(e)) = \alpha(ses^{-1})$, for all $s \in S$, $e \in E(S)$;
			\item\label{CM3} $\lambda_{\beta(n)}(n') = nn'n^{-1}$, for all $n,n' \in N$;
			\item\label{CM4} $\beta(\lambda_s(n)) = s \beta(n) s^{-1}$, for all $s \in S$, $n \in N$.
		\end{enumerate}
	\end{defn}
	
	\begin{rem}
		 Conditions \labelcref{CM1,CM2} mean that the pair $(\alpha,\lambda)$ defines a twisted $S$-module structure on $N$ with trivial twisting. Condition \labelcref{CM4} means that the homomorphism $\beta$ is equivariant with respect to the conjugation on $S$, and \labelcref{CM3} is sometimes called the \textit{Peiffer identity} in the literature on group cohomology (see \cite{Brown}).
	\end{rem}
	
	\begin{exm}
		Let $N$ be a semilattice of groups and $S = \mend N$. For $\e\in E(S)$ we define $\af(\e)=e_\e\in E(N)$ \cite[Corollary 3.5]{DK2}. We also set $\lambda=\id:S\to\mend N$ and $\beta:N\to S$ to be the map sending $n\in N$ to the conjugation $\bt_n$ by $n$. Then \labelcref{CM1,CM3} hold by definition. As for \labelcref{CM2}, we have:
			\begin{align*}
			 	(\f\e\bar\f)(n)&=\f(e_\e\bar\f(n))=\f(e_\e)\f(\bar\f(n))=\f(e_\e)\f(e_\f) n=\f(e_\e) n,
			\end{align*}
			so that $\af(\f\e\bar\f)=e_{\f\e\bar\f}=\f(e_\e)=\lb_\f(e_\e)=\lb_\f(\af(\e))$. Moreover,
			\begin{align*}
			(\varphi \beta_n \bar{\varphi})(n') & =  \varphi(n \bar{\varphi}(n')n^{-1}) =  \varphi(n) \varphi(\bar{\varphi}(n')) \varphi(n^{-1}) \\
			& =  \varphi(n) \varphi(e_\varphi)n' \varphi(n)^{-1}  =  \varphi(n) n' \varphi(n)^{-1}\\
			& = \beta_{\varphi(n)}(n') = \beta_{\lambda_\varphi(n)}(n'),
			\end{align*}
			proving \labelcref{CM4}.
	\end{exm}
	
	\begin{exm}
			Any $S$-module structure $(\af,\lb)$ on a semilattice of abelian groups $N$ (in the sense of \cite{Lausch}) is a crossed $S$-module with $\beta : N \to S$ defined by $\beta(n) = \alpha^{-1}(nn^{-1})$.
			
			Indeed, \labelcref{CM3} is
			\begin{align*}
				\lambda_{\beta(n)}(n') = \lambda_{\alpha^{-1}(nn^{-1})}(n') = \alpha(\alpha^{-1}(nn^{-1}))n' = nn^{-1}n' = nn'n^{-1},
			\end{align*}
			the latter equality being due to the commutativity of $N$. To prove \labelcref{CM4}, observe that $\bt(n)\in E(S)$ and $nn\m=\af(\bt(n))$. Then
			\begin{align*}
				\beta(\lambda_s(n))  =   \alpha^{-1}(\lambda_s(nn^{-1})) 
				 =  \alpha^{-1}(\lambda_s(\alpha(\bt(n)))) =  \alpha^{-1}(\alpha(s\bt(n)s^{-1})) =  s\bt(n)s^{-1}.
			\end{align*} 
	\end{exm}

\begin{exm}
	Any group kernel normal system  $N$ of an inverse semigroup $S$ is a crossed $S$-module with $S$ acting on $N$ by conjugation and $\beta$ being the inclusion.
\end{exm}

	\begin{lem}\label{i(A)-sst-C(N)}
		Given a crossed $S$-module $N$, the set\footnote{This is called the Kernel of $\beta$ in~\cite{Lawson}.}
		\begin{align}\label{A=bt^(-1)(E(S))}
			A = \beta^{-1}(E(S))
		\end{align}
		is a semilattice of groups contained in $C(N)$.
	\end{lem}
	\begin{proof}
		Since $\beta$ is idempotent-separating, we have the group kernel normal system $A=\bigsqcup_{e \in E(N)}A_e$, where $A_e  = \{ a \in N \mid \beta(a) = \beta(e) \}$.
		
		For each $e \in E(N)$ and $a\in A_e$, using \labelcref{CM3}, we have
		\begin{align*}
			ne=ene^{-1}=\lambda_{\beta(e)}(n)=\lambda_{\beta(a)}(n)=ana^{-1},
		\end{align*}
		whence $na=nea = ana^{-1}a = an$, so that $A \subseteq C(N)$.
	\end{proof}
	
	Since $s \beta(N) s^{-1} \subseteq \beta(N)$ by \labelcref{CM4} and $\beta(N)$ is a subsemigroup of $S$, the collection $\mathcal{B} = \{\beta(N_e)\}_{e \in E(N)}$ is a group kernel normal system in $S$ by \cite[Theorem 7.54]{Clifford-Preston-2}. Let $\rho_\B$ be the induced congruence on $S$ and
	\begin{align}\label{T=S-over-rho_B}
		T = S/{\rho_\mathcal{B}}.
	\end{align}
	Then $\pi = \rho_\mathcal{B}^\natural : S \to T$ is an idempotent-separating epimorphism, such that $\pi^{-1}(E(T)) = \beta(N)$.
	
	This motivates the following.
	\begin{defn}\label{defn-crossed-mod-ext}
		Let $A$ be a semilattice of abelian groups and $T$ an inverse semigroup. A \textit{crossed module extension of $A$ by $T$} is a $4$-term sequence 
		\begin{equation}\label{eq:seq4terms}
		A \xrightarrow{i} N \xrightarrow{\beta} S \xrightarrow{\pi} T,
		\end{equation}
		 where 
		 \begin{enumerate}[label=(CME\arabic*),leftmargin=2cm]
		 	\item\label{CME1} $N$ is a crossed $S$-module and $\bt$ is the corresponding crossed module homomorphism;
		 	\item\label{CME2} $i$ is a monomorphism and $\pi$ is an idempotent-separating epimorphism;
		 	\item\label{CME3} $i(A) = \beta^{-1}(E(S))$ and $\beta(N) = \pi^{-1}(E(T))$.
		 \end{enumerate}
	\end{defn}
	Condition \labelcref{CME3} says that the sequence \labelcref{eq:seq4terms} is \textit{exact}.

	\begin{defn}\label{equiv-ext-defn}
		By the \textit{equivalence of crossed module extensions} of $A$ by $T$ we mean the smallest equivalence relation identifying $A \xrightarrow{i} N \xrightarrow{\beta} S \xrightarrow{\pi} T$ and $A \xrightarrow{i'} N' \xrightarrow{\beta'} S' \xrightarrow{\pi'} T$, such that there are homomorphisms (not necessarily isomorphisms!) $\varphi_1:N\to N'$ and $\varphi_2:S\to S'$ 
		\begin{enumerate}[label=(CMEE\arabic*),leftmargin=2cm]
			\item\label{CMEE1} making the diagram
			\begin{equation}\label{eq:equivseqs4}
				\begin{tikzcd}
					A \ar[r, "i"] \ar[d, equal] & N \ar[r, "\beta"] \ar[d, "\varphi_1"] & S \ar[r, "\pi"] \ar[d, "\varphi_2"] & T \ar[d, equal] \\
					A \ar[r, "i'"] & N' \ar[r, "\beta'"] & S' \ar[r, "\pi'"] & T
				\end{tikzcd}
			\end{equation} 
			commute;
			\item\label{CMEE2} satisfying for all $s \in S$ 
			\begin{align}\label{eq:actcomp}
				\varphi_1\circ\lambda_s = \lambda'_{\varphi_2(s)}\circ\varphi_1,
			\end{align}
			where $\lb:S\to\mend N$ and $\lb':S'\to\mend N'$ are the corresponding crossed module homomorphisms.

		\end{enumerate}
	\end{defn}

\begin{rem}
	Observe that \labelcref{CMEE1} implies
	\begin{align}\label{f_1-circ-af=af'-circ-f_2}
	\f_1\circ\af=\af'\circ\f_2\mbox{ on }E(S).
	\end{align} 
\end{rem}
\begin{proof}
	Indeed, $\bt'\circ\f_1|_{E(N)}=\f_2\circ\bt|_{E(N)}\iff\bt'|_{E(N')}\circ\f_1|_{E(N)}=\f_2\circ\bt|_{E(N)}\iff\af'\m\circ\f_1|_{E(N)}=\f_2\circ\af\m$, whence \labelcref{f_1-circ-af=af'-circ-f_2}.
\end{proof}

	\begin{lem}
		Let $A$ be a semilattice of abelian groups and $T$ an inverse semigroup. Then any crossed module extension \labelcref{eq:seq4terms} of $A$ by $T$ induces a $T$-module structure on $A$.
	\end{lem}
	\begin{proof}
		Let $(\alpha, \lambda, \beta)$ be the involved crossed $S$-module structure on $N$. Given $t\in T$, define $\eta_t: A \to A$ by 
		\begin{align}\label{eta_t(a)=i^(-1)(lb_s(i(a)))}
			\eta_t = i\m\circ\lambda_{s}\circ i,
		\end{align}
		where $s\in S$ is such that $\pi(s) = t$. The right-hand side of \labelcref{eta_t(a)=i^(-1)(lb_s(i(a)))} makes sense in view of \cref{CM4,A=bt^(-1)(E(S))}. Moreover, the definition does not depend on the choice of $s \in S$ with $\pi(s) = t$. Indeed, if $\pi(s) = \pi(s')$, then $(s,s') \in \ker \pi$, so that $s' = s\beta(n)$ for some $n \in N_{\beta^{-1}(s^{-1}s)} = N_{\alpha(s^{-1}s)}$. It follows that
		\begin{align*}
		\lambda_{s'}(i(a)) & = \lambda_{s\beta(n)}(i(a)) = \lambda_{s}(\lambda_{\beta(n)}(i(a)))\\
		& = \lambda_{s}(ni(a)n^{-1}) & \text{(by \cref{CM3})}\\
		& = \lambda_{s}(nn^{-1}i(a)) &  \text{(by \cref{i(A)-sst-C(N)})}\\
		&= \lambda_{s}(\alpha(s^{-1}s))\lambda_{s}(i(a)) & \text{(as $n\in N_{\alpha(s^{-1}s)}$)}\\
		& = \alpha(ss^{-1})\lambda_{s}(i(a)) & \text{(by \cref{CM2})}\\
		&= \lambda_{s}(i(a)). & \text{(by \cite[Remark 3.9]{DK2})}
		\end{align*}
		
		We now define 
		\begin{align}\label{0=i^(-1).af.pi^(-1)}
			\0=i\m\circ\af\circ(\pi|_{E(S)})^{-1}.
		\end{align}
	
		It is an isomorphism $E(T) \to E(A)$, since $\af$ and $\pi$ are idempotent-separating. We are going to prove that the pair $(\0,\eta)$ determines a $T$-module structure on $A$. It is clear that $\eta$ is a homomorphism $T\to\End A$. Let $a \in A$, $e \in E(T)$ and $e' \in E(S)$ such that $\pi(e') = e$. Then
		\begin{align*}
			\eta_e(a) & =  i\m(\lambda_{e'}(i(a))) =  i\m(\alpha(e')i(a)) & \text{(by \labelcref{CM1})} \\
			& =  i\m(\alpha(\pi|_{E(S)}^{-1}(e)))a =  \theta(e)a.  & \text{(by \labelcref{0=i^(-1).af.pi^(-1)})}
		\end{align*}
		Now, given $e \in E(A)$, $t \in T$ and $s \in S$, such that $\pi(s) = t$, we have: 
		\begin{align*}
			\eta_t(\theta(e)) & =  i\m(\lambda_s(i(\theta(e))))  =  i\m(\lambda_s(\alpha(\pi|_{E(S)}^{-1}(e)))) & \text{(by \labelcref{0=i^(-1).af.pi^(-1)})}\\
			&  =  i\m(\alpha(s \pi|_{E(S)}^{-1}(e) s^{-1})) & \text{(by \labelcref{CM2})} \\
			& =  i\m(\alpha(\pi|_{E(S)}^{-1}(tet^{-1}))) \\
			& =  \theta(tet^{-1}).  & \text{(by \labelcref{0=i^(-1).af.pi^(-1)})}
		\end{align*}
		
	\end{proof}
	\begin{lem}
		Let $A$ be a semilattice of abelian groups and $T$ an inverse semigroup. Then equivalent crossed module extensions of $A$ by $T$ induce the same $T$-module structure on $A$.
	\end{lem}
	\begin{proof}
		It suffices to prove the statement for $A \xrightarrow{i} N \xrightarrow{\beta} S \xrightarrow{\pi} T$ and $A \xrightarrow{i'} N' \xrightarrow{\beta'} S' \xrightarrow{\pi'} T$, two crossed module extensions of $A$ by $T$ admitting a pair of maps $(\f_1,\f_2)$ satisfying \labelcref{CMEE1,CMEE2}. Let $(\alpha, \lambda, \beta)$, $(\alpha', \lambda', \beta')$ be the involved crossed module structures of $S$ and $S'$ on $N$ and $N'$, respectively. Denote by $(\0,\eta)$ and $(\0',\eta')$ the induced $T$-module structures on $A$. Given $t\in T$, there exists $s\in S$, such that $t=\pi(s)=\pi'(\f_2(s))$. Then
		\begin{align*}
			\eta'_t&=i'\m\circ\lb'_{\f_2(s)}\circ i' & \text{(by \labelcref{eta_t(a)=i^(-1)(lb_s(i(a)))})}\\
			&=i'\m\circ\lb'_{\f_2(s)}\circ\f_1\circ i & \text{(since $i'=\f_1\circ i$)}\\
			&=i'\m\circ\f_1\circ\lb_s\circ i & \text{(by \labelcref{eq:actcomp})}\\
			&=i'\m\circ\f_1\circ i\circ i\m\circ\lb_s\circ i & \text{(since $\lb_s(i(A))\sst i(A)$)}\\
			&=i'\m\circ i'\circ i\m\circ\lb_s\circ i & \text{(since $i'=\f_1\circ i$)}\\
			&=i\m\circ\lb_s\circ i=\eta_t. & \text{(by \labelcref{eta_t(a)=i^(-1)(lb_s(i(a)))})}
		\end{align*}
			Furthermore,
		\begin{align*}
		\0'\circ\pi|_{E(S)}&=\0'\circ\pi'\circ\f_2|_{E(S)} & \text{(since $\pi=\pi'\circ\f_2$)}\\
		&=i'\m\circ\af'\circ\f_2|_{E(S)} & \text{(by \labelcref{0=i^(-1).af.pi^(-1)})}\\
		&=i'\m\circ\f_1\circ\af & \text{(by \labelcref{f_1-circ-af=af'-circ-f_2})}\\
		&=i'\m\circ\f_1\circ i\circ i\m\circ\af & \text{(since $E(N)\sst i(A)$)}\\
		&=i\m\circ\af, & \text{(since $i'=\f_1\circ i$)}
		\end{align*}
		whence $\0'=i\m\circ\af\circ(\pi|_{E(S)})\m=i\m\circ\af\circ\pi\m|_{E(T)}=\0$ in view of \labelcref{0=i^(-1).af.pi^(-1)}.
	\end{proof}
	
	\begin{defn}
		Let $A$ be a $T$-module. By a \textit{crossed module extension of $A$ by $T$} we shall mean a sequence from \cref{defn-crossed-mod-ext} which induces the given $T$-module structure on $A$. We denote by $\mathcal{E}(T,A)$ the set of equivalence classes of crossed module extensions of a $T$-module $A$ by $T$.
	\end{defn}

	\section{From \texorpdfstring{$\mathcal{E}(T,A)$}{E(T,A)} to \texorpdfstring{$H^3(T^1,A^1)$}{H³(T¹,A¹)}}\label{sec-from-E(T_A)-to-H^3(T^1_A^1)}
	

	\begin{defn}
		Given a homomorphism of inverse semigroups $\f:S\to T$, a map $\rho : \f(T) \to S$ such that $\f \circ \rho = \id_{\f(T)}$ will be called a \textit{transversal} of $\f$. We say that $\rho$ \textit{respects idempotents} if $\rho(E(\f(T)))\sst E(S)$.
	\end{defn}
	\begin{rem}
		Observe that $E(\f(T))=\f(E(S))$, so one may always choose a transversal $\rho$ of $\f$ which respects idempotents. If, moreover, $\f$ is an idempotent-separating epimorphism and $\rho$ respects idempotents, then
		\begin{align}
		\rho(t)\rho(t)\m&=\rho(t)\rho(t\m)=\rho(tt\m), \mbox{ for any $t\in T$,}\label{rho(t)rho(t)-inv=rho(tt-inv)}\\
		\rho(t)\rho(e)\rho(t)\m&=\rho(tet\m), \mbox{ for any $t\in T$ and $e\in E(T)$,}\label{rho(t)rho(e)rho(t)-inv=rho(tet-inv)}\\
		\rho|_{E(T)}&=\f\m|_{E(T)}.\label{rho|_E(T)=pi-inv|_E(T)}
		\end{align}
	\end{rem}

	\begin{lem}\label{from-crossed-mod-ext-to-C^3}
		Let $T$ be an inverse semigroup and $A$ a $T$-module. Any crossed module extension of $A$ by $T$ determines an element $c\in C^3(T^1,A^1)$.
	\end{lem}
	\begin{proof}
		For a crossed module extension \labelcref{eq:seq4terms} of $A$ by $T$, we choose a transversal $\rho$ of $\pi$. Since $\pi(\rho(x)\rho(y)) = \pi(\rho(xy))$, there exists a function $f : T^2 \to \beta(N)$ such that $f(x,y) \in \beta(N)_{\rho(xy)\rho(xy)^{-1}}$ and
		 \begin{equation}\label{rho(x)rho(y)=f(x_y)rho(xy)}
		\rho(x)\rho(y) = f(x,y)\rho(xy).
		\end{equation} 
		For arbitrary $x,y,z\in T$ we shall calculate $\rho(x)\rho(y)\rho(z)$ in two different ways and use the associativity of $S$:
		\begin{align*}
			\rho(x)(\rho(y)\rho(z)) & =  \rho(x)f(y,z)\rho(yz) & \text{(by \labelcref{rho(x)rho(y)=f(x_y)rho(xy)})}\\
			& =  \rho(x)f(y,z)\rho(x)^{-1}\rho(x)\rho(yz) & \text{(since $E(S)\sst C(\bt(N))$)} \\
			& =  \rho(x)f(y,z)\rho(x)^{-1}f(x,yz)\rho(xyz) & \text{(by \labelcref{rho(x)rho(y)=f(x_y)rho(xy)})}
		\end{align*} 
			and
		\begin{align*}
			(\rho(x)\rho(y))\rho(z) & =  f(x,y)\rho(xy)\rho(z)& \text{(by \labelcref{rho(x)rho(y)=f(x_y)rho(xy)})}\\
			& =  f(x,y)f(xy,z)\rho(xyz), & \text{(by \labelcref{rho(x)rho(y)=f(x_y)rho(xy)})}
		\end{align*}
		whence
		\begin{equation}\label{eq:f2cocfake}
			\rho(x)f(y,z)\rho(x)\m f(x,yz) = f(x,y)f(xy,z).
		\end{equation}
		
		Since $f(x,y) \in \beta(N)_{\rho(xy)\rho(xy)\m}$, we can lift $f$ to a function $F : T^2 \to N$ by 
		\begin{equation}\label{bt(F(x_y))=f(x_y)}
		\beta(F(x,y)) = f(x,y),
		\end{equation} 
		where $F(x,y) \in N_{\af(\rho(xy)\rho(xy)\m)}$. Then the right-hand side of \labelcref{eq:f2cocfake} is $\beta(F(x,y)F(xy,z))$, and its left-hand side becomes
		\begin{align*}
			\rho(x)\beta(F(y,z))\rho(x)\m \beta(F(x,yz)) &= \beta(\lambda_{\rho(x)}(F(y,z))) \beta(F(x,yz)) & \text{(by \labelcref{CM4})}\\
			& =  \beta(\lambda_{\rho(x)}(F(y,z))F(x,yz)).
		\end{align*}
		It follows that $(\lambda_{\rho(x)}(F(y,z))F(x,yz),F(x,y)F(xy,z)) \in \ker \beta$, and hence there exists a function $c : T^3 \to A$ such that 
		\begin{equation}\label{lb_rho(x)(F(y_z))F(x_yz)=i(c(x_y_z))F(x_y)F(xy_z)}
		\lambda_{\rho(x)}(F(y,z)) F(x,yz) = i(c(x,y,z)) F(x,y) F(xy,z),
		\end{equation} 
		where $c(x,y,z) \in A_{i\m\circ\af(\rho(xyz)\rho(xyz)^{-1})}$. Since $\pi$ is idempotent-separating, we have $\rho(xyz)\rho(xyz)\m=\pi\m(xyzz\m y\m x\m)$. It follows from \labelcref{0=i^(-1).af.pi^(-1)} that $c(x,y,z) \in A_{\0(xyzz\m y\m x\m)}$, i.e. $c\in C^3(T^1,A^1)$.
	\end{proof}
	
	\begin{lem}\label{c-is-a-3-cocycle}
		The cochain $c\in C^3(T^1,A^1)$ from \cref{from-crossed-mod-ext-to-C^3} is a cocycle.
	\end{lem}
	\begin{proof}
		We shall express 
		\[
		\lambda_{\rho(x)}[\lambda_{\rho(y)}(F(z,w)) F(y,zw)] F(x,yzw)
		\] 
		in two ways. On the one hand
		\begin{align*}
			&\lambda_{\rho(x)}[\lambda_{\rho(y)}(F(z,w)) F(y,zw)] F(x,yzw) \\
			 &\quad= \lambda_{\rho(x)}[i(c(y,z,w))  F(y,z) F(yz,w)] F(x,yzw) & \text{(by \labelcref{lb_rho(x)(F(y_z))F(x_yz)=i(c(x_y_z))F(x_y)F(xy_z)})}\\
			 &\quad= \lambda_{\rho(x)}(i(c(y,z,w))) \underbrace{\lambda_{\rho(x)}(F(y,z))} \underbrace{\lambda_{\rho(x)}(F(yz,w)) F(x,yzw)} \\
			 &\quad= \lambda_{\rho(x)} (i(c(y,z,w))) i(c(x,y,z)) F(x,y) F(xy,z) F(x,yz)^{-1} & \text{(by \labelcref{lb_rho(x)(F(y_z))F(x_yz)=i(c(x_y_z))F(x_y)F(xy_z)})}\\
			 &\quad\quad\cdot i(c(x,yz,w)) F(x,yz) F(xyz,w) \\
			 &\quad= \lambda_{\rho(x)} (i(c(y,z,w))) i(c(x,y,z))i(c(x,yz,w)) & \text{(as $i(A)\sst C(N)$)}\\
			 &\quad\quad\cdot F(x,y) F(xy,z) F(x,yz)^{-1} F(x,yz) F(xyz,w) \\
			 &\quad= \lambda_{\rho(x)} (i(c(y,z,w))) i(c(x,y,z)) i(c(x,yz,w)) F(x,y) F(xy,z) F(xyz,w).
		\end{align*}
			On the other hand
			\begin{align*}
				& \lambda_{\rho(x)}[\lambda_{\rho(y)}(F(z,w)) F(y,zw)] F(x,yzw) \\
				 &\quad= \lambda_{\rho(x)}(\lambda_{\rho(y)}(F(z,w))) \underbrace{\lambda_{\rho(x)}(F(y,zw)) F(x,yzw)} \\
				&\quad= \lambda_{\rho(x)\rho(y)}(F(z,w)) i(c(x,y,zw)) F(x,y) F(xy,zw) & \text{(by \labelcref{lb_rho(x)(F(y_z))F(x_yz)=i(c(x_y_z))F(x_y)F(xy_z)})} \\
				&\quad= \lambda_{\beta(F(x,y))\rho(xy)}(F(z,w)) i(c(x,y,zw)) F(x,y) F(xy,zw) & \text{(by \labelcref{rho(x)rho(y)=f(x_y)rho(xy),bt(F(x_y))=f(x_y)})}\\
				&\quad= i(c(x,y,zw) \underbrace{\lambda_{\beta(F(x,y))\rho(xy)}(F(z,w))} F(x,y) F(xy,zw) & \text{(as $i(A)\sst C(N)$)}\\
				&\quad= i(c(x,y,zw) F(x,y) \lambda_{\rho(xy)}(F(z,w)) F(x,y)^{-1} F(x,y) F(xy,zw) & \text{(by \labelcref{CM3})}\\
				&\quad= i(c(x,y,zw) F(x,y) \underbrace{\lambda_{\rho(xy)}(F(z,w)) F(xy,zw)} \\
				&\quad= i(c(x,y,zw) F(x,y) i(c(xy,z,w)) F(xy,z) F(xyz,w) & \text{(by \labelcref{lb_rho(x)(F(y_z))F(x_yz)=i(c(x_y_z))F(x_y)F(xy_z)})} \\
				&\quad= i(c(x,y,zw)) i(c(xy,z,w)) F(x,y) F(xy,z) F(xyz,w). & \text{(as $i(A)\sst C(N)$)}
			\end{align*}	
			Cancelling the terms involving $F$, we have 
			\[
			\lambda_{\rho(x)}( i(c(y,z,w)) ) i(c(x,y,z)) i(c(x,yz,w)) = i(c(x,y,zw)) i(c(xy,z,w)).
			\] 
			In view of \labelcref{eta_t(a)=i^(-1)(lb_s(i(a)))} and the injectivity of $i$, the latter is equivalent to
			\begin{align*}
				\eta_x(c(y,z,w)) c(x,y,z) c(x,yz,w) = c(x,y,zw) c(xy,z,w),
			\end{align*}
			i.e. $c \in Z^3(T^1,A^1)$.
	\end{proof}
	
	\begin{lem}\label{another-choice-of-F}
		Another choice of a lifting $F$ in \cref{from-crossed-mod-ext-to-C^3} leads to a cocycle cohomologous to $c$.
	\end{lem}
	\begin{proof}
			Let $F' : T^2 \to N$ be another function such that $\beta(F'(x,y)) = f(x,y)$ and $F'(x,y) \in N_{\af(\rho(xy)\rho(xy)\m)}$. Then $(F(x,y),F'(x,y))\in\ker\bt$, so there exists $u: T^2 \to A$ such that 
			\begin{align}\label{F'(xy)=i(u(xy))F(xy)}
				F'(x,y) = i(u(x,y))F(x,y)
			\end{align}
			and $u(x,y) \in A_{\0(xyy\m x\m)}$. By \labelcref{lb_rho(x)(F(y_z))F(x_yz)=i(c(x_y_z))F(x_y)F(xy_z),F'(xy)=i(u(xy))F(xy)} there exists $c':T^3\to A$ such that
			 \[
			 \lambda_{\rho(x)}(i(u(y,z))F(y,z)) i(u(x,yz))F(x,yz) = i(c'(x,y,z)u(x,y)) F(x,y) i(u(xy,z))F(xy,z).
			 \]
			Moreover, $c'\in Z^3(T^1,A^1)$ thanks to  \cref{c-is-a-3-cocycle,from-crossed-mod-ext-to-C^3}. Since $\lambda_{\rho(x)}$ is an endomorphism of $N$ and $i(A) \sst C(N)$, we obtain
			\begin{align*}
				&\lambda_{\rho(x)}( i(u(y,z)))i(u(xy,z))\m i(u(x,yz))i(u(x,y))\m\lambda_{\rho(x)}( F(y,z)) F(x,yz)\\
				&\quad= i(c'(x,y,z)) F(x,y) F(xy,z).
			\end{align*} 
			Using \labelcref{eta_t(a)=i^(-1)(lb_s(i(a)))}, we rewrite this as follows
			\begin{align*}
				i(\delta^2u(x,y,z)) \underbrace{\lambda_{\rho(x)}( F(y,z)) F(x,yz)} = i(c'(x,y,z)) F(x,y) F(xy,z).
			\end{align*}
			Now applying \labelcref{lb_rho(x)(F(y_z))F(x_yz)=i(c(x_y_z))F(x_y)F(xy_z)} to the bracketed factors we come to
			\[
			i(\delta^2u(x,y,z)) i(c(x,y,z)) F(x,y) F(xy,z) = i(c'(x,y,z)) F(x,y) F(xy,z).
			\] 
			Finally, cancelling $F(x,y) F(xy,z)$ and using the injectivity of $i$ we obtain
			\[
			\delta^2u(x,y,z) c(x,y,z) = c'(x,y,z),
			\]
			i.e. $c'$ is cohomologous to $c$.
	\end{proof}
	
	\begin{lem}\label{another-choice-of-rho}
		Another choice of a transversal $\rho$ of $\pi$ in \cref{from-crossed-mod-ext-to-C^3} leads to a cocycle cohomologous to $c$.
	\end{lem}
	\begin{proof}
		Let $\rho': T \to S$ be another transversal of $\pi$. Then $(\rho(x),\rho'(x)) \in \ker \pi$, and hence there exists a function $v : T \to \beta(N)$ such that 
		\begin{align}\label{rho'(x)=v(x)rho(x)}
			\rho'(x) = v(x)\rho(x)
		\end{align}
		and $v(x) \in \beta(N)_{\rho(x)\rho(x)\m}$. We now express the corresponding $f'(x,y)$ (see \labelcref{rho(x)rho(y)=f(x_y)rho(xy)}) in terms of $f$ and $v$:
		\begin{align*}
		\rho'(x)\rho'(y) & = v(x)\rho(x)v(y)\rho(y) \\
		& = v(x)\rho(x)(\rho(x)^{-1}\rho(x))v(y)\rho(y) \\
		& = v(x)\rho(x)v(y)(\rho(x)^{-1}\rho(x))\rho(y) & \text{(as $E(S)\sst C(\bt(N))$)}\\
		& = v(x)\rho(x)v(y)\rho(x)^{-1}f(x,y)\rho(xy) & \text{(by \labelcref{rho(x)rho(y)=f(x_y)rho(xy)})}\\
		& = v(x)\rho(x)v(y)\rho(x)^{-1}f(x,y)v(xy)^{-1}\rho'(xy), & \text{(by \labelcref{rho'(x)=v(x)rho(x)})}
		\end{align*}
		whence
		\begin{align*}
			f'(x,y)=v(x)\rho(x)v(y)\rho(x)^{-1}f(x,y)v(xy)^{-1},
		\end{align*}
		so that
		\begin{align}
		f'(x,y) v(xy) & = v(x) \rho(x) v(y) \rho(x)\m f(x,y) \notag\\
		& = \rho'(x) v(y) \rho'(x)\m v(x) f(x,y) \label{eq:r'v}.
		\end{align}
		Let $F' : T^2 \to N$ and $V : T \to N$ be functions satisfying 
		\begin{align}
			\beta(F'(x,y)) &= f'(x,y),\label{bt(F'(xy))-is-f'(xy)}\\
			\beta(V(x)) &= v(x),\label{bt(V(x))-is-v(x)}
		\end{align}
		where $F'(x,y)\in N_{\af(\rho(xy)\rho(xy)\m)}$ and $V(x)\in N_{\af(\rho(x)\rho(x)\m)}$. Then we may rewrite \labelcref{eq:r'v} as follows:
		\begin{align*}
			\beta(F'(x,y)) \beta(V(xy)) & =  \rho'(x) \beta(V(y)) \rho'(x)^{-1} \beta(V(x))\beta(F(x,y)) \\
			& = \beta(\lambda_{\rho'(x)}( V(y))) \beta(V(x)) \beta(F(x,y)). & \text{(by \labelcref{CM4})}
		\end{align*}
		Hence, $(F'(x,y)V(xy), \lambda_{\rho'(x)}( V(y)) V(x) F(x,y)) \in \ker \beta$, so there exists a function $w : T^2 \to A$ such that $w(x,y)\in A_{\0(xyy\m x\m)}$ and
		\begin{equation}\label{eq:FV}
		F'(x,y) V(xy) = i(w(x,y)) \lambda_{\rho'(x)}( V(y)) V(x) F(x,y).
		\end{equation}
		We calculate
		\begin{align*}
			&\lambda_{\rho'(x)}(F'(y,z)) \underbrace{F'(x,yz)   V(xyz)}\\
			&\quad = \lambda_{\rho'(x)}(F'(y,z)) i(w(x,yz)) \lambda_{\rho'(x)}( V(yz)) V(x) F(x,yz) & \text{(by \labelcref{eq:FV})}\\
			&\quad = i(w(x,yz)) \lambda_{\rho'(x)}( \underbrace{F'(y,z)   V(yz)})   V(x)   F(x,yz) & \text{(since $i(A)\sst C(N)$)}\\
			&\quad = i(w(x,yz)) \lambda_{\rho'(x)}( i(w(y,z)) \lambda_{\rho'(y)}( V(z))   V(y)   F(y,z))   V(x)   F(x,yz) & \text{(by \labelcref{eq:FV})}\\
			&\quad = i(w(x,yz)) \lambda_{\rho'(x)}( i(w(y,z)))\lambda_{\rho'(x)\rho'(y)}( V(z))\\   &\quad\quad\cdot\lambda_{\rho'(x)}( V(y)) \lambda_{\rho'(x)}( F(y,z))   V(x)   F(x,yz).
		\end{align*}
		Now,
		\begin{align*}
		\lambda_{\rho'(x)\rho'(y)}( V(z))&=\lambda_{f'(x,y)\rho'(xy)}( V(z)) & \text{(by the def. of $f'(x,y)$)}\\
		&=\lambda_{\bt(F'(x,y))}(\lambda_{\rho'(xy)}( V(z))) & \text{(by \labelcref{bt(F'(xy))-is-f'(xy)})}\\
		&=F'(x,y)\lambda_{\rho'(xy)}( V(z))F'(x,y)\m, & \text{(by \labelcref{CM3})}
		\end{align*}
		and
		\begin{align*}
			&\lambda_{\rho'(x)}( V(y)) \underbrace{\lambda_{\rho'(x)}( F(y,z))}   V(x)   F(x,yz)\\
			&\quad = \lambda_{\rho'(x)}( V(y)) \lb_{\beta(V(x))}(\lambda_{\rho(x)}( F(y,z)))   V(x)   F(x,yz) & \text{(by \labelcref{rho'(x)=v(x)rho(x),bt(V(x))-is-v(x)})} \\
			&\quad = \lambda_{\rho'(x)}( V(y))V(x)\lambda_{\rho(x)}( F(y,z))V(x)\m V(x)F(x,yz) & \text{(by \labelcref{CM3})} \\
			&\quad = \lambda_{\rho'(x)}( V(y))  V(x) \underbrace{ \lambda_{\rho(x)}( F(y,z)) F(x,yz)} & \text{(as $E(N)\sst C(N)$)}\\
			&\quad = \lambda_{\rho'(x)}( V(y))  V(x) i(c(x,y,z)) F(x,y) F(xy,z) & \text{(by \labelcref{lb_rho(x)(F(y_z))F(x_yz)=i(c(x_y_z))F(x_y)F(xy_z)})}\\
			&\quad = i(c(x,y,z)) \underbrace{\lambda_{\rho'(x)}( V(y))  V(x)  F(x,y)} F(xy,z) & \text{(as $i(A)\sst C(N)$)}\\
			&\quad = i(c(x,y,z)) i(w(x,y))^{-1} F'(x,y) V(xy)  F(xy,z), & \text{(by \labelcref{eq:FV})}
		\end{align*}
		so
		\begin{align*}
			&\lambda_{\rho'(x)}(F'(y,z)) F'(x,yz)   V(xyz)\\
			&\quad = i(w(x,yz)) \lambda_{\rho'(x)}( i(w(y,z)))F'(x,y)\lambda_{\rho'(xy)}( V(z))F'(x,y)\m\\   
			&\quad\quad \cdot i(c(x,y,z)) i(w(x,y))^{-1} F'(x,y) V(xy)  F(xy,z)\\
			&\quad = i(w(x,yz)) \lambda_{\rho'(x)}( i(w(y,z)))i(c(x,y,z)) i(w(x,y))^{-1}\\   
			&\quad\quad \cdot  F'(x,y)\lambda_{\rho'(xy)}( V(z))(F'(x,y)\m F'(x,y)) V(xy)  F(xy,z) & \text{(as $i(A)\sst C(N)$)}\\
			&\quad = i(w(x,yz)) \lambda_{\rho'(x)}( i(w(y,z)))i(c(x,y,z)) i(w(x,y))^{-1}\\   
			&\quad\quad \cdot  F'(x,y)\underbrace{\lambda_{\rho'(xy)}( V(z)) V(xy)  F(xy,z)} & \text{(as $E(N)\sst C(N)$)}\\
			&\quad = i(w(x,yz)) \lambda_{\rho'(x)}( i(w(y,z)))i(c(x,y,z)) i(w(x,y))^{-1} \\ 
			&\quad \quad F'(x,y) i(w(xy,z))^{-1} F'(xy,z) V(xyz) & \text{(by \cref{eq:FV})}\\
			&\quad = i(c(x,y,z)) \lambda_{\rho'(x)}( i(w(y,z)))i(w(xy,z))^{-1} i(w(x,yz)) i(w(x,y))^{-1}  \\ 
			&\quad \quad  F'(x,y)  F'(xy,z) V(xyz) & \text{(as $i(A)\sst C(N)$)}\\
			&\quad = i(c(x,y,z)) i(\delta^2w(x,y,z)) F'(x,y)  F'(xy,z) V(xyz).  & \text{(by \cref{eta_t(a)=i^(-1)(lb_s(i(a)))})}
		\end{align*}
		Cancelling the factor $V(xyz)$ and using \labelcref{lb_rho(x)(F(y_z))F(x_yz)=i(c(x_y_z))F(x_y)F(xy_z)} for the pair $(\rho', F')$, we have \[
		i(c'(x,y,z)) = i(c(x,y,z)) i(\delta^2w(x,y,z))
		\]
	for all $x,y,z\in T$. Thus, $c'=c\cdot \delta^2w$ in view of the injectivity of $i$.
\end{proof}

	\begin{prop}\label{prop:Egenc}
		There is a well-defined function from $\mathcal{E}(T,A)$ to $H^3(T^1,A^1)$.
	\end{prop}
	\begin{proof}
		In view of \cref{from-crossed-mod-ext-to-C^3,c-is-a-3-cocycle,another-choice-of-F,another-choice-of-rho}, it remains to show that equivalent extensions \cref{eq:equivseqs4} of $A$ by $T$ induce cohomologous $3$-cocycles from $Z^3(T^1,A^1)$. 
		
		Let $\rho : T \to S$, $f : T^2 \to \beta(N)$, $F: T^2 \to N$ and $c : T^3 \to A$ be the maps corresponding to the top sequence of \cref{eq:equivseqs4} as in \cref{from-crossed-mod-ext-to-C^3}. We are going to choose the analogous maps for the bottom sequence of \cref{eq:equivseqs4} in a way that $c'=c$.  We first set $\rho'=\varphi_2 \circ \rho : T \to S'$. It is obviously a transversal of $\pi'$. Applying \cref{rho(x)rho(y)=f(x_y)rho(xy)} to $\rho$, we have
		 \begin{align*}
		\rho'(x)\rho'(y)  &=  \varphi_2(\rho(x)) \varphi_2(\rho(y))=  \varphi_2(f(x,y)\rho(xy))\\ 
		 &=  \varphi_2(f(x,y)) \varphi_2(\rho(xy)) =  \varphi_2(f(x,y)) \rho'(xy)
		\end{align*}
		and $\f_2(f(x,y)) \in \beta(N)_{\f_2(\rho(xy))\f_2(\rho(xy))^{-1}}=\beta(N)_{\rho'(xy)\rho'(xy)^{-1}}$, so $f': = \varphi_2\circ f$ is the function satisfying \cref{rho(x)rho(y)=f(x_y)rho(xy)} for $\rho'$. Since
		\begin{align*}
		\beta'(\varphi_1(F(x,y))) = \varphi_2(\beta(F(x,y))) = \varphi_2(f(x,y)) = f'(x,y),
		\end{align*}
		we may choose $F':=\f_1\circ F:T^2\to N'$ to be a lifting of $f'$. Then
		\begin{align*}
			i'(c(x,y,z))F'(x,y) F'(xy,z)&=\varphi_1(i(c(x,y,z))F(x,y)F(xy,z)) & \text{(as $i'=\f_1\circ i$)}\\
			&=\f_1(\lambda_{\rho(x)}(F(y,z)) F(x,yz)) & \text{(by \cref{lb_rho(x)(F(y_z))F(x_yz)=i(c(x_y_z))F(x_y)F(xy_z)})}\\
			&=\lambda_{\f_2(\rho(x))}(F(y,z)) F(x,yz) & \text{(by \cref{eq:actcomp})}\\
			&=\lambda_{\rho'(x)}(F(y,z)) F(x,yz),
		\end{align*}
		whence $c'(x,y,z)=c(x,y,z)$.
	\end{proof}
	
	\subsection{Admissible crossed module extensions and \texorpdfstring{$H^3_\le(T^1,A^1)$}{H³(T¹,A¹)}}
	
	\begin{defn}
		Given an inverse semigroup $S$, we introduce $\bd,\br:S\to E(S)$ defined by $\bd(s)=s\m s$ and $\br(s)=ss\m$, where $s\in S$.
	\end{defn}

	\begin{defn}
		Let $A$ be a $T$-module, with structure $(\theta, \eta)$, and $c\in C^n(T^1,A^1)$, where $n\ge 1$. We say that $c$ is \textit{normalized} if for all $x_1,\dots,x_n\in T$, $1\le i\le n$ and $e\in E(T)$:
		\begin{align}
			c(e,ex_1,x_2,\dots,x_{n-1})&=\0(\br(ex_1\dots x_{n-1})),\label{c(e_ex_1...x_(n-1))-triv}\\
			c(x_1,\dots,x_{i-2},x_{i-1}e,e,ex_{i+1},x_{i+2},\dots,x_n)&=\0(\br(x_1\dots x_{i-1}ex_{i+1}\dots x_n)),\label{c(x_1...x_(i-1)e_e_ex_(i+1)...x_n)-triv}\\
			c(x_1,\dots,x_{n-2},x_{n-1}e,e)&=\0(\br(x_1\dots x_{n-1}e)).\label{c(x_1...x_(n-1)e_x_n)-triv}
		\end{align}
		Moreover, $c\in C^n(T^1,A^1)$ is \textit{strongly normalized}, if for all $x_1,\dots,x_n\in T$, $1\le i\le n$ and $e\in E(T)$:
		\begin{align}\label{c(x_1...x_(i-1)_e_x_(i+1)...x_n)-triv}
		c(x_1,\dots,x_{i-1},e,x_{i+1},\dots,x_n)&=\0(\br(x_1\dots x_{i-1}ex_{i+1}\dots x_n)).
		\end{align}
	\end{defn}

\begin{prop}\label{c-is-normalized}
	Let $A$ be a $T$-module. Then any crossed module extension of $A$ by $T$ induces a normalized  $c\in Z^3(T^1,A^1)$.  
\end{prop}
\begin{proof}
	Given a crossed module extension \cref{eq:seq4terms} of $A$ by $T$, choose $\rho:T\to S$, $f:T^2\to\bt(N)$, $F:T^2\to N$ and $c\in Z^3(T^1,A^1)$ as in \cref{from-crossed-mod-ext-to-C^3}. We assume that $\rho$ respects idempotents. Then for all $x\in T$ and $e\in E(T)$:
	\begin{align*}
		\rho(e)\rho(e)\rho(x)=\rho(e)\rho(x)=f(e,x)\rho(ex) 
	\end{align*}
	by \cref{rho|_E(T)=pi-inv|_E(T),rho(x)rho(y)=f(x_y)rho(xy)}. On the other hand, applying \cref{rho(x)rho(y)=f(x_y)rho(xy)} twice and using the fact that $E(S)\sst C(\bt(N))$, we have
	\begin{align*}
	\rho(e)\rho(e)\rho(x)=\rho(e)f(e,x)\rho(ex)=f(e,x)\rho(e)\rho(ex)=f(e,x)f(e,ex)\rho(ex).
	\end{align*}
	Since $f(e,x),f(e,ex)\in\bt(N)_{\rho(ex)\rho(ex)\m}=\bt(N)_{\rho(exx\m)}$, we conclude that
	\begin{align}\label{f(e_ex)=rho(ex)rho(ex)-inv}
		f(e,ex)=\rho(exx\m).
	\end{align}
	Similarly,
	\begin{align}\label{f(xe_e)=rho(xe)rho(xe)-inv}
		f(xe,e)=\rho(xex\m).
	\end{align}
	Then \cref{f(e_ex)=rho(ex)rho(ex)-inv,f(xe_e)=rho(xe)rho(xe)-inv,bt(F(x_y))=f(x_y)} immediately imply
	\begin{align}
		F(e,ex)&=\af\circ\rho(exx\m),\label{F(e_ex)=af-circ-rho(exx-inv)}\\
		F(xe,e)&=\af\circ\rho(xex\m).\label{F(xe_e)=af-circ-rho(xex-inv)}
	\end{align}
	Let us now write \cref{lb_rho(x)(F(y_z))F(x_yz)=i(c(x_y_z))F(x_y)F(xy_z)} for the triple $(e,ex,y)\in T^3$:
	\begin{align}\label{c(e_ex_y)-in-terms-of-F}
		\lb_{\rho(e)}(F(ex,y)) F(e,exy) = i(c(e,ex,y)) F(e,ex) F(ex,y).
	\end{align}
	By \cref{CM1,F(e_ex)=af-circ-rho(exx-inv)} the left-hand side of \cref{c(e_ex_y)-in-terms-of-F} equals $F(ex,y)$, while the right-hand side of \cref{c(e_ex_y)-in-terms-of-F} is $i(c(e,ex,y)) F(ex,y)$. Since $F(ex,y),i(c(e,ex,y))\in N_{\af\circ\rho(exyy\m x\m)}$, we have $c(e,ex,y)=\0(exyy\m x\m)$. In a similar way one obtains $c(xe,e,ey)=\0(xeyy\m x\m)$ and $c(x,ye,e)=\0(xyey\m x\m)$ using \cref{lb_rho(x)(F(y_z))F(x_yz)=i(c(x_y_z))F(x_y)F(xy_z),F(e_ex)=af-circ-rho(exx-inv),F(xe_e)=af-circ-rho(xex-inv),CM2}. 
	
\end{proof}

\begin{defn}
	Given a homomorphism of inverse semigroups $\f:S\to T$ and a transversal $\rho : \f(S) \to S$ of $\f$, we say that $\rho$ is \textit{order-preserving} whenever for all $x,y\in \f(S)$:
	\begin{align*}
		x\le y\impl \rho(x)\le\rho(y).
	\end{align*}
\end{defn}

\begin{rem}\label{rho-ord-pres-resp-idemp}
	Let $\rho: \f(S) \to S$ be a transversal of $\f:S\to T$. The following statements are equivalent:
	\begin{enumerate}
		\item $\rho$ is order-preserving and respects idempotents;\label{rho-ord-pres}
		\item $\rho(ex)=\rho(e)\rho(x)$ for all $e \in E(\f(S))$ and $x \in \f(S)$;\label{rho(e)rho(x)=rho(ex)}
		\item $\rho(xe)=\rho(x)\rho(e)$ for all $e \in E(\f(S))$ and $x \in \f(S)$.\label{rho(x)rho(e)=rho(xe)}
	\end{enumerate}
\end{rem}
\begin{proof}
	Clearly, \cref{rho(e)rho(x)=rho(ex)} implies \cref{rho-ord-pres}. Conversely, if $\rho$ is order-preserving and respects idempotents, then $\rho(ex)\le\rho(x)$, so by \cref{rho(t)rho(t)-inv=rho(tt-inv),rho|_E(T)=pi-inv|_E(T)}
	\begin{align*}
		\rho(ex)=\rho(ex)\rho(ex)\m\rho(x)=\rho(exx\m)\rho(x)=\rho(e)\rho(xx\m)\rho(x)=\rho(e)\rho(x).
	\end{align*} 
	Equivalence of \cref{rho-ord-pres,rho(x)rho(e)=rho(xe)} is proved analogously. 
\end{proof}

	\begin{rem}\label{rho(x-inv)=rho(x)-inv}
		Let $\rho: \f(S) \to S$ be an order-preserving transversal of $\f:S\to T$ which respects idempotents. Then for any $x \in \f(S)$ we have $\rho(x\m)=\rho(x)\m$.
	\end{rem}
	\begin{proof}
		For, by \cref{rho(t)rho(t)-inv=rho(tt-inv),rho-ord-pres-resp-idemp} we have $\rho(x)\rho(x\m)\rho(x)=\rho(xx\m)\rho(x)=\rho(x)$ and, replacing $x$ by $x\m$, we obtain $\rho(x\m)\rho(x)\rho(x\m)=\rho(x\m)$.
	\end{proof}

\begin{rem}\label{rho(x)^ve-commutes-with-rho(e)}
	Let $\rho: \f(S) \to S$ be an order-preserving transversal of $\f:S\to T$ which respects idempotents. Then for any $e \in E(\f(S))$ and $x \in \f(S)$ and $\ve\in\{-1,1\}$ we have $\rho(x)^\ve \rho(e)=\rho(x^\ve ex^{-\ve})\rho(x)^\ve$.
\end{rem}
\begin{proof}
	Indeed, $\rho(x)^\ve=\rho(x^\ve)$ by \cref{rho(x-inv)=rho(x)-inv}, so it suffices to consider the case $\ve=1$. Then $\rho(x)^\ve \rho(e)=\rho(x)\rho(e)=\rho(xe)=\rho(xex\m x)=\rho(xex\m)\rho(x)=\rho(x^\ve ex^{-\ve})\rho(x)^\ve$.
\end{proof}

	\begin{defn}\label{admissible-ext-defn}
		A crossed module extension \cref{eq:seq4terms} of $A$ by $T$ whose homomorphisms $\bt$ and $\pi$ possess order-preserving transversals which respect idempotents will be called \textit{admissible}. The set of equivalence classes of admissible extensions will be denoted by $\mathcal{E}_\le(T,A)$. 
	\end{defn}

\begin{rem}\label{adm-ext-induces-str-norm-c}
	An admissible crossed module extension of $A$ by $T$ induces a strongly normalized $c\in Z^3(T^1,A^1)$. 
\end{rem}
\begin{proof}
	The proof is similar to the proof of \cref{c-is-normalized}, so we skip some details. We choose order-preserving transversals $\rho:T\to S$ and $\s:\bt(N)\to N$ of $\pi$ and $\bt$ which respect idempotents. Let $f(x,y)$ be defined by \cref{rho(x)rho(y)=f(x_y)rho(xy)} and $F(x,y)=\s(f(x,y))$, so that \cref{bt(F(x_y))=f(x_y)} holds. Using \cref{rho(x)rho(y)=f(x_y)rho(xy),rho-ord-pres-resp-idemp} we obtain $\rho(ex)=\rho(e)\rho(x)=f(e,x)\rho(ex)$, whence $f(e,x)=\rho(exx\m)$. It follows that $F(e,x)=\af\circ\rho(exx\m)$. Now, calculating $\rho(e)\rho(x)\rho(y)$ in two different ways, we have 
	\begin{align*}
		f(e,x)f(ex,y)\rho(exy)=f(x,y)f(e,xy)\rho(exy),
	\end{align*}
	so that $f(ex,y)=\rho(e)f(x,y)$. Hence, \cref{lb_rho(x)(F(y_z))F(x_yz)=i(c(x_y_z))F(x_y)F(xy_z)} written for the triple $(e,x,y)\in T^3$ gives
	\begin{align*}
	\af(\rho(e))F(x,y) \af(\rho(exyy\m x\m)) = i(c(e,x,y)) \af(\rho(exx\m)) \af(\rho(e))F(x,y),
	\end{align*}
	yielding $c(e,x,y)=\0(exyy\m x\m)$. The equalities $c(x,e,y)=\0(xeyy\m x\m)$ and $c(x,y,e)=\0(xyey\m x\m)$ are proved in a similar way.
\end{proof}
	It turns out that strongly normalized cocycles are order-preserving. We prove this in the following lemmas.
	\begin{lem}\label{c-ord-pres-iff}
		Let $n\ge 1$ and $c\in C^n(T^1,A^1)$. Then $c\in C^n_\le(T^1,A^1)$ if and only if for all $x_1,\dots,x_n\in T$, $1\le i\le n$ and $e\in E(T)$ one has
		\begin{align}\label{c(x_1...ex_i...x_n)}
		c(x_1,\dots,x_{i-1},ex_i,x_{i+1},\dots,x_n) = \0(\br(x_1\dots x_{i-1}e))c(x_1,\dots,x_n).
		\end{align}
	\end{lem}
	\begin{proof}
		The ``if'' part is obvious. For the ``only if'' part take $c\in C^n_\le(T^1,A^1)$. Then 
		$$
		c(x_1,\dots,x_{i-1},ex_i,x_{i+1},\dots,x_n)\le c(x_1,\dots,x_n),
		$$
		so that
		\begin{align*}
		c(x_1,\dots,x_{i-1},ex_i,x_{i+1},\dots,x_n)&=\br(c(x_1,\dots,x_{i-1},ex_i,x_{i+1},\dots,x_n)) c(x_1,\dots,x_n)\\
		&=\0(\br(x_1\dots x_{i-1}ex_i\dots x_n))c(x_1,\dots,x_n)\\
		&=\0(\br(x_1\dots x_{i-1}e))c(x_1,\dots,x_n),
		\end{align*}
		the latter equality being explained by $c(x_1,\dots,x_n)\in A_{\0(\br(x_1\dots x_n))}$ and
		\begin{align*}
		\br(set)\br(st)=se(tt\m) e (s\m s)(tt\m) s\m=se(s\m s)(tt\m) s\m=\br(se)\br(st).
		\end{align*}
	\end{proof}
	
	\begin{rem}
		The right-hand side analogue of \cref{c(x_1...ex_i...x_n)} is the following:
		\begin{align}\label{c(x_1...x_ie...x_n)}
		c(x_1,\dots,x_{i-1},x_ie,x_{i+1},\dots,x_n) = \0(\br(x_1\dots x_ie))c(x_1,\dots,x_n).
		\end{align}
		In particular,
		\begin{align}\label{c(x_1...x_ie...x_n)-and-}
		c(x_1,\dots,x_{i-1},x_ie,x_{i+1},\dots,x_n) = c(x_1,\dots,x_i,ex_{i+1},x_{i+2},\dots,x_n).
		\end{align}
	\end{rem}

	\begin{lem}\label{lem:E-cond}
		Let $n\ge 1$ and $c\in Z^n(T^1,A^1)$. Then $c$ is strongly normalized if and only if $c$ is normalized and $c\in Z^n_\le(T^1,A^1)$.
	\end{lem}
	\begin{proof}
		The ``if'' part easily follows from \cref{c(e_ex_1...x_(n-1))-triv,c(x_1...x_(i-1)e_e_ex_(i+1)...x_n)-triv,c(x_1...x_(n-1)e_x_n)-triv,c(x_1...ex_i...x_n),c(x_1...x_ie...x_n)}.
		
		The ``only if'' part. Clearly, a strongly normalized $n$-cocycle is normalized. Let us prove that it is order-preserving. Write the $n$-cocycle identity for $c$ and the $(n+1)$-tuple $(e,x_1,\dots,x_n)$:
		\begin{align*}
			&\0(e)c(x_1,\dots,x_n)c(ex_1,x_2,\dots,x_n)\m \\
			&\quad\cdot\prod_{i=1}^{n-1} c(e,x_1,\dots,x_ix_{i+1},\dots,x_n)^{(-1)^{i-1}}\\
			&\quad\cdot c(e,x_1,\dots,x_{n-1})^{(-1)^{n-1}}=\0(\br(ex_1\dots x_n)).
		\end{align*}
		Then \cref{c(x_1...x_(i-1)_e_x_(i+1)...x_n)-triv} implies \cref{c(x_1...ex_i...x_n)} for $i=1$. By induction on $i$ one proves \cref{c(x_1...ex_i...x_n)} for all $1\le i\le n$, where on the $i$-th step the $(n+1)$-tuple $(x_1,\dots,x_{i-1},e,x_i,\dots,x_n)$ is used. The result now follows from \cref{c-ord-pres-iff}.
	\end{proof}
	
	\cref{prop:Egenc,adm-ext-induces-str-norm-c,lem:E-cond} imply the next.
	
	\begin{cor}\label{E_le(T_A)->H3_le(T^1_A^1)}
		There is a well-defined function from $\mathcal{E}_\le(T,A)$ to $H^3_\le(T^1,A^1)$.
	\end{cor}
	\begin{proof}
		Analyzing the proofs of \cref{another-choice-of-F,another-choice-of-rho}, we see that whenever we choose order-preserving transversals of $\bt$ and $\pi$, the corresponding $2$-cochains $u$ and $w$ will also be order-preserving.
	\end{proof}

	\section{From \texorpdfstring{$H^3_\le(T^1,A^1)$}{H³<(T¹,A¹)} to \texorpdfstring{$\mathcal{E}(T,A)$}{E(T,A)}}\label{sec:H3toE}
	The goal of this section is to construct a map from $H^3_\le(T^1,A^1)$ to $\mathcal{E}(T,A)$ which induces a bijective correspondence between $H^3_\le(T^1,A^1)$ and $\mathcal{E}_\le(T,A)$, when $T$ is $F$-inverse.
	
	\subsection{Normalized order-preserving \texorpdfstring{$3$}{3}-cocycles}
	
	\cref{lem:E-cond} shows that each strongly normalized $n$-cocycle is order-preserving. Our aim now is to prove that each order-preserving $3$-cocycle is cohomologous to a (strongly) normalized one.

	\begin{lem}\label{c(e_t_e)=trivial}
		Let $c\in Z^3_\le(T^1,A^1)$. Then for any $t\in T$ and $e,f\in E(T)$ one has
		\begin{align}\label{c(e_t_f)=0(etft^(-1))}
		c(e,t,f)=\0(etft^{-1}).
		\end{align}
	\end{lem}
	\begin{proof}
		Indeed, applying the $3$-cocycle identity to the quadruple $(e,t,f,f)$ and using \labelcref{c(x_1...x_ie...x_n),c(x_1...ex_i...x_n)}, we have
		\begin{align*}
		\0(etft^{-1})&=\0(e)c(t,f,f)c(et,f,f)^{-1} c(e,tf,f) c(e,t,f^2)^{-1} c(e,t,f)\\
		&=\0(etft^{-1})c(e,t,f)=c(e,t,f).
		\end{align*}
	\end{proof}
	
	\begin{lem}\label{c-normalized-iff-c(t_e_e)=c(e_e_t)=triv}
		Let $c\in Z^3_\le(T^1,A^1)$. Then $c$ is normalized if and only if
		\begin{align}
		c(t,e,e)&=\0(tet^{-1}),\label{c(t_e_e)=0(tet^(-1))}\\
		c(e,e,t)&=\0(ett^{-1})\label{c(e_e_t)=0(ett^(-1))}
		\end{align}
		for all $t\in T$ and $e\in E(T)$.
	\end{lem}
	\begin{proof}
		The ``only if'' part is trivial. Let us prove the ``if'' part. Applying the $3$-cocycle identity to the quadruple $(x,y,e,e)$ and using \cref{c(x_1...x_ie...x_n),c(t_e_e)=0(tet^(-1))}, we obtain
		\begin{align*}
		\0(xyey^{-1} x^{-1})&=\eta_x(c(y,e,e))c(xy,e,e)^{-1} c(x,ye,e) c(x,y,e^2)^{-1} c(x,y,e)\\
		&=\eta_x(\0(yey^{-1}))\0(xyey^{-1} x^{-1}) c(x,y,e)=c(x,y,e).
		\end{align*}
		Similarly, using \cref{c(x_1...ex_i...x_n),c(e_e_t)=0(ett^(-1))} and applying the $3$-cocycle identity to $(e,e,x,y)$, we have
		\begin{align*}
		\0(exyy^{-1} x^{-1})&=\eta_e(c(e,x,y))c(e^2,x,y)^{-1} c(e,ex,y) c(e,e,xy)^{-1} c(e,e,x)\\
		&=\0(e) c(e,x,y) \0(exyy^{-1} x^{-1}) \0(exx^{-1})=c(e,x,y).
		\end{align*}
		Finally, \cref{c(e_e_t)=0(ett^(-1)),c(t_e_e)=0(tet^(-1)),c(x_1...x_ie...x_n),c(x_1...ex_i...x_n)} and the $3$-cocycle identity applied to $(x,e,e,y)$ give 
		\begin{align*}
		\0(xeyy^{-1} x^{-1})&=\eta_x(c(e,e,y))c(xe,e,y)^{-1} c(x,e^2,y) c(x,e,ey)^{-1} c(x,e,e)\\
		&=\eta_x(\0(eyy^{-1}))c(x,e,y)\0(xex^{-1})=c(x,e,y).
		\end{align*}
	\end{proof}
	
	\begin{lem}
		Let $c\in Z^3_\le(T^1,A^1)$. Define
		\begin{align}\label{d(x_y)=c(xx^(-1)_x_y)^(-1)c(x_y_y^(-1)y)}
		d(x,y)=c(xx^{-1},x,y)^{-1} c(x,y,y^{-1}y).
		\end{align}
		Then $d\in C^2_\le(T^1,A^1)$ and for all $t\in T$ and $e\in E(T)$
		\begin{align}
		d(e,et)&=c(e,e,t)^{-1},\label{d(e_et)=c(e_e_t)^(-1)}\\
		d(te,e)&=c(t,e,e).\label{d(te_e)=c(t_e_e)}
		\end{align}
	\end{lem}
	\begin{proof}
		Indeed, $d$ is order-preserving, since $c$ is. Moreover, using \labelcref{c(x_1...x_ie...x_n),c(x_1...ex_i...x_n),c(e_t_f)=0(etft^(-1))} we have
		\begin{align*}
		d(e,et)&=c(e,e,et)^{-1} c(e,et,t^{-1}et)=\0(e)c(e,e,t)^{-1}\0(ett^{-1})=c(e,e,t)^{-1},\\
		d(te,e)&=c(tet^{-1},te,e)^{-1} c(te,e,e)=\0(tet^{-1})c(t,e,e)=c(t,e,e).
		\end{align*}
	\end{proof}
	
	\begin{lem}\label{c.delta^2d-normalized}
		Let $c\in Z^3_\le(T^1,A^1)$ and $d\in C^2_\le(T^1,A^1)$ be defined by means of \cref{d(x_y)=c(xx^(-1)_x_y)^(-1)c(x_y_y^(-1)y)}. Then $\tl c:=c\cdot \delta^2 d$ is normalized.
	\end{lem}
	\begin{proof}
		Let $t\in T$ and $e\in E(T)$. Using \cref{d(x_y)=c(xx^(-1)_x_y)^(-1)c(x_y_y^(-1)y),d(e_et)=c(e_e_t)^(-1),d(te_e)=c(t_e_e)} we shall calculate
		\begin{align*}
		(\delta^2 d)(t,e,e)&=\eta_t(d(e,e))d(te,e)^{-1} d(t,e^2)d(t,e)^{-1}\\
		&=\eta_t(\0(e))c(t,e,e)^{-1}\0(tet\m)=c(t,e,e)^{-1},
		\end{align*}
		and
		\begin{align*}
		(\delta^2 d)(e,e,t)&=\eta_e(d(e,t))d(e^2,t)^{-1} d(e,et)d(e,e)^{-1}\\
		&=\0(ett^{-1})c(e,e,t)^{-1} \0(e)=c(e,e,t)^{-1}.
		\end{align*}
		It follows that $\tl c(e,e,t)$ and $\tl c(t,e,e)$ are trivial, so $\tl c$ is normalized by \cref{c-normalized-iff-c(t_e_e)=c(e_e_t)=triv}.
	\end{proof}
	
	\begin{lem}
		Let $c=\delta^2 d$, where $d\in C^2_\le(T^1,A^1)$. If $c$ is normalized, then for all $t\in T$ and $e\in E(T)$
		\begin{align}
		d(e,t)d(e,e)^{-1}&=\0(ett^{-1}),\label{d(e_t)d(e_e)-inv-triv}\\
		\eta_t(d(e,e))d(t,e)^{-1}&=\0(tet^{-1}).\label{eta_t(d(e_e))d(t_e)-inv-triv}
		\end{align}
	\end{lem}
	\begin{proof}
		Indeed, using \cref{c(x_1...ex_i...x_n),c(x_1...x_(i-1)_e_x_(i+1)...x_n)-triv} we have
		\begin{align*}
		\0(ett^{-1})&=c(e,e,t)=(\delta^2 d)(e,e,t)\\
		&=\0(e)d(e,t)d(e^2,t)^{-1} d(e,et) d(e,e)^{-1}\\
		&=d(e,t)d(e,e)^{-1}
		\end{align*}
		and similarly by \cref{c(x_1...x_ie...x_n),c(x_1...x_(i-1)_e_x_(i+1)...x_n)-triv}
		\begin{align*}
		\0(tet^{-1})&=c(t,e,e)=(\delta^2 d)(t,e,e)\\
		&=\eta_t(d(e,e))d(te,e)^{-1} d(t,e^2) d(t,e)^{-1}\\
		&=\eta_t(d(e,e))d(t,e)^{-1}.
		\end{align*}
	\end{proof}
	
	\begin{lem}\label{delta^2d-normalized}
		Let $c\in B^3_\le(T^1,A^1)$ be normalized. Then there is a strongly normalized $\tl d\in C^2_\le(T^1,A^1)$ such that $c=\delta^2\tl d$.
	\end{lem}
	\begin{proof}
		If $c=\delta^2 d$ for some $d\in C^2_\le(T^1,A^1)$, then we define $\tl d = d\cdot \delta^1 u$, where $u(t)=d(t,t^{-1})^{-1}$. We immediately have $u\in C^1_\le(T^1,A^1)$ and $\delta^2\tl d=\delta^2 d=c$. It remains to show that $\tl d$ is strongly normalized. Indeed, using \cref{c(x_1...ex_i...x_n),c(x_1...x_ie...x_n),d(e_t)d(e_e)-inv-triv,eta_t(d(e_e))d(t_e)-inv-triv} we have
		\begin{align*}
		\tl d(e,t)&=d(e,t)(\delta^1 u)(e,t)=d(e,t)\0(e)u(t)u(et)^{-1} u(e)=d(e,t)d(e,e)^{-1}=\0(ett^{-1})\\
		\tl d(t,e)&=d(t,e)(\delta^1 u)(t,e)=d(t,e)\eta_t(u(e))u(te)^{-1} u(t)=d(t,e)\eta_t(d(e,e)^{-1})=\0(tet^{-1}).
		\end{align*} 
	\end{proof}
	
	Summarizing the results of \cref{c.delta^2d-normalized,delta^2d-normalized} we have the following.
	\begin{prop}\label{3-cocycle-cohom-to-normalized}
		Every order-preserving $3$-cocycle is cohomologous to a normalized order-preserving $3$-cocycle. Every normalized order-preserving $3$-coboundary is the coboundary of a strongly normalized order-preserving $2$-cochain.
	\end{prop}

	\subsection{The \texorpdfstring{$E$}{E}-unitary cover through a free group}
	Recall from~\cite{McAlisterReilly77} that an \textit{$E$-unitary cover} of an inverse semigroup $T$ through a group $G$ is an $E$-unitary inverse semigroup $S$, such that
	\begin{enumerate}
		\item $\cG S\cong G$;
		\item there is an idempotent-separating epimorphism $\pi:S\to T$.
	\end{enumerate}
	Each inverse semigroup admits an $E$-unitary cover. We are going to use a specific $E$-unitary cover whose construction is given in~\cite[\S 5.1]{McAlisterReilly77}.
	
	Given a set $X$, we denote by $F(X)$ (resp. $F(X)^1$, $FG(X)$) the free semigroup (resp. monoid, group) over $X$. The length of a word $w\in F(X)$ will be denoted by $l(w)$.
	
	Let $T$ be an inverse semigroup. Consider $T$ as an alphabet $\{[t]\mid t\in T\}$. Let $T\m$ be the alphabet $\{[t]\m\mid t\in T\}$. The map $T\sqcup T\m\to T$, $[t]\mapsto t$, $[t]\m\mapsto t\m$, extends to a (unique) epimorphism of semigroups $\f:F(T\sqcup T\m)\to T$, such that $\f(w\m)=\f(w)\m$, where $w\mapsto w\m$ is the natural involution on $F(T\sqcup T\m)$ sending $[t_1]^{\ve_1}\dots[t_n]^{\ve_n}$ to $[t_n]^{-\ve_n}\dots[t_1]^{-\ve_1}$. Similarly, there is a (unique) epimorphism of semigroups $\psi:F(T\sqcup T\m)\to FG(T)$, such that $\psi([t])=[t]$ and $\psi([t]\m)=[t]\m$. It also satisfies $\psi(w\m)=\psi(w)\m$. Define $\Phi:T\to 2^{FG(T)}$ by
	\begin{align*}
		\Phi(t)=\psi(\f\m(t)).
	\end{align*}
	Then by \cite[Prop. 3.2 and 5.1]{McAlisterReilly77} the semigroup $S:=\Pi(T,FG(T),\Phi)$ is an $E$-unitary cover of $T$ through $FG(T)$, where
	\begin{align*}
		\Pi(T,FG(T),\Phi)=\{(t,w)\in T\times FG(T)\mid w\in\Phi(t)\}
	\end{align*}
	with the coordinatewise multiplication, and the epimorphism $\pi:S\to T$ is given by $\pi(t,w)=t$.
	
	We would like to rewrite the condition $w\in\Phi(t)$ in a more convenient form. Let us say that $w\in F(T\sqcup T\m)$ is \textit{irreducible}, if it has no subwords of the form $uu\m$, where $u\in F(T\sqcup T\m)$. Denote by $\e$ the identity element of $FG(T)$. Since each $w\in FG(T)\setminus\{\e\}$ admits a unique representation as a non-empty irreducible word $\irr w$ over $T\sqcup T\m$, there is a well-defined map $\mathrm{irr}:FG(T)\setminus\{\e\}\to F(T\sqcup T\m)$. Observe that $\irr{w\m}=\irr w\m$. Moreover, by definition
	\begin{align}
		\psi(\irr u)&=u \mbox{ for all $u\in FG(T)\setminus\{\e\}$},\label{psi(irr(u))=u}\\
		\irr{\psi(v)}&=v\mbox{ if and only if $v$ is irreducible},\label{irr(psi(v))=v}\\
		l(\irr{\psi(v)})&\le l(v) \mbox{ for all $v$, and the equality holds only for irreducible $v$}.\label{l(irr(psi(v)))<=l(v)}
	\end{align}
	We now define $\nu:FG(T)\to T^1$, where $T^1=T\sqcup\{1\}$, by 
	\begin{align}\label{nu(w)=f(irr(w))}
		\nu(w)=
		\begin{cases}
			\f(\irr w), & w\ne\e,\\
			1,          & w=\e.
		\end{cases}
	\end{align}
	
	\begin{rem}\label{nu(u)nu(v)<=nu(uv)}
		We have $\nu(u)\nu(v)\le\nu(uv)$ for all $u,v\in FG(T)$, and the equality holds whenever $\e\in\{u,v,uv\}$ or $\irr u\irr v$ is irreducible.
	\end{rem}
	\begin{proof}
		The case $\e\in\{u,v,uv\}$ is obvious, so we assume $\e\not\in\{u,v,uv\}$. Suppose that $\irr u\irr v$ is not irreducible. Let $c$ be the maximal suffix of $\irr u$ such that $c\m$ is a prefix of $\irr v$. Then $\irr u=ac$, $\irr v=c\m b$ and $\irr{uv}=ab$ for some $a,b,c\in F(T\sqcup T\m)$, so
		\begin{align*}
			\nu(u)\nu(v)=\f(ac)\f(c\m b)=\f(a)\f(c)\f(c)\m\f(b)\le\f(a)\f(b)=\f(ab)=\nu(uv).
		\end{align*}
		If $\irr u\irr v$ is irreducible, then $\irr{uv}=\irr u\irr v$, so applying the homomorphism $\f$, we obtain the desired equality.
	\end{proof}
	
	\begin{lem}
		Let $T$ be an inverse semigroup, $t\in T$ and $w\in FG(T)$. Then $w\in\Phi(t)$ if and only if $t\le\nu(w)$.
	\end{lem}
	\begin{proof}
		Consider first the case $w=\e$. Let $\e\in\Phi(t)$. Then $t\in E(T)$ by \cite[Proposition 5.1]{McAlisterReilly77}, and thus $t\le 1=\nu(\e)$. Conversely, if $t\le \nu(\e)=1$, then $t\in E(T)$, i.e. $tt\m=t$, so $\f([t][t]\m)=t$, whence $\e=\psi([t][t]\m)\in\Phi(t)$.

		Now let $w \ne\e$. By definition $w\in\Phi(t)$ means that $w=\psi(u)$, where $\f(u)=t$. If $u$ is irreducible, then $\irr w=\irr{\psi(u)}=u$ by \cref{irr(psi(v))=v}. Therefore, $\nu(w)=\f(\irr w)=\f(u)=t$ by \cref{nu(w)=f(irr(w))}. If $u$ is not irreducible, then $u=u' vv\m u''$ for some $u',u''\in F(T\sqcup T\m)^1$ and $v\in F(T\sqcup T\m)$. Hence, $t=\f(u)=\f(u')\f(v)\f(v)\m\f(u'')\le \f(u'u'')$, where $\psi(u'u'')=\psi(u)=w$ and $l(u'u'')<l(u)$. The proof of $t\le\nu(w)$ now follows by induction on $l(u)$, where $\psi(u)=w$.
		
		Conversely, let $t\le\nu(w)$. Then $t=tt\m\nu(w)=\f([t][t]\m\irr w)$ by \cref{nu(w)=f(irr(w))} and 
		\begin{align*}
			\psi([t][t]\m\irr w)=\psi(\irr w)=w
		\end{align*}
		by \cref{psi(irr(u))=u}. Thus, $w\in\Phi(t)$.
	\end{proof}

	\begin{cor}
		We have $S=\{(t,w)\in T\times FG(T)\mid t\le\nu(w)\}$ and $E(S)=\{(e,\e)\mid e\in E(T)\}=E(T)\times\{\e\}$.
	\end{cor}
	
	\begin{rem}
		Each $(t,w)\in S$ decomposes as 
		\begin{align}
			(t,w)=(tt\m,\e)(\nu(w),w).
		\end{align}
		In particular, if $T$ is a group, then $(t,w)=(\nu(w),w)$ and thus $S\cong FG(T)$.
	\end{rem}
	
	\subsection{From \texorpdfstring{$c\in Z^3_\le(T^1,A^1)$}{c in Z³<(T¹,A¹)} to a crossed module extension of \texorpdfstring{$A$}{A} by \texorpdfstring{$T$}{T}}\label{c->ext-sec}
	Let $T$ be an inverse semigroup, $S$ its $E$-unitary cover through $FG(T)$ and $\pi:S\to T$ the corresponding covering epimorphism. We put $K=\pi\m(E(T))$. Since $\pi$ is idempotent-separating, $K$ is a semilattice of groups $\bigsqcup_{e\in E(T)}K_e$, where
	\begin{align*}
		K_e=\{(e,w)\in E(T)\times FG(T)\mid e\le\nu(w)\}.
	\end{align*}
	Observe that
	\begin{align*}
		e\le\nu(w)\iff e\nu(w)=e\iff \nu(w)e=e.
	\end{align*}
	Given a $T$-module structure $(\0,\eta)$ on a semilattice of abelian groups $A$, we set $N:=\bigsqcup_{e\in E(T)}(A_e\times K_e)$ with coordinatewise multiplication. Then $N$ is a semilattice of groups\footnote{In fact, $N$ is a direct product of $A$ and $K$ in the category of $E(T)$-semilattices of groups.}. To simplify the notation, we write the elements of $N$ as triples $(a,e,w)$, so that
	\begin{align}\label{N=triples-(a_e_w)}
		N=\{(a,e,w)\in A\times E(T)\times FG(T)\mid a\in A_e\mbox{ and }e\le\nu(w)\}.
	\end{align}
	
	We also define $i:A\to N$, $\af:E(S)\to E(N)$ and $\bt:N\to S$ as follows:
	\begin{align}
		i(a)&=(a,\0\m(aa\m),\e),\label{i(a)=(a_0(aa-inv)_1)}\\
		\af(e,\e)&=(\0(e),e,\e),\label{af(e_1)=(0(e)_e_1)}\\
		\bt(a,e,w)&=(e,w).\label{bt(a_e_w)=(e_w)}
	\end{align}
	
	\begin{prop}\label{properties-of-i-af-bt}
		The following holds.
		\begin{enumerate}
			\item The map $i$ is a monomorphism $A\to N$ such that $i(A)=\bt\m(E(S))$.\label{i-mono-A->N}
			\item The map $\af$ is an isomorphism $E(S)\to E(N)$.\label{af-iso-E(S)->E(N)} 
			\item The map $\bt$ is an idempotent-separating homomorphism $N\to S$, such that $\bt(N)=\pi\m(E(T))$ and $\bt|_{E(N)}=\af\m$.\label{bt-idemp-sep-hom}
		\end{enumerate}
	\end{prop}
	\begin{proof}
		It is obvious that $i$ is well-defined and injective. Since $(ab)(ab)\m=(aa\m)(bb\m)$ for all $a,b\in A$, we have $i(ab)=i(a)i(b)$. Now, $i(A)=\{(a,e,w)\in N\mid w=\e\}=\bt\m(E(S))$. So, \cref{i-mono-A->N} is proved.
			
		Since $\0$ is an isomorphism $E(T)\to E(A)$, we have \cref{af-iso-E(S)->E(N)}.
		 
		The multiplication on $N$ is coordinatewise, so $\bt$ is a homomorphism. By definition $\bt(N)=K=\pi\m(E(T))$. The idempotents of $N$ are of the form $(\0(e),e,\e)=\af(e,\e)$, where $e\in E(T)$. Hence $\bt$ is idempotent-separating and $\bt|_{E(N)}=\af\m$, proving \cref{bt-idemp-sep-hom}.
	\end{proof}
	
	Let now $c\in Z^3_\le(T^1,A^1)$ and assume that $c$ is normalized, then $c$ is strongly normalized by \cref{lem:E-cond}. We are going to use $c$ to construct a homomorphism $\lb:S\to\mend N$. We first define, for any $t\in T$, the auxiliary map $\xi_t:F(T\sqcup T\m)\to A$ as follows. 
	
	If $w=[x]\m u$ for some $x\in T$ and $u\in F(T\sqcup T\m)^1$, then
	\begin{align}\label{xi_t([x]-inv.u)}
	\xi_t(w):=c(t,x\m,x)\m\xi_t([x\m]u).
	\end{align}
	The rest of the definition will be given by induction on $l(w)$.
	
	\textit{Base of induction.} If $w=[x]$ for some $x\in T$, then
	\begin{align}\label{xi_t([x])}
		\xi_t(w)&:=\0(\br(tx)).
	\end{align}
	
	\textit{Inductive step.} $l(w)>1$ and $w=[x]u$ for some $x\in T$ and $u\in F(T\sqcup T\m)$. 
	
	\textbf{Case 1.} If $w=[x][y]v$ for some $x,y\in T$ and $v\in F(T\sqcup T\m)^1$, then 
	\begin{align}\label{xi_t([x][y]v)}
	\xi_t(w):=c(t,x,y)\xi_t([xy]v).		
	\end{align}
	Since $l([xy]v)<l(w)$, we may use the inductive step.  
	
	\textbf{Case 2.} If $w=[x][y]\m v$ for some $x,y\in T$ and $v\in F(T\sqcup T\m)^1$, then 
	\begin{align}\label{xi_t([x][y]-inv.v)}
	\xi_t(w):=c(t,xy\m,y)\m\xi_t([xy\m]v).
	\end{align}
	Since $l([xy\m]v)<l(w)$, we may use the inductive step.

	 \begin{lem}\label{xi_t(w)-belongs-to-A_r(tf(w))}
	 	For any $t\in T$ and $w\in F(T\sqcup T\m)$ we have $\xi_t(w)\in A_{\br(t\f(w))}$.
	 \end{lem}
	 \begin{proof}
	 	It is enough to prove that
        \begin{align}\label{r(xi_t(w))=r(tf(w))}
	 	\br(\xi_t(w))=\theta(\br(t\f(w))).
	 	\end{align}
	 	
	 	Suppose first that $w=[x]\m w'$. Then $\xi_t(w)=c(t,x\m,x)\m\xi_t([x\m]w')$ by \cref{xi_t([x]-inv.u)}. Observe that $c(t,x\m,x)\in A_{\br(tx\m)}$ and $\f(w)=x\m\f(w')=\f([x\m]w')$, so $\br(t\f(w))\le \br(tx\m)$. The equality \cref{r(xi_t(w))=r(tf(w))} reduces to $\br(\xi_t([x\m]w'))=\theta(\br(t\f([x\m]w')))$. Thus, it suffices to consider the case $w=[x]w'$, which will be proved by induction on $l(w)$.
	 	
	 	The equality \cref{r(xi_t(w))=r(tf(w))} is clear, when $w=[x]$ by \cref{xi_t([x])}. So let $l(w)>1$.
	 	 
	 	\textbf{Case 1.} $w=[x][y]w''$. Then $\xi_t(w)=c(t,x,y)\xi_t([xy]w'')$ by \cref{xi_t([x][y]v)}. Since $c(t,x,y)\in A_{\br(txy)}$, $\f(w)=\f([xy]w'')$ and $\br(t\f(w))\le \br(txy)$, it is enough to prove \cref{r(xi_t(w))=r(tf(w))} for $[xy]w''$. In view of $l([xy]w'')<l(w)$, we may apply the inductive step to $[xy]w''$.  
	 	
	 	\textbf{Case 2.} $w=[x][y]\m w''$. This case is similar to the previous one.
	 \end{proof}

 	 \begin{lem}\label{xi_t-mult-reduction-to-u=[x]}
 	 	Let $t\in T$  and $u\in F(T\sqcup T\m)$ and $v\in F(T\sqcup T\m)^1$. Then 
 	 	\begin{align}\label{xi_t(uv)=xi_t(u)xi_t([f(u)]v)}
 	 		\xi_t(uv)=\xi_t(u)\xi_t([\f(u)]v).
 	 	\end{align}
 	 \end{lem}
 	 \begin{proof}
 	 	When $v$ is empty, \cref{xi_t(uv)=xi_t(u)xi_t([f(u)]v)} is trivial in view of \cref{xi_t(w)-belongs-to-A_r(tf(w)),xi_t([x])}.  We hence fix a non-empty $v$ and prove the statement for all $u\in F(T\sqcup T\m)$. If $u=[x]\m u'$, then $\xi_t(u)=c(t,x\m,x)\m\xi_t([x\m]u')$ and $\xi_t(uv)=c(t,x\m,x)\m\xi_t([x\m]u'v)$ by \cref{xi_t([x]-inv.u)}, and moreover $\f([x\m]u')=\f(u)$, so it suffices to consider the case $u=[x]u'$, which will be proved by induction on $l(u)$.  
 	 	
 	 	If $u=[x]$, then $\xi_t(u)=\0(\br(tx))$ thanks to \cref{xi_t([x])}. Since $\xi_t([\f(u)]v)\in A_{\br(tx\f(v))}$ by \cref{xi_t(w)-belongs-to-A_r(tf(w))} and $\br(tx\f(v))\le \br(tx)$, \cref{xi_t(uv)=xi_t(u)xi_t([f(u)]v)} becomes trivial.
 	 	
 	 	Let $l(u)>1$. We have the following cases.
 	 	
 	 	\textbf{Case 1.} $u=[x][y]u''$. Then $\xi_t(u)=c(t,x,y)\xi_t([xy]u'')$ and $\xi_t(uv)=c(t,x,y)\xi_t([xy]u''v)$ by \cref{xi_t([x][y]v)}. Since $\f(u)=\f([xy]u'')$ and $l([xy]u'')<l(u)$, we may apply the induction hypothesis to $[xy]u''$.  
 	 	
 	 	\textbf{Case 2.} $u=[x][y]\m u''$. This case is similar to the previous one.
 	 \end{proof}
  
  	\begin{lem}\label{xi_t-mult-case-u=[x]}
  		Let $t,x\in T$, $v\in F(T\sqcup T\m)$ and $e\in E(T)$ such that $e\le x$. Then
  		\begin{align}\label{xi_t([x]v)=xi_t([x])xi_t(v)}
  		\0(\br(te))\xi_t([x] v)=\0(\br(te))\xi_t(v).
  		\end{align}
  	\end{lem}
  	\begin{proof}
  		
  		Consider first the case $v=[y]v'$. Then by \cref{xi_t([x][y]v),c(x_1...x_(i-1)_e_x_(i+1)...x_n)-triv,c(x_1...ex_i...x_n)}
  		\begin{align*}
  		\0(\br(te))\xi_t([x] v)&=\0(\br(te))c(t,x,y)\xi_t([xy]v')=c(t,ex,y)\xi_t([xy]v')\\
  		&=c(t,e,y)\xi_t([xy]v')=\0(\br(tey))\xi_t([xy]v').
  		\end{align*}
  		
    Now, since $\xi_t([xy]v') \in A_{\br(txy\varphi(v'))}$ and using the fact that $e \leq x$:
  		\begin{align*}
  		    \theta(\br(tey))\theta(\br(txy\varphi(v'))) & =  \theta(te \cdot yy^{-1} \cdot t^{-1}t \cdot xy\varphi(v')\varphi(v')^{-1}y^{-1}x^{-1}t^{-1}) \\
  		    & = \theta(teyy^{-1}xy\varphi(v')\varphi(v')^{-1}y^{-1}x^{-1}t^{-1}) \\
  		    & = \theta(tyy^{-1}exy\varphi(v')\varphi(v')^{-1}y^{-1}x^{-1} t^{-1}) \\
  		    & = \theta(t \cdot yy^{-1} \cdot ey\varphi(v')\varphi(v')^{-1}y^{-1}x^{-1} t^{-1}) \\
  		    & = \theta(tey\varphi(v')\varphi(v')^{-1}y^{-1}x^{-1} t^{-1}) \\
  		    & = \theta(te \cdot xy\varphi(v')\varphi(v')^{-1}y^{-1}x^{-1} \cdot t^{-1} t \cdot t^{-1}) \\
  		    & = \theta(tet^{-1} t xy \varphi(v')\varphi(v')^{-1}y^{-1}x^{-1} t^{-1}) \\
  		    & = \theta(tet^{-1}) \theta(txy\varphi(v')\varphi(v')^{-1}y^{-1} x^{-1} t^{-1}) \\
  		    & = \theta(\br(te)) \theta(\br(txy\varphi(v'))).
  		    \end{align*}
  		So we need to prove
  		\begin{align}\label{xi_t([xy]v')=xi_t([y]v')}
  		\0(\br(te))\xi_t([xy]v')=\0(\br(te))\xi_t([y]v').
  		\end{align}
  		The proof will be by induction on $l(v')$. If $v'$ is empty, then the result holds by \cref{xi_t([x])} and $e\le x$. Otherwise, we have the following cases.
  		
  		\textbf{Case 1.} $v'=[z]v''$. Then by \cref{xi_t([x][y]v),c(x_1...ex_i...x_n)}, $e\le x$ and the induction hypothesis applied to $v''$
  		\begin{align*}
  			\0(\br(te))\xi_t([xy]v')&=\0(\br(te))\xi_t([xy][z]v'')=\0(\br(te))c(t,xy,z)\xi_t([xyz]v'')\\
  			&=c(t,exy,z)\xi_t([xyz]v'')=c(t,ey,z)\xi_t([xyz]v'')\\
  			&=\0(\br(te))c(t,y,z)\xi_t([xyz]v'')=\0(\br(te))c(t,y,z)\xi_t([yz]v'')\\
  			&=\0(\br(te))\xi_t([y][z]v'')=\0(\br(te))\xi_t([y]v').
  		\end{align*}
  		
  		\textbf{Case 2.} $v'=[z]\m v''$. Then by \cref{xi_t([x][y]-inv.v),c(x_1...ex_i...x_n)}, $e\le x$ and the induction hypothesis applied to $v''$
  		\begin{align*}
  		\0(\br(te))\xi_t([xy]v')&=\0(\br(te))\xi_t([xy][z]\m v'')=\0(\br(te))c(t,xyz\m,z)\m\xi_t([xyz\m]v'')\\
  		&=c(t,exyz\m,z)\m\xi_t([xyz\m]v'')=c(t,eyz\m,z)\m\xi_t([xyz\m]v'')\\
  		&=\0(\br(te))c(t,yz\m,z)\m\xi_t([xyz\m]v'')\\
  		&=\0(\br(te))c(t,yz\m,z)\m\xi_t([yz\m]v'')\\
  		&=\0(\br(te))\xi_t([y][z]\m v'')=\0(\br(te))\xi_t([y]v').
  		\end{align*}
  		
  		Finally, if $v=[y]\m v'$, then by \cref{xi_t([x][y]-inv.v),c(x_1...ex_i...x_n),xi_t([xy]v')=xi_t([y]v'),xi_t([x]-inv.u)}
  		\begin{align*}
  		\0(\br(te))\xi_t([x] v)&=\0(\br(te))c(t,xy\m,y)\m\xi_t([xy\m]v')=c(t,exy\m,y)\m\xi_t([xy\m]v')\\
  		&=c(t,ey\m,y)\m\xi_t([xy\m]v')=\0(\br(te))c(t,y\m,y)\m\xi_t([xy\m]v')\\
  		&=\0(\br(te))c(t,y\m,y)\m\xi_t([y\m]v')=\0(\br(te))\xi_t([y]\m v')=\0(\br(te))\xi_t(v).
  		\end{align*}
  	\end{proof}
  
 
 \begin{cor}\label{xi_t(e[x])}
 	Let $t\in T$, $v\in F(T\sqcup T\m)$ and $e\in E(T)$. Then
 	\begin{align}\label{xi_t([e]v)=0(r(te))xi_t(v)}
 	\xi_t([e] v)=\0(\br(te))\xi_t(v).
 	\end{align}
 \end{cor}
\begin{proof}
	Indeed, $\xi_t([e] v)\in A_{\br(te\f(v))}$ by \cref{xi_t(w)-belongs-to-A_r(tf(w))} and $\br(te\f(v))\le\br(te)$, so $\xi_t([e] v)=\0(\br(te))\xi_t([e] v)$, the latter being $\0(\br(te))\xi_t(v)$ thanks to \cref{xi_t-mult-case-u=[x]}.
\end{proof}

 	\begin{lem}\label{xi_t([x]w)xi_t(w-inv)-reduction}
 		Let $t,x\in T$, $w\in F(T\sqcup T\m)$ and $e\in E(T)$ such that $e\le x\f(w)$. Then
 		\begin{align}\label{xi_t([x]w)xi_t(w-inv)-triv}
 			\0(\br(te))\xi_t([x]w)\xi_t(w\m)=\0(\br(te)).
 		\end{align}
 	\end{lem}
 \begin{proof}
 	The proof will be by induction on $l(w)$. If $w=[y]$, then by \cref{xi_t([x][y]v),xi_t([x]),xi_t([x]-inv.u)}
 	\begin{align*}
 	\0(\br(te))\xi_t([x]w)\xi_t(w\m)=\0(\br(te))c(t,x,y)c(t,y\m,y)\m.
 	\end{align*}
 	But thanks to \labelcref{c(x_1...ex_i...x_n),c(x_1...x_ie...x_n)-and-} and $e\le xy$
 	\begin{align*}
 		\0(\br(te))c(t,y\m,y)=c(t,ey\m,y)=c(t,exyy\m,y)=\0(\br(te))c(t,x,y),
 	\end{align*}
 	whence \cref{xi_t([x]w)xi_t(w-inv)-triv}. If $w=[y]\m$, then by \cref{xi_t([x]),xi_t([x]-inv.u),c(x_1...ex_i...x_n),c(x_1...x_(i-1)_e_x_(i+1)...x_n)-triv} and $e\le xy\m$
 	\begin{align*}
 	\0(\br(te))\xi_t([x]w)\xi_t(w\m)&=\0(\br(te))c(t,xy\m,y)\m \0(\br(ty))=c(t,exy\m,y)\m \0(\br(ty))\\
 	&=c(t,e,y)\m \0(\br(ty))=\0(\br(tey))=\0(\br(teyy\m))\\
 	&=\0(\br(texy\m yy\m))=\0(\br(texy\m))=\0(\br(te)).
 	\end{align*}
 	
 	Let $l(w)>1$. 
 	
 	\textbf{Case 1.} $w=[y]w'$, where $w'\in F(T\sqcup T\m)$. Then by \cref{xi_t([x][y]v),xi_t(uv)=xi_t(u)xi_t([f(u)]v)}
 	\begin{align*}
 		\0(\br(te))\xi_t([x]w)\xi_t(w\m)&=\0(\br(te))c(t,x,y)\xi_t([xy]w')\xi_t(w'\m[y]\m)\\
 		&=\0(\br(te))c(t,x,y)\xi_t([xy]w')\xi_t(w'\m)\xi_t([\f(w'\m)][y]\m)\\
 		&=\0(\br(te))c(t,x,y)\xi_t([xy]w')\xi_t(w'\m)c(t,\f(w')\m y\m,y)\m.
 	\end{align*}
 	Observe that by \cref{c(x_1...ex_i...x_n),c(x_1...x_ie...x_n)} and $e\le xy\f(w')$
 	\begin{align*}
 		\0(\br(te))c(t,\f(w')\m y\m,y)&=c(t,e\f(w')\m y\m,y)=c(t,exy\f(w')\f(w')\m y\m,y)\\
		&=\0(\br(te))c(t,x,y).
 	\end{align*}
 	Moreover, since $\f([xy]w')=xy\f(w')=x\f(w)\ge e$ and $l(w')<l(w)$, by induction hypothesis we have $\0(\br(te))\xi_t([xy]w')\xi_t(w'\m)=\0(\br(te))$. It follows that
 	\begin{align*}
 		\0(\br(te))\xi_t([x]w)\xi_t(w\m)&=\0(\br(te))\0(\br(txy))=\0(\br(txyy\m x\m e))\\
 		&=\0(\br(txyy\m x\m exy\f(w')))=\0(\br(texy\f(w')))=\0(\br(te)).
 	\end{align*}

 	\textbf{Case 2.} $w=[y]\m w'$, where $w'\in F(T\sqcup T\m)$. Then by \cref{xi_t([x]),xi_t(uv)=xi_t(u)xi_t([f(u)]v),xi_t([x][y]-inv.v),xi_t([x][y]v)}
 	\begin{align*}
 	\0(\br(te))\xi_t([x]w)\xi_t(w\m)&=\0(\br(te))c(t,xy\m,y)\m\xi_t([xy\m]w')\xi_t(w'\m[y])\\
 	&=\0(\br(te))c(t,xy\m,y)\m\xi_t([xy\m]w')\xi_t(w'\m)\xi_t([\f(w'\m)][y])\\
 	&=\0(\br(te))c(t,xy\m,y)\m\xi_t([xy\m]w')\xi_t(w'\m) c(t,\f(w')\m,y).
 	\end{align*}
 	Thanks to \cref{c(x_1...ex_i...x_n),c(x_1...x_ie...x_n)} and $e\le xy\m\f(w')$
 	\begin{align*}
 		\0(\br(te))c(t,\f(w')\m,y)&=c(t,e\f(w')\m,y)=c(t,exy\m\f(w')\f(w')\m,y)\\
 		&=\0(\br(te))c(t,xy\m,y).
 	\end{align*}
 	By induction hypothesis $\0(\br(te))\xi_t([xy\m]w')\xi_t(w'\m)=\0(\br(te))$, whence the result follows as in the previous case.
 \end{proof}

\begin{cor}\label{xi_t(w)xi_t(w-inv)-reduction}
	Let $t\in T$, $w\in F(T\sqcup T\m)$ and $e\in E(T)$ such that $e\le \f(w)$. Then
	\begin{align}\label{xi_t(w)xi_t(w-inv)-triv}
	\0(\br(te))\xi_t(w)\xi_t(w\m)=\0(\br(te)).
	\end{align}
\end{cor}
\begin{proof}
	If $w=[x]$, then by \cref{xi_t([x]-inv.u),xi_t([x]),c(x_1...ex_i...x_n)} and $e\le x$
	\begin{align*}
		\0(\br(te))\xi_t(w)\xi_t(w\m)&=\0(\br(te))\0(\br(tx))c(t,x\m,x)\m=c(t,ex\m,x)\m\\
		&=c(t,e,x)\m=\0(\br(tex))=\0(\br(te)).
	\end{align*}
	The case $w=[x]\m$ is analogous.
	
	Let $l(w)>1$. 
	
	\textbf{Case 1.} $w=[x]w'$. Then by \cref{xi_t([x]w)xi_t(w-inv)-reduction,xi_t-mult-reduction-to-u=[x],xi_t([x][y]-inv.v),c(x_1...ex_i...x_n),c(x_1...x_(i-1)_e_x_(i+1)...x_n)-triv} and $e\le\f(w)$
	\begin{align*}
		\0(\br(te))\xi_t(w)\xi_t(w\m)&=\0(\br(te))\xi_t([x]w')\xi_t(w'\m[x]\m)\\
		&=\0(\br(te))\xi_t([x]w')\xi_t(w'\m)\xi_t([\f(w')\m][x]\m)\\
		&=\0(\br(te))c(t,\f(w')\m x\m,x)\m=c(t,e\f(w)\m,x)\m\\
		&=c(t,e,x)\m=\0(\br(tex))=\0(\br(txx\m e))\\
		&=\0(\br(txx\m e x\f(w')))=\0(\br(te x\f(w')))=\0(\br(te)).
	\end{align*}
	
	\textbf{Case 2.} $w=[x]\m w'$. Then by \cref{xi_t([x]w)xi_t(w-inv)-reduction,xi_t-mult-reduction-to-u=[x],xi_t([x][y]v),c(x_1...ex_i...x_n),c(x_1...x_(i-1)_e_x_(i+1)...x_n)-triv,xi_t([x]-inv.u),xi_t([x]),c(x_1...x_ie...x_n)} and $e\le\f(w)$
	\begin{align*}
		\0(\br(te))\xi_t(w)\xi_t(w\m)&=\0(\br(te))\xi_t([x]\m w')\xi_t(w'\m[x])\\
		&=\0(\br(te))c(t,x\m,x)\m\xi_t([x\m] w')\xi_t(w'\m)\xi_t([\f(w')\m][x])\\
		&=\0(\br(te))c(t,x\m,x)\m c(t,\f(w')\m,x)\\
		&=c(t,x\m,x)\m c(t,e\f(w')\m,x)\\
		&=c(t,x\m,x)\m c(t,ex\m\f(w')\f(w')\m,x)\\
		&=c(t,x\m,x)\m c(t,x\m,x)\0(\br(tex\m\f(w')))\\
		&=\0(\br(tx\m))\0(\br(tex\m\f(w')))=\0(\br(tex\m\f(w')))=\0(\br(te)).
	\end{align*}
\end{proof}

	\begin{lem}\label{xi_t([x]w)-equals-xi_t([y-inv]w)}
		Let $t,x,y\in T$, $w\in F(T\sqcup T\m)^1$ and $e\in E(T)$ such that $e\le xy$. Then
		\begin{align}\label{xi_t([x]w)=xi_t([y-inv]w)}
		\0(\br(te))\xi_t([x]w)=\0(\br(te))\xi_t([y\m]w).
		\end{align}
	\end{lem}
\begin{proof}
	We will prove by induction on $l(w)$ the more general equality
	\begin{align}\label{xi_t([xz]w)=xi_t([y-inv.z]w)}
	\0(\br(te))\xi_t([xz]w)=\0(\br(te))\xi_t([y\m z]w),
	\end{align}
	where $z\in T^1$.
	
	If $w$ is empty, then $\0(\br(te))\xi_t([xz]w)=\0(\br(texz))$ and $\0(\br(te))\xi_t([y\m z]w)=\0(\br(tey\m z))$ by \cref{xi_t([x])}. But $ey\m z=exyy\m z=exyy\m x\m xz=exz$ in view of $e\le xy$.
	
	Now let $l(w)>0$. 
	
	\textbf{Case 1.} $w=[s]w'$ for some $w'\in F(T\sqcup T\m)^1$. Then by \labelcref{xi_t([x][y]v)} 
	\begin{align*}
	\0(\br(te))\xi_t([xz]w)&=\0(\br(te))c(t,xz,s)\xi_t([xzs]w'),\\
	\0(\br(te))\xi_t([y\m z]w)&=\0(\br(te))c(t,y\m z,s)\xi_t([y\m zs]w').
	\end{align*}
	But by 
	\labelcref{c(x_1...ex_i...x_n)} and $e\le xy$
	\begin{align*}
	\0(\br(te))c(t,y\m z,s)&=c(t,ey\m z,s)=c(t,exyy\m z,s)\\
	&=c(t,exyy\m x\m xz,s)=\0(\br(te))c(t,xz,s).
	\end{align*}
	The result now follows from the induction hypothesis applied to $w'$. 
	
	\textbf{Case 2.} $w=[s]\m w'$, then by \labelcref{xi_t([x][y]-inv.v)}
	\begin{align*}
	\0(\br(te))\xi_t([xz]w)&=\0(\br(te))c(t,xzs\m,s)\m\xi_t([xzs\m]w'),\\
	\0(\br(te))\xi_t([y\m z]w)&=\0(\br(te))c(t,y\m zs\m,s)\m\xi_t([y\m zs\m]w').
	\end{align*}
	But by \labelcref{c(x_1...ex_i...x_n)} and $e\le xy$
	\begin{align*}
	\0(\br(te))c(t,y\m zs\m,s)&=c(t,ey\m zs\m,s)=c(t,exyy\m zs\m,s)\\
	&=c(t,exyy\m x\m xzs\m,s)=\0(\br(te))c(t,xzs\m,s).
	\end{align*}
	It remains to use the induction hypothesis.
\end{proof}

\begin{rem}\label{rem-didnt-use-3coc-commut} Observe that in the proofs of \cref{xi_t(w)-belongs-to-A_r(tf(w)),xi_t-mult-reduction-to-u=[x],xi_t-mult-case-u=[x],xi_t([x]w)xi_t(w-inv)-reduction,xi_t(w)xi_t(w-inv)-reduction,xi_t([x]w)-equals-xi_t([y-inv]w)} we did not use the $3$-cocycle identity for $c$ and the commutativity of $A$.
\end{rem}

	We now introduce $\z_t:FG(T)\to A$, where $t\in T$, as follows:
	\begin{align}\label{z_t(w)=xi_t(irr(w))}
	\z_t(w)=
	\begin{cases}
	\xi_t(\irr w), & w\ne\e,\\
	\0(\br(t)),    & w=\e.
	\end{cases}
	\end{align}

	Finally, for any $t\in T$ and $(a,e,w)\in N$ define
	\begin{align}\label{gm_t(a_e_w)=(z_t(w)eta_t(a)_tet-inv_[t]w[t]-inv)}
		\gm_t(a,e,w)=(\z_t(w)\eta_t(a),tet\m,[t]w[t]\m).
	\end{align}

	\begin{lem}
		For any $t\in T$ and $(a,e,w)\in N$ we have $\gm_t(a,e,w)\in N$. 
	\end{lem}
	\begin{proof}
		Observe that in view of \cref{nu(u)nu(v)<=nu(uv),nu(w)=f(irr(w)),N=triples-(a_e_w)}
		\begin{align*}
		tet\m\le t\nu(w)t\m=\nu([t])\nu(w)\nu([t]\m)\le\nu ([t]w[t]\m),
		\end{align*}
		so $(tet\m,[t]w[t]\m)\in K_{tet\m}$. It remains to show that $\z_t(w)\eta_t(a)\in A_{tet\m}$. This is obvious, when $w=\e$, because $\eta_t(a)\in A_{tet\m}$. Otherwise, $\z_t(w)=\xi_t(\irr w)\in A_{t\f(\irr w)\f(\irr w)^{-1}t\m}=A_{t\nu(w)\nu(w)^{-1}t\m}$ by \cref{z_t(w)=xi_t(irr(w)),xi_t(w)-belongs-to-A_r(tf(w)),nu(w)=f(irr(w))}. Since, $e\le\nu(w)$ and $\eta_t(a)\in A_{tet\m}$, the product $\z_t(w)\eta_t(a)$ lies in $A_{tet\m}$. 
	\end{proof}

	\begin{prop}\label{gm_t-endo}
		For any $t\in T$ we have $\gm_t\in\End N$.
	\end{prop}
	\begin{proof}
		Let $(a,e,u),(b,f,v)\in N$. Then $\gm_t(ab,ef,uv)=\gm_t(a,e,u)\gm_t(b,f,v)$ if and only if
		\begin{align}\label{z_t(uv)=z_t(u)z_t(v)}
		\0(\br(tef))\z_t(uv)=\0(\br(tef))\z_t(u)\z_t(v).
		\end{align}
		Indeed, $\eta_t(a)\eta_t(b)=\eta_t(ab)\in A_{\br(tef)}$, the conjugation by $t$ is an endomorphism of $E(T)$ and the conjugation by $[t]$ is an automorphism of $FG(T)$.
		
		Observe that \cref{z_t(uv)=z_t(u)z_t(v)} is trivial when $\e\in\{u,v\}$. So, we assume that $\e\not\in\{u,v\}$. Notice also that $e\le\nu(u)=\f(\irr u)$.
		
		\textbf{Case 1.} $\irr u\irr v$ is irreducible. Then $\irr{uv}=\irr u\irr v$. By \cref{xi_t-mult-reduction-to-u=[x],xi_t-mult-case-u=[x],z_t(w)=xi_t(irr(w))} we have
		\begin{align*}
		\0(\br(tef))\z_t(uv)&=\0(\br(tef))\xi_t(\irr{uv})=\0(\br(tef))\xi_t(\irr u\irr v)\\
		&=\0(\br(tef))\xi_t(\irr u)\xi_t([\f(\irr u)]\irr v)\\
		&=\0(\br(tef))\xi_t(\irr u)\xi_t(\irr v)=\0(\br(tef))\z_t(u)\z_t(v).
		\end{align*}
		
		\textbf{Case 2.} $\irr u=u'w$ and $\irr v=w\m v'$ for some $u'\in F(T\sqcup T\m)$ and $v'\in F(T\sqcup T\m)^1$, where $w$ is the maximal suffix of $\irr u$ such that $w\m$ is a prefix of $\irr v$. Then $\irr{uv}=u'v'$. In view \cref{xi_t([x]w)xi_t(w-inv)-reduction,xi_t-mult-reduction-to-u=[x]} we have
		\begin{align*}
		\0(\br(tef))\z_t(u)\z_t(v)&=\0(\br(tef))\xi_t(u'w)\xi_t(w\m v')\\
		&=\0(\br(tef))\xi_t(u')\xi_t([\f(u')]w)\xi_t(w\m)\xi_t([\f(w)\m] v')\\
		&=\0(\br(tef))\xi_t(u')\xi_t([\f(w)\m] v'),\\
		\0(\br(tef))\xi_t(u'v')&=\0(\br(tef))\xi_t(u')\xi_t([\f(u')]v').
		\end{align*}
		But $\0(\br(tef))\xi_t([\f(u')]v')=\0(\br(tef))\xi_t([\f(w)\m] v')$ thanks to \cref{xi_t([x]w)-equals-xi_t([y-inv]w)}, as $e\le\f(\irr u)=\f(u')\f(w)$.
		
		\textbf{Case 3.} $\irr u=w$ and $\irr v=w\m v'$ for some $w,v'\in F(T\sqcup T\m)$. Then $\irr{uv}=v'$. Thanks to \cref{xi_t-mult-reduction-to-u=[x],xi_t(w)xi_t(w-inv)-reduction,xi_t-mult-case-u=[x]} we have
		\begin{align*}
		\0(\br(tef))\z_t(u)\z_t(v)&=\0(\br(tef))\xi_t(w)\xi_t(w\m v')\\
		&=\0(\br(tef))\xi_t(w)\xi_t(w\m)\xi_t([\f(w)\m] v')\\
		&=\0(\br(tef))\xi_t([\f(w)\m] v')=\0(\br(tef))\xi_t(v')\\
		&=\0(\br(tef))\xi_t(\irr{uv})=\0(\br(tef))\z_t(uv).
		\end{align*}
		
		\textbf{Case 4.} $\irr u=w$ and $\irr v=w\m$ for some $w\in F(T\sqcup T\m)$. Then $uv=\e$, so $\0(\br(tef))\z_t(uv)=\0(\br(tef))\0(\br(t))=\0(\br(tef))$ by \cref{z_t(w)=xi_t(irr(w))}. Now, using \cref{xi_t(w)xi_t(w-inv)-reduction} we have
		\begin{align*}
		\0(\br(tef))\z_t(u)\z_t(v)=\0(\br(tef))\xi_t(w)\xi_t(w\m)=\0(\br(tef)).
		\end{align*}
	\end{proof}

For any $t\in T$ and $(a,e,w)\in N$ we similarly define
\begin{align}\label{gm-inv_t(a_e_w)=}
\gm\m_t(a,e,w)=(\z_{t\m}(w)\eta_{t\m}(a),t\m et,[t]\m w[t]).
\end{align}

\begin{rem}\label{gm-inv_t-endo}
	For any $t\in T$ we have $\gm\m_t\in\End N$. 
\end{rem}
\begin{proof}
	The proof reduces to $\0(\br(t\m ef))\z_{t\m}(uv)=\0(\br(t\m ef))\z_{t\m}(u)\z_{t\m}(v)$, which is \cref{z_t(uv)=z_t(u)z_t(v)} constituting a part of the proof of \cref{gm_t-endo}. 
\end{proof}

\begin{rem}\label{gm_t(af(e_e))=af(r(te)_e)}
	For any $t\in T$ and $(e,\e)\in E(S)$ we have $\gm_t(\af(e,\e))=\af(\br(te),\e)$ and $\gm\m_t(\af(e,\e))=\af(\bd(et),\e)$.
\end{rem}
\begin{proof}
	Indeed, by \cref{af(e_1)=(0(e)_e_1),gm_t(a_e_w)=(z_t(w)eta_t(a)_tet-inv_[t]w[t]-inv),z_t(w)=xi_t(irr(w))} we have
	\begin{align*}
		\gm_t(\af(e,\e))&=\gm_t(\0(e),e,\e)=(\z_t(\e)\eta_t(\0(e)),tet\m,[t]\e[t]\m)\\
		&=(\0(tet\m),tet\m,\e)=\af(tet\m,\e).
	\end{align*} 
\end{proof}

We will see that $\gm\m_t$ is indeed the inverse of $\gm_t$ in $\mend N$, but we need some more general results that will be used in what follows.

\begin{lem}\label{gm_t_1.gm_t_2=gm_t_1t_2}
	Let $t,u\in T$, $\ve\in\{-1,1\}$, $w\in F(T\sqcup T\m)$ and $e\in E(T)$ such that $e\le\f(w)$. Then
	\begin{align}\label{xi_t(u^ewu^-e)eta_t(xi_u(w))=xi_tu(w)}
		\0(\br(tu^\ve e))\xi_t([u]^\ve w[u]^{-\ve})\eta_t(\xi_{u^\ve}(w))=\0(\br(tu^\ve e))\xi_{tu^\ve}(w).
	\end{align}
\end{lem}
\begin{proof}
	Consider first the case $\ve=-1$.
	
	Assume that $w=[x]\m w'$ for some $w'\in F(T\sqcup T\m)^1$. Then by \cref{xi_t([x]-inv.u),xi_t([x][y]-inv.v),xi_t([x][y]v)}
	\begin{align*}
		\xi_t([u]^\ve w[u]^{-\ve})\eta_t(\xi_{u^\ve}(w))&=c(t,u\m,u)\m c(t,u\m x\m,x)\m\eta_t(c(u\m,x\m,x))\m\\
		&\quad\cdot\xi_t([u\m x\m]w'[u])\eta_t(\xi_{u\m}([x\m]w'))\\
		&=c(t,u\m,u)\m c(t,u\m x\m,x)\m\eta_t(c(u\m,x\m,x))\m c(t,u\m,x\m)\m\\
		&\quad\cdot\xi_t([u\m][x\m]w'[u])\eta_t(\xi_{t\m}([x\m]w'))\\
		&=c(t,u\m x\m,x)\m\eta_t(c(u\m,x\m,x))\m c(t,u\m,x\m)\m\\
		&\quad\cdot\xi_t([u]^\ve[x\m]w'[u]^{-\ve})\eta_t(\xi_{u^\ve}([x\m]w')),\\
		\xi_{tu^\ve}(w)&=c(tu\m,x\m,x)\m\xi_{tu^\ve}([x\m]w'). 
	\end{align*}
	Using the $3$-cocycle identity for $c$ with the quadruple $(t,u\m,x\m,x)$ and \cref{c(x_1...x_(i-1)_e_x_(i+1)...x_n)-triv}, we obtain
	\begin{align*}
		\eta_t(c(u\m,x\m,x))c(tu\m,x\m,x)\m c(t,u\m x\m,x) c(t,u\m,x\m)=\0(\br(tu\m x\m x)).
	\end{align*}
	Since $\eta_t(\xi_{u^\ve}([x\m]w')),\xi_{tu^\ve}([x\m]w')\in A_{\br(tu\m x\m\f(w'))}$ by \cref{xi_t(w)-belongs-to-A_r(tf(w))} and $\br(x\m\f(w'))\le x\m x$, we conclude that \cref{xi_t(u^ewu^-e)eta_t(xi_u(w))=xi_tu(w)} is equivalent to
	\begin{align*}
	\0(\br(tu^\ve e))\xi_t([u]^\ve [x\m]w'[u]^{-\ve})\eta_t(\xi_{u^\ve}([x\m]w'))=\0(\br(tu^\ve e))\xi_{tu^\ve}([x\m]w').
	\end{align*}
	Observe that $\f([x]\m w')=\f([x\m]w')$. Thus, it is enough to consider $w$ of the form $[x]w'$.
	
	We proceed by induction on $l(w)$. If $w=[x]$, then by \cref{xi_t([x]-inv.u),xi_t([x][y]v),xi_t([x]),c(x_1...ex_i...x_n),c(x_1...x_(i-1)_e_x_(i+1)...x_n)-triv} and $e\le x$
	\begin{align*}
	\0(\br(tu^\ve e))\xi_t([u]^\ve w[u]^{-\ve})\eta_t(\xi_{u^\ve}(w))&=\0(\br(tu\m e))c(t,u\m,u)\m c(t,u\m, x)c(t,u\m x, u)\\
	&=c(t,u\m,u)\m c(t,u\m, ex)c(t,u\m xe, u)\\
	&=c(t,u\m,u)\m c(t,u\m, e)c(t,u\m e, u)=\0(\br(tu^\ve e)).\\
	\0(\br(tu^\ve e))\xi_{tu^\ve}(w)&=\0(\br(tu^\ve e))\0(\br(tu^\ve x))=\0(\br(tu^\ve e)).
	\end{align*}

	Let $l(w)>1$.
	
	\textbf{Case 1.} $w=[x][y]w'$ for some $w'\in F(T\sqcup T\m)^1$. Then by \cref{xi_t([x]-inv.u),xi_t([x][y]v)}
	\begin{align*}
	\xi_t([u]^\ve w[u]^{-\ve})\eta_t(\xi_{u^\ve}(w))&=c(t,u\m,u)\m c(t,u\m,x)c(t,u\m x,y)\eta_t(c(u\m,x,y))\\
	&\quad\cdot\xi_t([u\m xy]w'[u])\eta_t(\xi_{u\m}([xy]w'))\\
	&=c(t,u\m,u)\m c(t,u\m,x)c(t,u\m x,y)\eta_t(c(u\m,x,y))c(t,u\m,xy)\m\\
	&\quad\cdot\xi_t([u\m][xy]w'[u])\eta_t(\xi_{u\m}([xy]w'))\\
	&=c(t,u\m,x)c(t,u\m x,y)\eta_t(c(u\m,x,y))c(t,u\m,xy)\m\\
	&\quad\cdot\xi_t([u]^\ve[xy]w'[u]^{-\ve})\eta_t(\xi_{u^\ve}([xy]w')),\\
	\xi_{tu^\ve}(w)&=c(tu\m,x,y)\xi_{tu^\ve}([xy]w').
	\end{align*}
	Using the $3$-cocycle identity for $c$ with the quadruple $(t,u\m,x,y)$, we obtain
	\begin{align*}
	\eta_t(c(u\m,x,y))c(tu\m,x,y)\m c(t,u\m x,y)c(t,u\m,xy)\m c(t,u\m,x)=\0(\br(tu^\ve xy)).
	\end{align*}
	Now, $\eta_t(\xi_{u^\ve}([xy]w')),\xi_{tu^\ve}([xy]w')\in A_{\br(tu^\ve xy\f(w'))}$ and $\br(xy\f(w'))\le\br(xy)$, so \cref{xi_t(u^ewu^-e)eta_t(xi_u(w))=xi_tu(w)} is equivalent to
	\begin{align*}
	\0(\br(tu^\ve e))\xi_t([u]^\ve [xy]w'[u]^{-\ve})\eta_t(\xi_{u^\ve}([xy]w'))=\0(\br(tu^\ve e))\xi_{tu^\ve}([xy]w').
	\end{align*}
	Since $\f([xy]w')=\f(w)\ge e$ and $l([xy]w')<l(w)$, we may use the induction hypothesis.
	
	\textbf{Case 2.} $w=[x][y]\m w'$ for some $w'\in F(T\sqcup T\m)^1$. Then by \cref{xi_t([x]-inv.u),xi_t([x][y]-inv.v),xi_t([x][y]v)}
	\begin{align*}
	\xi_t([u]^\ve w[u]^{-\ve})\eta_t(\xi_{u^\ve}(w))&=c(t,u\m,u)\m c(t,u\m,x)c(t,u\m xy\m,y)\m\eta_t(c(u\m,xy\m,y))\m\\
	&\quad\cdot\xi_t([u\m xy\m]w'[u])\eta_t(\xi_{u\m}([xy\m]w'))\\
	&=c(t,u\m,u)\m c(t,u\m,x)c(t,u\m xy\m,y)\m\eta_t(c(u\m,xy\m,y))\m c(t,u\m,xy\m)\m\\
	&\quad\cdot\xi_t([u\m][xy\m]w'[u])\eta_t(\xi_{u\m}([xy\m]w'))\\
	&=c(t,u\m,x)c(t,u\m xy\m,y)\m\eta_t(c(u\m,xy\m,y))\m c(t,u\m,xy\m)\m\\
	&\quad\cdot\xi_t([u]^\ve[xy\m]w'[u]^{-\ve})\eta_t(\xi_{u^\ve}([xy\m]w')),\\
	\xi_{tu^\ve}(w)&=c(tu\m,xy\m,y)\m \xi_{tu^\ve}([xy\m]w').
	\end{align*}
	Using the $3$-cocycle identity for $c$ with the quadruple $(t,u\m,xy\m,y)$ and \cref{c(x_1...x_ie...x_n)}, we obtain
	\begin{align*}
	\eta_t(c(u\m,xy\m,y))c(tu\m,xy\m,y)\m c(t,u\m xy\m,y)c(t,u\m,x)\m c(t,u\m,xy\m)=\0(\br(tu^\ve xy\m)).
	\end{align*}
	Now, $\eta_t(\xi_{u^\ve}([xy\m]w')),\xi_{tu^\ve}([xy\m]w')\in A_{\br(tu^\ve xy\m\f(w'))}$ and $\br(xy\m\f(w'))\le\br(xy\m)$, so \cref{xi_t(u^ewu^-e)eta_t(xi_u(w))=xi_tu(w)} is equivalent to
	\begin{align*}
	\0(\br(tu^\ve e))\xi_t([u]^\ve [xy\m]w'[u]^{-\ve})\eta_t(\xi_{u^\ve}([xy\m]w'))=\0(\br(tu^\ve e))\xi_{tu^\ve}([xy\m]w').
	\end{align*}
	Since $\f([xy\m]w')=\f(w)\ge e$ and $l([xy\m]w')<l(w)$, we may use the induction hypothesis.
	
    The case $\ve=1$ is similar.
\end{proof}

The analog of \cref{gm_t_1.gm_t_2=gm_t_1t_2} holds for $\z$. 

\begin{lem}\label{gm^ve_t-circ-gm^dl_u}
	Let $t,u\in T$, $\ve\in\{-1,1\}$ and $w\in FG(T)$, such that $e\le\nu(w)$. Then
	\begin{align}\label{z_t^ve(u^dlwu^-dl)eta_t^ve(xi_u^dl(w))=xi_t^veu^dl(w)}
	\0(\br(t u^\ve e))\z_{t}([u]^\ve w[u]^{-\ve})\eta_{t}(\z_{u^\ve}(w))=\0(\br(t u^\ve e))\z_{t u^\ve}(w).
	\end{align}
\end{lem}
\begin{proof}
	We first consider the case $w=\e$. Then by \cref{z_t(w)=xi_t(irr(w))}
	\begin{align*}
		\z_{t}([u]^\ve w[u]^{-\ve})\eta_{t}(\z_{u^\ve}(w))&=\z_{t}([u]^\ve \e[u]^{-\ve})\eta_{t}(\z_{u^\ve}(\e))=\z_{t}(\e)\eta_{t}(\z_{u^\ve}(\e))\\
		&=\0(\br(t))\eta_{t}(\0(\br(u^\ve)))=\0(\br(t))\0(\br(t u^\ve))\\
		&=\0(\br(t u^\ve))=\z_{t u^\ve}(\e).
	\end{align*}
	
	Now let $w\ne\e$.
	
	\textbf{Case 1.} $[u]^\ve \irr w[u]^{-\ve}$ is irreducible. Then $\irr{[u]^\ve w[u]^{-\ve}}=[u]^\ve \irr w[u]^{-\ve}$. We calculate by \cref{z_t(w)=xi_t(irr(w)),xi_t(u^ewu^-e)eta_t(xi_u(w))=xi_tu(w)} and $e\le\nu(w)=\f(\irr w)$
	\begin{align}
		\0(\br(t u^\ve e))\z_{t}([u]^\ve w[u]^{-\ve})\eta_{t}(\z_{u^\ve}(w))
		&=\0(\br(t u^\ve e))\xi_{t}([u]^\ve \irr w[u]^{-\ve})\eta_{t}(\xi_{u^\ve}(\irr w))\notag\\
		&=\0(\br(t u^\ve e))\xi_{t u^\ve}(\irr w)\notag\\
		&=\0(\br(t u^\ve e))\z_{t u^\ve}(w).\label{z_t^ve(u^dlwu^-dl)eta_t^ve(xi_u^dl(w))-irr-w}
	\end{align}
	
	\textbf{Case 2.} $\irr w=[u]^{-\ve}w'$, where $w'$ is not empty and the last letter of $w'$ is different from $[u]^\ve$. Then $\irr{[u]^\ve w[u]^{-\ve}}=w'[u]^{-\ve}$ and
	\begin{align*}
		\xi_{t}([u]^\ve \irr w[u]^{-\ve})=\xi_{t}([u]^\ve[u]^{-\ve}w'[u]^{-\ve}).
	\end{align*}
	If $\ve=1$, then by \cref{c(x_1...x_(i-1)_e_x_(i+1)...x_n)-triv,xi_t([e]v)=0(r(te))xi_t(v),xi_t([x][y]-inv.v)}
	\begin{align*}
	\xi_{t}([u]^\ve \irr w[u]^{-\ve})&=\xi_{t}([u][u]\m w'[u]\m)=c(t,uu\m,u)\m\xi_{t}([uu\m]w'[u]\m)\\
	&=\0(\br(t u^\ve))\xi_{t}(w'[u]^{-\ve})=\0(\br(t u^\ve))\z_{t}([u]^\ve w[u]^{-\ve}),
	\end{align*}
	so the proof of \cref{z_t^ve(u^dlwu^-dl)eta_t^ve(xi_u^dl(w))-irr-w} remains valid in this case. If $\ve=-1$, then by \cref{xi_t([e]v)=0(r(te))xi_t(v),xi_t([x]-inv.u),xi_t([x][y]v)}
	\begin{align*}
	\xi_{t}([u]^\ve \irr w[u]^{-\ve})&=\xi_{t}([u]\m[u]w'[u])=c(t,u\m,u)\m c(t,u\m,u)\xi_{t}([u\m u]w'[u])\\
	&=\0(\br(t u^\ve))\xi_{t}(w'[u]^{-\ve})=\0(\br(t u^\ve))\z_{t}([u]^\ve w[u]^{-\ve}),
	\end{align*}
	and the proof of \cref{z_t^ve(u^dlwu^-dl)eta_t^ve(xi_u^dl(w))-irr-w} again works.
	
	\textbf{Case 3.} $\irr w=w'[u]^\ve$, where $w'$ is not empty and the first letter of $w'$ is different from $[u]^{-\ve}$. Then $\irr{[u]^\ve w[u]^{-\ve}}=[u]^\ve w'$. By \cref{xi_t-mult-reduction-to-u=[x]}
	\begin{align*}
	\xi_{t}([u]^\ve \irr w[u]^{-\ve})&=\xi_{t}([u]^\ve w'[u]^\ve[u]^{-\ve})=\xi_{t}([u]^\ve w')\xi_{t}([u^\ve \f(w')][u]^\ve[u]^{-\ve})\\
	&=\z_{t}([u]^\ve w[u]^{-\ve})\xi_{t}([u^\ve \f(w')][u]^\ve[u]^{-\ve}).
	\end{align*}
	If $\ve=1$, then by \cref{xi_t([x][y]-inv.v),xi_t([x][y]v),xi_t([x]),c(x_1...x_ie...x_n)}
	\begin{align*}
	\xi_{t}([u^\ve \f(w')][u]^\ve[u]^{-\ve})&=\xi_{t}([u\f(w')][u][u]\m)\\
	&=c(t,u\f(w'),u)c(t,u\f(w')uu\m,u)\m\xi_t([u\f(w')uu\m])\\
	&=\0(\br(t u\f(w')u))=\0(\br(t u^\ve\f(\irr w))),
	\end{align*}
	so the proof of \cref{z_t^ve(u^dlwu^-dl)eta_t^ve(xi_u^dl(w))-irr-w} is applicable. 
	If $\ve=-1$, then by \cref{xi_t([x][y]-inv.v),xi_t([x][y]v),xi_t([x])}
	\begin{align*}
	\xi_{t}([u^\ve \f(w')][u]^\ve[u]^{-\ve})&=\xi_{t}([u\m\f(w')][u]\m[u])\\
	&=c(t,u\m \f(w')u\m,u)\m c(t,u\m \f(w')u\m,u)\xi_{t}([u\m\f(w')u\m u])\\
	&=\0(\br(t u\m\f(w')u\m))=\0(\br(t u^\ve\f(\irr w))),
	\end{align*}
	and the proof of \cref{z_t^ve(u^dlwu^-dl)eta_t^ve(xi_u^dl(w))-irr-w} again works.
	
	\textbf{Case 4.} $\irr w=[u]^{-\ve} w'[u]^\ve$, where $w'$ is not empty. Then $\irr{[u]^\ve w[u]^{-\ve}}=w'$. As in Cases 2 and 3, we have $\xi_{t}([u]^\ve \irr w[u]^{-\ve})=\0(\br(t u^\ve\f(\irr w)))\z_{t}([u]^\ve w[u]^{-\ve})$.

	\textbf{Case 5.} $\irr w=[u]^{-\ve}$. Then $\z_{t}([u]^\ve w[u]^{-\ve})=\0(\br(t u^{-\ve}))$, $\z_{u^\ve}(w)=\0(\br(u^\ve))$ and $\z_{t u^\ve}(w)=\0(\br(t u^\ve))$. Since $e\le u^{-\ve}$ (whence $e\le u^\ve$ too), the result of \cref{z_t^ve(u^dlwu^-dl)eta_t^ve(xi_u^dl(w))-irr-w} follows. 
	
	\textbf{Case 6.} $\irr w=[u]^\ve$. This case is similar to the previous one.
\end{proof}

\begin{lem}\label{0(e)xi_t=xi_et}
	Let $t\in T$ and $w\in F(T\sqcup T\m)$. Then for any $e\in E(T)$
	\begin{align}\label{0(e)xi_e(w)=xi_et(w)}
		\0(e)\xi_t(w)=\xi_{et}(w).
	\end{align}
\end{lem}
\begin{proof}
	Assume that $w=[x]\m w'$ for some $w'\in F(T\sqcup T\m)^1$. Then by \cref{xi_t([x]-inv.u),c(x_1...ex_i...x_n)}
	\begin{align*}
		\0(e)\xi_t(w)&=\0(e)c(t,x\m,x)\m\xi_t([x\m]w')=c(et,x\m,x)\m\0(e)\xi_t([x\m]w'),\\
		\xi_{et}(w)&=c(et,x\m,x)\m\xi_{et}([x\m]w').
	\end{align*}
	Hence it is enough to consider $w$ of the form $[x]w'$.
	
	If $w=[x]$, then \cref{0(e)xi_e(w)=xi_et(w)} follows from \cref{xi_t([x])}, so let $l(w)>1$.
	
	\textbf{Case 1.} $w=[x][y]w'$ for some $w'\in F(T\sqcup T\m)^1$. Then by \cref{xi_t([x][y]v),c(x_1...ex_i...x_n)}
	\begin{align*}
	\0(e)\xi_t(w)&=\0(e)c(t,x,y)\xi_t([xy]w')=c(et,x,y)\0(e)\xi_t([xy]w'),\\
	\xi_{et}(w)&=c(et,x,y)\xi_{et}([xy]w').
	\end{align*}
	Since $l([xy]w')<l(w)$, we may apply the induction hypothesis.
	
	\textbf{Case 2.} $w=[x][y]\m w'$ for some $w'\in F(T\sqcup T\m)^1$. This case is similar to the previous one.
\end{proof}

\begin{lem}\label{xi_e-trivial}
	Let $e\in E(T)$ and $w\in F(T\sqcup T\m)$. Then 
	\begin{align}\label{xi_e(w)=0(er(f(w)))}
		\xi_e(w)=\0(e\br(\f(w))).
	\end{align}
\end{lem}
\begin{proof}
	Assume that $w=[x]\m w'$ for some $w'\in F(T\sqcup T\m)^1$. Then by \cref{xi_t([x]-inv.u),c(x_1...x_(i-1)_e_x_(i+1)...x_n)-triv,xi_t(w)-belongs-to-A_r(tf(w))}
	\begin{align*}
	\xi_e(w)&=c(e,x\m,x)\m\xi_e([x\m]w')=\0(e\bd(x))\xi_e([x\m]w')=\xi_e([x\m]w').
	\end{align*}
	Since $\f([x\m]w')=\f(w)$, it is enough to consider $w$ of the form $[x]w'$.
	
	If $w=[x]$, then \cref{xi_e(w)=0(er(f(w)))} follows from \cref{xi_t([x])}, so let $l(w)>1$.
	
	\textbf{Case 1.} $w=[x][y]w'$ for some $w'\in F(T\sqcup T\m)^1$. Then by \cref{xi_t([x][y]v),c(x_1...x_(i-1)_e_x_(i+1)...x_n)-triv,xi_t(w)-belongs-to-A_r(tf(w))}
	\begin{align*}
	\xi_e(w)&=c(e,x,y)\xi_e([xy]w')=\0(e\br(xy))\xi_e([xy]w')=\xi_e([xy]w').
	\end{align*}
	Since $l([xy]w')<l(w)$ and $\f([xy]w')=\f(w)$, we may apply the induction hypothesis.
	
	\textbf{Case 2.} $w=[x][y]\m w'$ for some $w'\in F(T\sqcup T\m)^1$. This case is similar to the previous one.
\end{proof}

\begin{cor}\label{0(e)z_t-trivial}
	Let $e\in E(T)$ and $w\in FG(T)$. Then we have $\0(e)\z_t(w)=\z_{et}(w)$ and $\z_e(w)=\0(e\br(\nu(w)))$.
\end{cor}
\begin{proof}
	Indeed, for $w=\e$ this holds by \cref{z_t(w)=xi_t(irr(w)),nu(w)=f(irr(w))}, and for $w\ne\e$ by \cref{xi_e-trivial,nu(w)=f(irr(w)),0(e)xi_t=xi_et}.
\end{proof}

\begin{prop}\label{gm_t-rel-inv}
	For any $t\in T$ the endomorphism $\gm_t\in\End N$ is relatively invertible whose inverse is $\gm\m_t$.
\end{prop}
\begin{proof}
	Let $n=(a,e,w)\in N$. By \cref{gm^ve_t-circ-gm^dl_u,0(e)z_t-trivial,gm_t(a_e_w)=(z_t(w)eta_t(a)_tet-inv_[t]w[t]-inv),gm-inv_t(a_e_w)=} and $e\le\nu(w)$
	\begin{align}
		(\gm_t\circ\gm\m_t)(n)&=\gm_t(\z_{t\m}(w)\eta_{t\m}(a),t\m et,[t]\m w[t])\notag\\
		&=(\z_t([t]\m w[t])\eta_t(\z_{t\m}(w)\eta_{t\m}(a)),tt\m ett\m,[t][t]\m w[t][t]\m)\notag\\
		&=(\z_{tt\m}(w)\0(\br(t))a,\br(t) e,w)=(\0(\br(\nu(w)))\0(\br(t))a,\br(t) e,w)\notag\\
		&=(\0(\br(t))a,\br(t) e,w)=(\0(\br(t)),\br(t),\e)n=\af(\br(t),\e)n.\label{gm_t-circ-gm-inv_t}
	\end{align}
 Similarly
	\begin{align}\label{gm-inv_t-circ-gm_t}
		(\gm\m_t\circ\gm_t)(n)=\af(\bd(t),\e)n.
	\end{align}
	 
	Finally, $\gm_t(\af(\bd(t),\e))=\af(\br(t),\e)$ by \cref{gm_t(af(e_e))=af(r(te)_e)}. It is clear from \cref{gm_t(a_e_w)=(z_t(w)eta_t(a)_tet-inv_[t]w[t]-inv),gm-inv_t(a_e_w)=} that $\af(\br(t),\e)\gm_t(n)=\gm_t(n)$ and $\af(\bd(t),\e)\gm\m_t(n)=\gm\m_t(n)$ for any $n\in N$, which completes the proof.
\end{proof}

We are ready to define a map $\lb:S\to\mend N$. Let $s=(t,u)\in S$ and $n=(a,e,w)\in N$. We put
\begin{align}\label{lb_(t_u)(a_e_w)=}
	\lb_s(n)=
	\begin{cases}
		\af(\br(s))(\gm_{t_1}^{\ve_1}\circ\dots\circ \gm_{t_k}^{\ve_k})(n), & u\ne\e\mbox{ and }\irr u=[t_1]^{\ve_1}\dots[t_k]^{\ve_k},\\
		\af(s)n, & u=\e.
	\end{cases}
\end{align}
It is easy to see that the definition makes sense when $u = \e$, because in this case $t\in E(T)$, so $s\in E(S)$.

\begin{prop}\label{lb-homo}
	The map $\lb:S\to\mend N$ is a homomorphism.
\end{prop}
\begin{proof}
	Let $(t,u),(p,v)\in S$ and $n=(a,e,w)\in N$. We shall prove that
	\begin{align}\label{lb_(tp_uv)=lb_(t_u)lb_(p_v)}
		\lb_{(tp,uv)}(n)=(\lb_{(t,u)}\circ\lb_{(p,v)})(n).
	\end{align}
	If $u=\e$, then $t\in E(T)$ and $\af(\br(tp),\e)=\af(t\br(p),\e)=\af(t,\e)\af(\br(p,v))$, so \cref{lb_(tp_uv)=lb_(t_u)lb_(p_v)} follows directly from \cref{lb_(t_u)(a_e_w)=}. If $v=\e$ and $\irr u=[t_1]^{\ve_1}\dots[t_k]^{\ve_k}$, then using \cref{gm_t(af(e_e))=af(r(te)_e),lb_(t_u)(a_e_w)=} and $t\le\nu(u)=t_1^{\ve_1}\dots t_k^{\ve_k}$ we obtain
	\begin{align*}
		(\lb_{(t,u)}\circ\lb_{(p,v)})(n)&=\af(\br(s))(\gm_{t_1}^{\ve_1}\circ\dots\circ \gm_{t_k}^{\ve_k})(\af(p,\e)n)\\
		&=\af(\br(t),\e)\af(\br(t_1^{\ve_1}\dots t_k^{\ve_k}p),\e)(\gm_{t_1}^{\ve_1}\circ\dots\circ \gm_{t_k}^{\ve_k})(n)\\
		&=\af(\br(t)\br(t_1^{\ve_1}\dots t_k^{\ve_k}p),\e)(\gm_{t_1}^{\ve_1}\circ\dots\circ \gm_{t_k}^{\ve_k})(n)\\
		&=\af(\br(tp),\e)(\gm_{t_1}^{\ve_1}\circ\dots\circ \gm_{t_k}^{\ve_k})(n)\\
		&=\af(\br(tp,uv))(\gm_{t_1}^{\ve_1}\circ\dots\circ \gm_{t_k}^{\ve_k})(n)=\lb_{(tp,uv)}(n).
	\end{align*}
	Now, assume that $\e\not\in\{u,v\}$.
	
	\textbf{Case 1.} $\irr u\irr v$ is irreducible. Then $\irr{uv}=\irr u\irr v$. Let $\irr u=[t_1]^{\ve_1}\dots[t_k]^{\ve_k}$ and $\irr v=[p_1]^{\dl_1}\dots[p_l]^{\dl_l}$, so that $\irr{uv}=[t_1]^{\ve_1}\dots[t_k]^{\ve_k}[p_1]^{\dl_1}\dots[p_l]^{\dl_l}$. We calculate using \cref{gm_t(af(e_e))=af(r(te)_e),lb_(t_u)(a_e_w)=} and $t\le\nu(u)$
	\begin{align*}
		(\lb_{(t,u)}\circ\lb_{(p,v)})(n)&=\af(\br(t,u))(\gm_{t_1}^{\ve_1}\circ\dots\circ \gm_{t_k}^{\ve_k})(\af(\br(p,v))(\gm_{p_1}^{\dl_1}\circ\dots\circ \gm_{p_l}^{\dl_l})(n))\\
		&=\af(\br(t),\e)(\gm_{t_1}^{\ve_1}\circ\dots\circ \gm_{t_k}^{\ve_k})(\af(\br(p),\e)(\gm_{p_1}^{\dl_1}\circ\dots\circ \gm_{p_l}^{\dl_l})(n))\\
		&=\af(\br(t),\e)\af(\br(t_1^{\ve_1}\dots t_k^{\ve_k}p),\e)(\gm_{t_1}^{\ve_1}\circ\dots\circ \gm_{t_k}^{\ve_k}\circ\gm_{p_1}^{\dl_1}\circ\dots\circ \gm_{p_l}^{\dl_l})(n)\\
		&=\af(\br(t)\br(t_1^{\ve_1}\dots t_k^{\ve_k}p),\e)(\gm_{t_1}^{\ve_1}\circ\dots\circ \gm_{t_k}^{\ve_k}\circ\gm_{p_1}^{\dl_1}\circ\dots\circ \gm_{p_l}^{\dl_l})(n)\\
		&=\af(\br(tp),\e)(\gm_{t_1}^{\ve_1}\circ\dots\circ \gm_{t_k}^{\ve_k}\circ\gm_{p_1}^{\dl_1}\circ\dots\circ \gm_{p_l}^{\dl_l})(n)\\
		&=\af(\br(tp,uv))(\gm_{t_1}^{\ve_1}\circ\dots\circ \gm_{t_k}^{\ve_k}\circ\gm_{p_1}^{\dl_1}\circ\dots\circ \gm_{p_l}^{\dl_l})(n)=\lb_{(tp,uv)}(n).
	\end{align*}
	
	\textbf{Case 2.} $\irr u=u'w$ and $\irr v=w\m v'$ for some $u'\in F(T\sqcup T\m)$ and $v'\in F(T\sqcup T\m)$, where $w$ is the maximal suffix of $\irr u$ such that $w\m$ is a prefix of $\irr v$. Then $\irr{uv}=u'v'$. Let $u'=[t_1]^{\ve_1}\dots[t_k]^{\ve_k}$, $v'=[p_1]^{\dl_1}\dots[p_l]^{\dl_l}$ and $w=[q_1]^{\s_1}\dots[q_m]^{\s_m}$, so that $\irr u=[t_1]^{\ve_1}\dots[t_k]^{\ve_k}[q_1]^{\s_1}\dots[q_m]^{\s_m}$, $\irr v=[q_m]^{-\s_m}\dots[q_1]^{-\s_1}[p_1]^{\dl_1}\dots[p_l]^{\dl_l}$ and $\irr{uv}=[t_1]^{\ve_1}\dots[t_k]^{\ve_k}[p_1]^{\dl_1}\dots[p_l]^{\dl_l}$. We have by \cref{gm_t(af(e_e))=af(r(te)_e),lb_(t_u)(a_e_w)=,gm_t-circ-gm-inv_t,gm-inv_t-circ-gm_t} and $t\le\nu(u)$
	\begin{align*}
		(\lb_{(t,u)}\circ\lb_{(p,v)})(n)&=\af(\br(t,u))(\gm_{t_1}^{\ve_1}\circ\dots\circ \gm_{t_k}^{\ve_k}\circ \gm_{q_1}^{\s_1}\circ \dots \circ \gm_{q_m}^{\s_m})\\
		&\quad(\af(\br(p,v))(\gm_{q_m}^{-\s_m}\circ \dots \circ \gm_{q_1}^{-\s_1}\circ\gm_{p_1}^{\dl_1}\circ\dots\circ \gm_{p_l}^{\dl_l})(n))\\
		&=\af(\br(t),\e)(\gm_{t_1}^{\ve_1}\circ\dots\circ \gm_{t_k}^{\ve_k}\circ \gm_{q_1}^{\s_1}\circ \dots \circ \gm_{q_{m-1}}^{\s_{m-1}})\\
		&\quad(\af(\br(q_m^{\s_m}p),\e)(\gm_{q_m}^{\s_m}\circ\gm_{q_m}^{-\s_m}\circ\gm_{q_{m-1}}^{-\s_{m-1}}\circ \dots \circ \gm_{q_1}^{-\s_1}\circ\gm_{p_1}^{\dl_1}\circ\dots\circ \gm_{p_l}^{\dl_l})(n))\\
		&=\af(\br(t),\e)(\gm_{t_1}^{\ve_1}\circ\dots\circ \gm_{t_k}^{\ve_k}\circ \gm_{q_1}^{\s_1}\circ \dots \circ \gm_{q_{m-1}}^{\s_{m-1}})\\
		&\quad(\af(\br(q_m^{\s_m}p)\br(q_m^{\s_m})),\e)(\gm_{q_{m-1}}^{-\s_{m-1}}\circ \dots \circ \gm_{q_1}^{-\s_1}\circ\gm_{p_1}^{\dl_1}\circ\dots\circ \gm_{p_l}^{\dl_l})(n))\\
		&=\af(\br(t),\e)(\gm_{t_1}^{\ve_1}\circ\dots\circ \gm_{t_k}^{\ve_k}\circ \gm_{q_1}^{\s_1}\circ \dots \circ \gm_{q_{m-1}}^{\s_{m-1}})\\
		&\quad(\af(\br(q_m^{\s_m}p),\e)(\gm_{q_{m-1}}^{-\s_{m-1}}\circ \dots \circ \gm_{q_1}^{-\s_1}\circ\gm_{p_1}^{\dl_1}\circ\dots\circ \gm_{p_l}^{\dl_l})(n))\\
		&\dots\\
		&=\af(\br(t),\e)(\gm_{t_1}^{\ve_1}\circ\dots\circ \gm_{t_k}^{\ve_k})(\af(\br(q_1^{\s_1}\dots q_m^{\s_m}p),\e)(\gm_{p_1}^{\dl_1}\circ\dots\circ \gm_{p_l}^{\dl_l})(n))\\
		&=\af(\br(t),\e)(\af(\br(t_1^{\ve_1}\dots t_k^{\ve_k} q_1^{\s_1}\dots q_m^{\s_m}p),\e)(\gm_{t_1}^{\ve_1}\circ\dots\circ \gm_{t_k}^{\ve_k}\circ\gm_{p_1}^{\dl_1}\circ\dots\circ \gm_{p_l}^{\dl_l})(n)\\
		&=\af(\br(tp),\e)(\gm_{t_1}^{\ve_1}\circ\dots\circ \gm_{t_k}^{\ve_k}\circ\gm_{p_1}^{\dl_1}\circ\dots\circ \gm_{p_l}^{\dl_l})(n)=\lb_{(tp,uv)}(n).
	\end{align*}
	
	The remaining $3$ cases $\irr u=u'w$ and $\irr v=w\m$; $\irr u=w$ and $\irr v=w\m v'$; $\irr u=w$ and $\irr v=w\m$ are similar to Case 2.
\end{proof}

We need a more convenient form of the endomorphism $\lb_{(t,w)}$.

\begin{lem}\label{lb_(t_u)-explicit}
	Let $s=(t,u)\in S$ and $n=(a,e,w)\in N$. Assume that $u\ne\e$. Then
	\begin{align}\label{lb_(t_u)(a_e_w)=triple}
	\lb_s(n)=(\z_t(w)\eta_t(a),tet\m,uwu\m).
	\end{align}
\end{lem}
\begin{proof}
	Since $t\le\nu(u)$, then $t=\br(t)\nu(u)$ and $\eta_t(a)=\eta_{\br(t)}(\eta_{\nu(u)}(a))=\0(\br(t))\eta_{\nu(u)}(a)$, so that by \cref{0(e)z_t-trivial}
	\begin{align*}
		(\z_t(w)\eta_t(a),tet\m,uwu\m)=\af(\br(s))(\z_{\nu(u)}(w)\eta_{\nu(u)}(a),\nu(u)e\nu(u)\m,uwu\m).
	\end{align*} 
	Let $\irr u=[t_1]^{\ve_1}\dots[t_k]^{\ve_k}$. According to \labelcref{lb_(t_u)(a_e_w)=}, it suffices to prove that
	\begin{align*}
		(\gm_{t_1}^{\ve_1}\circ\dots\circ \gm_{t_k}^{\ve_k})(n)=(\z_{\nu(u)}(w)\eta_{\nu(u)}(a),\nu(u)e\nu(u)\m,uwu\m).
	\end{align*}
	The proof will be by induction on $k$. If $k=1$, this follows directly from \labelcref{gm-inv_t(a_e_w)=,gm_t(a_e_w)=(z_t(w)eta_t(a)_tet-inv_[t]w[t]-inv),nu(w)=f(irr(w))}.
	
	Let $k>1$. Then using \cref{gm-inv_t(a_e_w)=,gm_t(a_e_w)=(z_t(w)eta_t(a)_tet-inv_[t]w[t]-inv),gm^ve_t-circ-gm^dl_u} and the induction hypothesis
	\begin{align*}
		(\gm_{t_1}^{\ve_1}\circ\dots\circ \gm_{t_k}^{\ve_k})(n)&=(\gm_{t_1}^{\ve_1}\circ\dots\circ \gm_{t_{k-1}}^{\ve_{k-1}})(\z_{t_k^{\ve_k}}(w)\eta_{t_k^{\ve_k}}(a),t_k^{\ve_k}et_k^{-\ve_k},[t_k]^{\ve_k}w[t_k]^{-\ve_k})\\
		&=(\z_{t_1^{\ve_1}\dots t_{k-1}^{\ve_{k-1}}}([t_k]^{\ve_k}w[t_k]^{-\ve_k})\eta_{t_1^{\ve_1}\dots t_{k-1}^{\ve_{k-1}}}(\z_{t_k^{\ve_k}}(w))\eta_{t_1^{\ve_1}\dots t_{k-1}^{\ve_{k-1}}}(\eta_{t_k^{\ve_k}}(a)),\\
		&\quad t_1^{\ve_1}\dots t_k^{\ve_k}et_k^{-\ve_k}\dots t_1^{-\ve_1},[t_1]^{\ve_1}\dots[t_k]^{\ve_k}w[t_k]^{-\ve_k}\dots[t_1]^{-\ve_1})\\
		&=(\z_{\nu(u)}(w)\eta_{\nu(u)}(a),\nu(u)e\nu(u)\m,uwu\m).
	\end{align*}
\end{proof}

\begin{prop}\label{(af_lb_bt)-crossed-module}
	The triple $(\af,\lb,\bt)$ is a crossed $S$-module structure on $N$.
\end{prop}
\begin{proof}
	In view of \cref{properties-of-i-af-bt,lb-homo} we only need to check \cref{CM1,CM2,CM3,CM4}.
	
	Let $(e,\e)\in E(S)$ and $n = (a,f,w)\in N$. Then $\lb_{(e,\e)}(n)=\af(e,\e)n$ directly from \cref{lb_(t_u)(a_e_w)=}, which is \cref{CM1}.
	
	Let $(t,u)\in S$ and $(e,\e)\in E(S)$. If $u=\e$, then $t\in E(T)$ and $\lb_{(t,u)}(\af(e,\e))=\af(t,\e)\af(e,\e)=\af(te,\e)=\af(tet\m,\e)$. Otherwise, by \cref{lb_(t_u)(a_e_w)=triple,z_t(w)=xi_t(irr(w))} 
	\begin{align*}
		\lb_{(t,u)}(\af(e,\e))&=(\z_t(\e)\eta_t(\0(e)),tet\m,u\e u\m)=(\0(\br(t))\0(tet\m),tet\m,\e)\\
		&=(\0(tet\m),tet\m,\e)=\af(tet\m,\e)=\af((t,u)(e,\e)(t,u)\m),
	\end{align*}
	so \cref{CM2} is proved.
	
	Take $n=(a,e,w)$ and $n'=(a',e',w')\in N$. Then $\bt(n)=(e,w)$ thanks to \cref{bt(a_e_w)=(e_w)}. If $w=\e$, then 
	\begin{align*}
		\lb_{\bt(n)}(n')=\af(e,\e)n'=(\0(e)a',ee', w')=(aa\m a',ee',w')=nn'n\m.
	\end{align*}
	Otherwise, by \cref{0(e)z_t-trivial,lb_(t_u)(a_e_w)=triple} and $e'\le\nu(w')$
	\begin{align*}
		\lb_{(e,w)}(a',e',w')&=(\z_e(w')\eta_e(a'),ee',ww'w\m)=(\0(e\br(\nu(w')))a',ee',ww'w\m)\\
		&=(\0(e)a',ee',ww'w\m)=(aa\m a',ee',ww'w\m)=nn'n\m,
	\end{align*}
	proving \cref{CM3}.
	
	Let $s=(t,u)\in S$ and $n=(a,e,w)\in N$. If $u=\e$, then $t\in E(T)$, so by \cref{lb_(t_u)(a_e_w)=}
	\begin{align*}
		\bt(\lb_s(n))=(te,w)=(t,u)(e,w)(t,u)\m=s\bt(n)s\m.
	\end{align*}
	Otherwise, we have by \cref{lb_(t_u)(a_e_w)=triple,bt(a_e_w)=(e_w),af(e_1)=(0(e)_e_1)}
	\begin{align*}
		\bt(\lb_s(n))=(tet\m,uwu\m)=(t,u)(e,w)(t,u)\m=s\bt(n)s\m,
	\end{align*}
	which gives \cref{CM4}.
\end{proof}

\begin{prop}\label{from-c-to-crossed-mod-ext}
	The sequence $A \xrightarrow{i} N \xrightarrow{\beta} S \xrightarrow{\pi} T$ is a crossed module extension of the $T$-module $A$ by $T$.
\end{prop}
\begin{proof}
	Thanks to \cref{(af_lb_bt)-crossed-module,properties-of-i-af-bt} it only remains to show that the induced $T$-module structure on $A$ coincides with $(\0,\eta)$. Let $s=(t,u)\in S$ and $a\in A$. If $u=\e$, then $t\in E(T)$. Using \cref{i(a)=(a_0(aa-inv)_1),lb_(t_u)(a_e_w)=,z_t(w)=xi_t(irr(w))} we have
	\begin{align*}
	(i\m\circ\lb_s\circ i)(a)&=(i\m\circ\lb_s)(a,\0\m(\br(a)),\e)=i\m(\af(t,\e)(a,\0\m(\br(a)),\e))\\
	&=i\m((\0(t),t,\e)(a,\0\m(\br(a)),\e))=i\m(\0(t)a,t\0\m(\br(a)),\e)\\
	&=\0(t)a=\eta_t(a)=\eta_{\pi(s)}(a).
	\end{align*}
	If $u\ne\e$, then by \cref{i(a)=(a_0(aa-inv)_1),lb_(t_u)(a_e_w)=triple,z_t(w)=xi_t(irr(w))} and $t\le\nu(u)$
	\begin{align*}
	(i\m\circ\lb_s\circ i)(a)&=(i\m\circ\lb_s)(a,\0\m(\br(a)),\e)\\
	&=i\m(\z_t(\e)\eta_t(a),t\0\m(\br(a))t\m,u\e u\m)\\
	&=i\m(\0(\br(t))\eta_t(a),t\0\m(\br(a))t\m,\e)\\
	&=i\m(\eta_t(a),t\0\m(\br(a))t\m,\e)=\eta_t(a)=\eta_{\pi(s)}(a).
	\end{align*}
	
	Now take $e\in E(T)$. Then by \cref{af(e_1)=(0(e)_e_1)}
	\begin{align*}
	(i\m\circ\af\circ(\pi|_{E(S)})^{-1})(e)=(i\m\circ\af)(e,\e)=i\m(\0(e),e,\e)=\0(e).
	\end{align*}
	Thus, the $T$-module structure defined in \cref{eta_t(a)=i^(-1)(lb_s(i(a))),0=i^(-1).af.pi^(-1)} coincides with $(\0,\eta)$.
\end{proof}

\subsection{From cohomologous \texorpdfstring{$c,c'\in Z^3_\le(T^1,A^1)$}{c, c' in Z³<(T¹,A¹)} to equivalent crossed module extensions}

Given two cohomologous $c,c'\in Z^3_\le(T^1,A^1)$, we are going to prove that $c$ and $c'$ induce equivalent crossed module extensions of $A$ by $T$. The crossed module extension $A \xrightarrow{i'} N' \xrightarrow{\beta'} S' \xrightarrow{\pi'} T$ induced by $c'$ coincides with $A \xrightarrow{i} N \xrightarrow{\beta} S \xrightarrow{\pi} T$ induced by $c$ as a sequence of inverse semigroups. It also has $\af'=\af$. The only difference is in the $\lb':S'\to\mend N'$ since its definition involves $c'$. We thus define $\f_2=\id:S\to S'$, and the definition of $\f_1:N\to N'$ will be given below. To this end, we introduce the auxiliary map analogous to $\xi_t$ we have associated with $c\in Z^3_\le(T^1,A^1)$ in \cref{xi_t([x]-inv.u),xi_t([x][y]v),xi_t([x][y]-inv.v),xi_t([x])}. Since we will eventually apply this construction not only here, but also in \cref{ext->c->ext-sec}, we will consider a more general setting.

Let $A$ be a semilattice of (not necessarily abelian) groups and $\0:E(T)\to E(A)$ an isomorphism. Let $d:T^2\to A$ be such that $d(x,y)\in A_{\0(\br(xy))}$ for all $x,y\in T$. We assume that $d$ satisfies \cref{c(x_1...x_(i-1)_e_x_(i+1)...x_n)-triv,c(x_1...ex_i...x_n),c(x_1...x_ie...x_n)}. We now define $\tau=\tau_d:F(T\sqcup T\m)\to A$ as follows.

Given $w\in F(T\sqcup T\m)$, if $w=[x]\m u$ for some $x\in T$ and $u\in F(T\sqcup T\m)^1$, then
\begin{align}\label{tau([x]-inv.u)}
\tau(w):=d(x\m,x)\m\tau([x\m]u).
\end{align}
The rest of the definition will be given by induction on $l(w)$.

\textit{Base of induction.} If $w=[x]$ for some $x\in T$, then
\begin{align}\label{tau([x])}
\tau(w):=\0(\br(x)).
\end{align}

\textit{Inductive step.} $l(w)>1$ and $w=[x]u$ for some $x\in T$ and $u\in F(T\sqcup T\m)$. 

\textbf{Case 1.} If $w=[x][y]v$ for some $x,y\in T$ and $v\in F(T\sqcup T\m)^1$, then 
\begin{align}\label{tau([x][y]v)}
\tau(w):=d(x,y)\tau([xy]v).		
\end{align}
Since $l([xy]v)<l(w)$, we may use the inductive step.  

\textbf{Case 2.} If $w=[x][y]\m v$ for some $x,y\in T$ and $v\in F(T\sqcup T\m)^1$, then 
\begin{align}\label{tau([x][y]-inv.v)}
\tau(w):=d(xy\m,y)\m\tau([xy\m]v).
\end{align}
Since $l([xy\m]v)<l(w)$, we may use the inductive step. 

We list the properties of $\tau$ whose proofs are totally analogous to the proofs of \cref{xi_t(w)-belongs-to-A_r(tf(w)),xi_t-mult-reduction-to-u=[x],xi_t-mult-case-u=[x],xi_t([x]w)xi_t(w-inv)-reduction,xi_t(w)xi_t(w-inv)-reduction,xi_t([x]w)-equals-xi_t([y-inv]w)}, where we do not use the $3$-cocycle identity for $c$ and the commutativity of $A$ (see \cref{rem-didnt-use-3coc-commut}).

\begin{lem}\label{properties-of-tau}
	The following properties hold.
	\begin{enumerate}
		\item For any $w\in F(T\sqcup T\m)$ we have $\tau(w)\in A_{\br(\f(w))}$.
		\item For any $u\in F(T\sqcup T\m)$ and $v\in F(T\sqcup T\m)^1$ we have  $\tau(uv)=\tau(u)\tau([\f(u)]v)$.
		\item Let $x\in T$, $v\in F(T\sqcup T\m)$ and $e\in E(T)$ such that $e\le x$. Then $\0(e)\tau([x] v)=\0(e)\tau(v)$.
		\item Let $x\in T$, $w\in F(T\sqcup T\m)$ and $e\in E(T)$ such that $e\le x\f(w)$. Then $\0(e)\tau([x]w)\tau(w\m)=\0(e)$.
		\item Let $w\in F(T\sqcup T\m)$ and $e\in E(T)$ such that $e\le \f(w)$. Then $\0(e)\tau(w)\tau(w\m)=\0(e)$.
		\item Let $x,y\in T$, $w\in F(T\sqcup T\m)^1$ and $e\in E(T)$ such that $e\le xy$. Then $\0(e)\tau([x]w)=\0(e)\tau([y\m]w)$.
	\end{enumerate}
\end{lem}

We now consider a concrete $d$. Let $d\in C^2_\le(T^1,A^1)$ be such that
\begin{align}\label{c=(dl^2d)c'}
c=(\dl^2d)c'.
\end{align}
We assume that $d$ is strongly normalized (this is allowed by \cref{delta^2d-normalized}). Let $\tau=\tau_d$ be as constructed above. Given $(a,e,w)\in N$, we define $\f_1(a,e,w)\in N'$ as follows
\begin{align}\label{f_1(a_e_w)=(tau(irr(w))_e_w)}
	\f_1(a,e,w)=
	\begin{cases}
	(\tau(\irr w)a,e,w), & w\ne \e,\\
	(a,e,w), & w=\e.
	\end{cases}
\end{align}
Since $\tau(\irr w)\in A_{\br(\nu(w))}$ for $w\ne\e$ and $a\in A_e$, where $e\le\nu(w)$, we have $\tau(\irr w)a\in A_e$, so that $\f_1(a,e,w)$ indeed belongs to $N'$.

\begin{lem}\label{f_1-homo}
	The map $\f_1$ is a homomorphism of semigroups $N\to N'$.
\end{lem}
\begin{proof}
	Let $m=(a,e,u)$ and $n=(b,f,v)$ be elements of $N$. The only nontrivial part of the proof of $\f_1(mn)=\f_1(m)\f_1(n)$ is the equality 
	\begin{align*}
		\0(ef)\tau(\irr{uv})=\0(ef)\tau(\irr u)\tau(\irr v),
	\end{align*}
	where $\e\not\in\{u,v\}$. Its proof mimics the proof of \cref{z_t(uv)=z_t(u)z_t(v)} and is based on \cref{properties-of-tau}.
	
	
\end{proof}

\begin{lem}\label{tau.xi_t=xi'_t.eta_t-circ-tau}
	Let $t\in T$, $\ve\in\{-1,1\}$, $w\in F(T\sqcup T\m)$ and $e\in E(T)$ such that $e\le\f(w)$. Then
	\begin{align}\label{tau(w)xi_t(w)=xi'_t(w)eta_t(tau(w))}
		\0(\br(t^\ve e))\tau([t]^\ve w[t]^{-\ve})\xi_{t^\ve}(w)=\0(\br(t^\ve e))\xi'_{t^\ve}(w)\eta_{t^\ve}(\tau(w)).
	\end{align}
\end{lem}
\begin{proof}
	Assume first that $w=[x]\m w'$ for some $w'\in F(T\sqcup T\m)^1$. Then by \labelcref{xi_t([x]-inv.u),tau([x]-inv.u)} 
	\begin{align*}
		\xi'_{t^\ve}(w)\eta_{t^\ve}(\tau(w))&=c'(t^\ve,x\m,x)\m\eta_{t^\ve}(d(x\m,x)\m)\xi'_{t^\ve}([x\m]w')\eta_{t^\varepsilon}(\tau([x\m]w')).
	\end{align*}
	
	Observe from \cref{c=(dl^2d)c',c(x_1...x_(i-1)e_e_ex_(i+1)...x_n)-triv} that 
	\begin{align}
	c(t^\ve,x\m,x)&=c'(t^\ve,x\m,x)(\dl^2d)(t^\ve,x\m,x)\notag\\
	&=c'(t^\ve,x\m,x)\eta_{t^\ve}(d(x\m,x))d(t^\ve x\m,x)\m d(t^\ve,x\m x) d(t^\ve,x\m)\m\notag\\
	&=c'(t^\ve,x\m,x)\eta_{t^\ve}(d(x\m,x))d(t^\ve x\m,x)\m d(t^\ve,x\m)\m.\label{c(t^ve_x-inv_x)=c'(t^ve_x-inv_x)eta_t-inv(...)}
	\end{align}
	
	Consider the case $\ve=1$. Then by \cref{tau([x][y]-inv.v),tau([x][y]v),xi_t([x]-inv.u)}
	\begin{align*}
		\tau([t]^\ve w[t]^{-\ve})\xi_{t^\ve}(w)&=d(t^\ve x\m,x)\m d(t^\ve,x\m)\m c(t^\ve,x\m,x)\m \tau([t]^\ve[x\m]w'[t]^{-\ve})\xi_{t^\ve}([x\m]w'),
	\end{align*}
	so in view of \cref{c(t^ve_x-inv_x)=c'(t^ve_x-inv_x)eta_t-inv(...)} it suffices to prove
	\begin{align}\label{tau(w)xi_t([x-inv]w')=xi'_t(w)eta_t(tau([x-inv]w'))}
		\tau([t]^\ve[x\m]w'[t]^{-\ve})\xi_{t^\ve}([x\m]w')=\xi'_{t^\ve}([x\m]w')\eta_{t^\ve}(\tau([x\m]w')).
	\end{align}
	
	
	Now let $\ve=-1$. Then by \cref{tau([x][y]-inv.v),tau([x][y]v),xi_t([x]-inv.u),tau([x]-inv.u)}
	\begin{align*}
	\tau([t]^\ve w[t]^{-\ve})\xi_{t^\ve}(w)&=d(t\m,t)\m d(t\m x\m,x)\m d(t\m,x\m)\m c(t\m,x\m,x)\m\\ &\quad\cdot\tau([t\m][x\m]w'[t])\xi_{t\m}([x\m]w')\\
	&=d(t^\ve x\m,x)\m d(t^\ve,x\m)\m c(t^\ve,x\m,x)\m\\ &\quad\cdot\tau([t]^\ve[x\m]w'[t]^{-\ve})\xi_{t^\ve}([x\m]w'),
	\end{align*}
	and again \cref{tau(w)xi_t(w)=xi'_t(w)eta_t(tau(w))} reduces to \cref{tau(w)xi_t([x-inv]w')=xi'_t(w)eta_t(tau([x-inv]w'))} thanks to \cref{c(t^ve_x-inv_x)=c'(t^ve_x-inv_x)eta_t-inv(...)}.
	Since, moreover, $\f(w)=\f([x\m]w')$, it is enough to consider $w=[x]w'$.
	
	We proceed by induction on $l(w)$. Let $w=[x]$. Then $e\le x$ and by \cref{xi_t([x]),tau([x])}
	\begin{align*}
		\0(\br(t^\ve e))\xi'_{t^\ve}(w)\eta_{t^\ve}(\tau(w))&=\0(\br(t^\ve e))\0(\br(t^\ve x))\eta_{t^\ve}(\0(\br(x)))=\0(\br(t^\ve e).
	\end{align*}
	For $\ve=1$ we have by \cref{tau([x][y]-inv.v),tau([x][y]v),tau([x]),xi_t([x]),c(x_1...x_(i-1)_e_x_(i+1)...x_n)-triv,c(x_1...ex_i...x_n)}
	\begin{align*}
		\0(\br(t^\ve e))\tau([t]^\ve w[t]^{-\ve})\xi_{t^\ve}(w)&=\0(\br(te))d(t,x)d(txt\m,t)\m \0(\br(txt^{-1}))\\
		&=d(t,ex)d(text\m,t)\m = d(t,e)d(tet\m,t)\m = \0(\br(t^\ve e)).
	\end{align*}
	
	For $\ve=-1$ the calculation is the following
	\begin{align*}
	\0(\br(t^\ve e))\tau([t]^\ve w[t]^{-\ve})\xi_{t^\ve}(w)&=\0(\br(t\m e))d(t\m,t)\m d(t\m,x)d(t\m x,t) \0(\br(t\m xt))\\
	&=d(t\m,t)\m d(t\m,ex)d(t\m ex,t)=d(t\m,t)\m d(t\m,e)d(t\m e,t)\\
	&=\0(\br(t\m e))d(t\m,t)\m d(t\m,t) =\0(\br(t^\ve e)).
	\end{align*}
	
	Let $l(w)>1$. Then we have two cases.
	
	\textbf{Case 1.} $w=[x][y]w'$. Then by \cref{xi_t([x][y]v),tau([x][y]v)}
	\begin{align*}
		\xi'_{t^\ve}(w)\eta_{t^\ve}(\tau(w))=c'(t^\ve,x,y)\eta_{t^\ve}(d(x,y))\xi'_{t^\ve}([xy]w')\eta_{t^\ve}(\tau([xy]w')).
	\end{align*}
	Let $\ve=1$. Then by \cref{tau([x][y]v),xi_t([x][y]v),tau([x][y]-inv.v)}
	\begin{align*}
	\tau([t]^\ve w[t]^{-\ve})\xi_{t^\ve}(w)&=d(t^\ve,x) d(t^\ve x,y) d(t^\ve,xy)\m c(t^\ve,x,y) \tau([t]^\ve[xy]w'[t]^{-\ve})\xi_{t^\ve}([xy]w').
	\end{align*}
	However,
	\begin{align*}
		c(t^\ve,x,y)&=c'(t^\ve,x,y)(\dl^2d)(t^\ve,x,y)=c'(t^\ve,x,y)\eta_{t^\ve}(d(x,y))d(t^\ve x,y)\m d(t^\ve,xy) d(t^\ve,x)\m,
	\end{align*}
	so it is enough to prove 
	\begin{align}\label{tau([xy]w')xi_t([xy]w')=xi'_t([xy]w')eta_t(tau([xy]w'))}
		\xi'_{t^\ve}([xy]w')\eta_{t^\ve}(\tau([xy]w'))=\tau([t]^\ve[xy]w'[t]^{-\ve})\xi_{t^\ve}([xy]w').
	\end{align} 
	Let $\ve=-1$. Then by \cref{tau([x][y]v),xi_t([x][y]v),tau([x][y]-inv.v),tau([x]-inv.u)}
	\begin{align*}
	\tau([t]^\ve w[t]^{-\ve})\xi_{t^\ve}(w)&=d(t\m, t)\m d(t\m,x) d(t\m x,y) d(t\m,xy)\m c(t\m,x,y)\\
	&\quad\cdot \tau([t\m][xy]w'[t])\xi_{t\m}([xy]w')\\
	&=d(t^\ve,x) d(t^\ve x,y) d(t^\ve,xy)\m c(t^\ve,x,y) \tau([t]^\ve[xy]w'[t]^{-\ve})\xi_{t^\ve}([xy]w'),
	\end{align*}
	and again \cref{tau(w)xi_t(w)=xi'_t(w)eta_t(tau(w))} reduces to \cref{tau([xy]w')xi_t([xy]w')=xi'_t([xy]w')eta_t(tau([xy]w'))}.
	Since, moreover, $\f([xy]w')=\f(w)$ and $l([xy]w')<l(w)$, we may use the induction hypothesis.
	
	\textbf{Case 2.} $w=[x][y]\m w'$. This case is similar to the previous one.
\end{proof}

\begin{lem}\label{f_1-commutes-with-gm}
	For any $t\in T$ and $\ve\in\{-1,1\}$ we have $\f_1\circ\gm^\ve_t = (\gm'_t)^\ve\circ\f_1$.
\end{lem}
\begin{proof}
	Let $n=(a,e,w)\in N$. If $w=\e$, then $(\f_1\circ\gm_t^\ve)(n)=(\0(\br(t^\ve))\eta_{t^\ve}(a),t^\ve et^{-\ve},\e)=((\gm'_t)^\ve\circ\f_1)(n)$. Let $w\ne\e$. 
	
	\textbf{Case 1.} $[t]^\ve\irr w[t]^{-\ve}$ is irreducible. Then $\irr {[t]^\ve w[t]^{-\ve}}=[t]^\ve\irr w[t]^{-\ve}$.
	By \cref{gm_t(a_e_w)=(z_t(w)eta_t(a)_tet-inv_[t]w[t]-inv),f_1(a_e_w)=(tau(irr(w))_e_w),z_t(w)=xi_t(irr(w)),tau.xi_t=xi'_t.eta_t-circ-tau,gm-inv_t(a_e_w)=}
	\begin{align}
		(\f_1\circ\gm_t^\ve)(n)&=(\tau(\irr {[t]^\ve w[t]^{-\ve}})\xi_{t^\ve}(\irr w)\eta_{t^\ve}(a),t^\ve et^{-\ve},[t]^\ve w[t]^{-\ve})\notag\\
		&=(\tau([t]^\ve\irr w[t]^{-\ve})\xi_{t^\ve}(\irr w)\eta_{t^\ve}(a),t^\ve et^{-\ve},[t]^\ve w[t]^{-\ve})\notag\\
		&=(\xi'_{t^\ve}(\irr w)\eta_{t^\ve}(\tau(\irr w)a),t^\ve et^{-\ve},[t]^\ve w[t]^{-\ve})\notag\\
		&=(\z'_{t^\ve}(w)\eta_{t^\ve}(\tau(\irr w)a),t^\ve et^{-\ve},[t]^\ve w[t]^{-\ve})=((\gm'_t)^\ve\circ\f_1)(n).\label{(f_1-circ-gm_t)(n)=(lb'_t-circ-f_1)(n)}
	\end{align}
	
	\textbf{Case 2.} $\irr w=[t]^{-\ve} w'$, where $w'$ is not empty and the last letter of $w'$ is different from $[t]^\ve$. Then $\irr{[t]^\ve w[t]^{-\ve}}=w'[t]^{-\ve}$ and $\tau([t]^\ve\irr w[t]^{-\ve})=\tau([t]^\ve[t]^{-\ve} w'[t]^{-\ve})$. If $\ve=1$, then by \cref{properties-of-tau,tau([x][y]-inv.v),c(x_1...x_(i-1)_e_x_(i+1)...x_n)-triv}
	\begin{align*}
		\tau([t]^\ve[t]^{-\ve} w'[t]^{-\ve})&=d(tt\m,t)\m\tau([tt\m]w'[t]\m)=\0(\br(t))\tau(w'[t]\m)\\
		&=\0(\br(t^\ve))\tau(\irr{[t]^\ve w[t]^{-\ve}}),
	\end{align*}
	so the proof of \cref{(f_1-circ-gm_t)(n)=(lb'_t-circ-f_1)(n)} works. If $\ve=-1$, then by \cref{properties-of-tau,tau([x][y]-inv.v),tau([x]-inv.u)}
	\begin{align*}
	\tau([t]^\ve[t]^{-\ve} w'[t]^{-\ve})&=d(t\m, t)\m d(t\m,t)\tau([t\m t]w'[t]\m)=\0(\bd(t))\tau(w'[t]\m)\\
	&=\0(\br(t^\ve))\tau(\irr{[t]^\ve w[t]^{-\ve}}),
	\end{align*}
	and the proof of \cref{(f_1-circ-gm_t)(n)=(lb'_t-circ-f_1)(n)} again works.
	
	\textbf{Case 3.} $\irr w=w'[t]^\ve$, where $w'$ is not empty and the first letter of $w'$ is different from $[t]^{-\ve}$. Then $\irr{[t]^\ve w[t]^{-\ve}}=[t]^\ve w'$ and $\tau([t]^\ve\irr w[t]^{-\ve})=\tau([t]^\ve w'[t]^\ve[t]^{-\ve})$. If $\ve=1$, then by \cref{properties-of-tau,tau([x][y]v),c(x_1...x_ie...x_n)}
	\begin{align*}
	\tau([t]^\ve w'[t]^\ve[t]^{-\ve})&=\tau([t]w')\tau([t\f(w')][t][t]\m)\\
	&=\tau(\irr{[t]w[t]\m})d(t\f(w'),t)d(t\f(w')tt\m,t)\m\\
	&=\0(\br(t^\ve\f(\irr w)))\tau(\irr{[t]^\ve w[t]^{-\ve}}),
	\end{align*}
	so the proof of \cref{(f_1-circ-gm_t)(n)=(lb'_t-circ-f_1)(n)} remains valid. If $\ve=-1$, then by \cref{properties-of-tau,tau([x][y]v),tau([x][y]-inv.v)}
	\begin{align*}
	\tau([t]^\ve w'[t]^\ve[t]^{-\ve})&=\tau([t]\m w')\tau([t\m\f(w')][t]\m [t])\\
	&=\tau(\irr{[t]\m w[t]})d(t^{-1}\f(w')t\m,t)\m d(t^{-1}\f(w')t\m,t)\\
	&=\0(\br(t^\ve\f(\irr w)))\tau(\irr{[t]^\ve w[t]^{-\ve}}),
	\end{align*}
	and the proof of \cref{(f_1-circ-gm_t)(n)=(lb'_t-circ-f_1)(n)} still remains valid.
	The rest of the cases are similar (for more details see Cases 4--6 of \cref{gm^ve_t-circ-gm^dl_u}).
\end{proof}

\begin{lem}\label{f_1-commutes-with-lb}
	For any $s\in S$ we have $\f_1\circ\lb_s = \lb'_{\f_2(s)}\circ\f_1$.
\end{lem}
\begin{proof}
	Let $s=(t,u)\in S$ and $n=(a,e,w)\in N$. Consider first the case $u=\e$. Then $t\in E(T)$.
	If $w\ne\e$, then $(\f_1\circ\lb_s)(n)=\f_1(\0(t)a,te,w)=(\tau(\irr w)\0(t)a,te,w)=\lb'_s(\tau(\irr w)a,e,w)=(\lb'_{\f_2(s)}\circ\f_1)(n)$ by \cref{lb_(t_u)(a_e_w)=,f_1(a_e_w)=(tau(irr(w))_e_w)}. Otherwise, $(\f_1\circ\lb_s)(n)=\f_1(\0(t)a,te,w)=(\0(t)a,te,w)=\lb'_s(a,e,w)=(\lb'_{\f_2(s)}\circ\f_1)(n)$.
	
	Now assume $u\ne\e$. Let $\irr u=[t_1]^{\ve_1}\dots[t_k]^{\ve_k}$. By \cref{lb_(t_u)(a_e_w)=,f_1-commutes-with-gm,f_1(a_e_w)=(tau(irr(w))_e_w)}
	\begin{align*}
		(\f_1\circ\lb_s)(n)&=\f_1(\af(\br(s))(\gm_{t_1}^{\ve_1}\circ\dots\circ \gm_{t_k}^{\ve_k})(n))=\af(\br(s))(\f_1\circ\gm_{t_1}^{\ve_1}\circ\dots\circ \gm_{t_k}^{\ve_k})(n)\\
		&=\af(\br(s))((\gm'_{t_1})^{\ve_1}\circ\dots\circ (\gm'_{t_k})^{\ve_k})(\f_1(n))=(\lb'_{\f_2(s)}\circ\f_1)(n).
	\end{align*}
\end{proof}

\begin{prop}\label{cohom-cocycles-induce-equiv-ext}
	The crossed module extension $A \xrightarrow{i'} N' \xrightarrow{\beta'} S' \xrightarrow{\pi'} T$ induced by $c'$ is equivalent to the crossed module extension $A \xrightarrow{i} N \xrightarrow{\beta} S \xrightarrow{\pi} T$ induced by $c$.
\end{prop}
\begin{proof}
	The commutativity of the right square of \cref{eq:equivseqs4} is obvious, as $S=S'$, $\pi=\pi'$ and $\f_2=\id$. Since $\f_1$ changes only the first coordinate of the triple $(a,e,w)\in N$, we have $\bt=\bt'\circ\f_1$, which is the commutativity of the middle square of \cref{eq:equivseqs4}. Now, $\f_1\circ i=i'$ is also trivial, because $\f_1$ acts as the identity map on the triples $(a,e,w)\in N$ with $w=\e$. So, the left square of \cref{eq:equivseqs4} is commutative too, and \cref{CMEE1} of \cref{equiv-ext-defn} is established. Item \cref{CMEE2} of \cref{equiv-ext-defn} is \cref{f_1-commutes-with-lb}.
\end{proof}

\begin{prop}\label{from-H^3_le-to-E(T_A)}
	There is a map from $H^3_\le(T^1,A^1)$ to $\mathcal{E}(T,A)$. 
\end{prop}
\begin{proof}
	Let $A$ be a $T$-module. By \cref{from-c-to-crossed-mod-ext} each $c\in Z^3_\le(T^1,A^1)$ induces a crossed module extension of $A$ by $T$, and in view of \cref{cohom-cocycles-induce-equiv-ext} cohomologous $c,c'\in Z^3_\le(T^1,A^1)$ induce equivalent extensions.
\end{proof}

\section{The correspondence between \texorpdfstring{$H^3_\le(T^1,A^1)$}{H³<(T¹,A¹)} and \texorpdfstring{$\mathcal{E}_\le(T,A)$}{E<(T,A)}}\label{sec-H^3_le<->E_le}

\subsection{From \texorpdfstring{$c\in Z^3_\le(T^1,A^1)$}{c in Z³<(T¹,A¹)} to a crossed module extension and back again}

As in \cref{c->ext-sec} we begin with a $T$-module $A$ and $c\in Z^3_\le(T^1,A^1)$. Let $A \xrightarrow{i} N \xrightarrow{\beta} S \xrightarrow{\pi} T$ be the corresponding crossed module extension of $A$ by $T$ constructed in \cref{from-c-to-crossed-mod-ext}. Our aim is to prove that this crossed module extension is admissible and induces $c'\in Z^3_\le(T^1,A^1)$ cohomologous to $c$, provided that $T$ is an $F$-inverse monoid.

\begin{lem}\label{rho-and-sigma-giving-c}
	There are transversals $\rho:T\to S$ and $\s:\bt(N)\to N$ of $\pi$ and $\bt$ inducing the original cocycle $c\in Z^3_\le(T^1,A^1)$.
\end{lem}
\begin{proof}
	For any $t\in T$ we put
	\begin{align}\label{rho-for-E-unitary-cover}
		\rho(t)=(t,[t]).
	\end{align}
	Since $\nu([t])=t$, we have a well-defined map $\rho:T\to S$. Clearly, $\pi\circ\rho=\id_T$, so $\rho$ is a transversal of $\pi$. We now proceed as in \cref{from-crossed-mod-ext-to-C^3}:
	\begin{align*}
		\rho(x)\rho(y)=(xy,[x][y])=(\br(xy),[x][y][xy]\m)(xy,[xy])=f(x,y)\rho(xy),
	\end{align*}
	where
	\begin{align}\label{f-for-E-unitary-cover}
		f(x,y)=(\br(xy),[x][y][xy]\m).
	\end{align}
	We thus have $f:T^2\to\bt(N)$, such that $f(x,y)\in\bt(N)_{\br(\rho(xy))}$. Furthermore, define $\s:\bt(N)\to N$ as follows:
	\begin{align}\label{sigma-for-E-unitary-cover}
		\s(e,w)=(\0(e),e,w),
	\end{align}
	where $(e,w)\in\bt(N)$. Obviously, $\s$ is a transversal of $\bt$. We thereafter put
	\begin{align}\label{F-for-E-unitary-cover}
		F(x,y)=\s(f(x,y))=(\0(\br(xy)),\br(xy),[x][y][xy]\m),
	\end{align}
	and we obtain $F:T^2\to N$ with $F(x,y) \in N_{\af(\br(\rho(xy)))}$. Let us first calculate
	\begin{align*}
		F(x,y)F(xy,z)&=(\0(\br(xy)),\br(xy),[x][y][xy]\m)(\0(\br(xyz)),\br(xyz),[xy][z][xyz]\m)\\
		&=(\0(\br(xyz)),\br(xyz),[x][y][z][xyz]\m).
	\end{align*}
	Now, by \cref{lb_(t_u)(a_e_w)=,lb_(t_u)(a_e_w)=triple,z_t(w)=xi_t(irr(w))}
	\begin{align*}
		\lb_{\rho(x)}(F(y,z))F(x,yz)&=(\z_x([y][z][yz]\m)\eta_x(\0(\br(yz))),x\br(yz)x\m,[x][y][z][yz]\m[x]\m)\\
		&\quad\cdot(\0(\br(xyz)),\br(xyz),[x][yz][xyz]\m)\\
		&=(\z_x([y][z][yz]\m)\0(\br(xyz)),\br(xyz),[x][y][z][xyz]\m).
	\end{align*}
	
	\textbf{Case 1.} $[y][z][yz]\m$ is irreducible. Then by \cref{z_t(w)=xi_t(irr(w)),xi_t([x][y]v),xi_t([x][y]-inv.v),xi_t([x]),c(x_1...x_(i-1)e_e_ex_(i+1)...x_n)-triv}
	\begin{align*}
		\z_x([y][z][yz]\m)=\xi_x([y][z][yz]\m)=c(x,y,z)c(x,yzz\m y\m,yz)\m=c(x,y,z),
	\end{align*}
	whence \cref{lb_rho(x)(F(y_z))F(x_yz)=i(c(x_y_z))F(x_y)F(xy_z)}.
	
	\textbf{Case 2.} $z=yz$. Then $\irr{[y][z][yz]\m}=[y]$ and
	\begin{align*}
		\z_x([y][z][yz]\m)\0(\br(xyz))=\xi_x([y])\0(\br(xyz))=\0(\br(xy))\0(\br(xyz))=\0(\br(xyz)).
	\end{align*}
	However, $\br(z) =y\br(z)$, so by \cref{c(x_1...x_ie...x_n)-and-,c(x_1...x_(i-1)_e_x_(i+1)...x_n)-triv}
	\begin{align*}
		c(x,y,z)=c(x,y,\br(z)z)=c(x,y\br(z),z)=c(x,\br(z),z)=\0(\br(xz))=\0(\br(xyz)).
	\end{align*}
	Thus, \cref{lb_rho(x)(F(y_z))F(x_yz)=i(c(x_y_z))F(x_y)F(xy_z)} holds again.
\end{proof}

\begin{lem}\label{T-F-inverse=>admissible}
	Let $T$ be an $F$-inverse monoid. Then the crossed module extension $A \xrightarrow{i} N \xrightarrow{\beta} S \xrightarrow{\pi} T$ is admissible.
\end{lem}
\begin{proof}
	According to \cref{admissible-ext-defn} we need to show that $\pi$ and $\bt$ have order-preserving transversals which respect idempotents. Observe that the transversal $\s$ of $\bt$ defined in \cref{sigma-for-E-unitary-cover} respects idempotents and is order-preserving for arbitrary $T$. Now, given $t\in T$, we define
	\begin{align}\label{rho-for-F-inverse}
	\rho(t)=
	\begin{cases}
	(t,[\max t]), & t\not\in E(T),\\
	(t,\e), & t\in E(T).
	\end{cases}
	\end{align}
	Since $t\le\max t=\nu([\max t])$, then $\rho(t)\in S$, and $\rho$ is a transversal of $\pi$. It respects idempotents by its definition. Clearly, for $e,f\in E(T)$, $e\le f$ implies $\rho(e)\le\rho(f)$. Moreover, if $t\le u$ and $t,u\not\in E(T)$, then $\max t=\max u$ and hence $\rho(t)=(t,[\max t])\le (u,[\max u])=\rho(u)$.  So, $\rho$ is order-preserving.
\end{proof}

\begin{cor}\label{H^3_le(T^1_A^1)->E_le(T_A)}
	Let $T$ be an $F$-inverse monoid and $A$ a $T$-module. Then there is a map from $H^3_\le(T^1,A^1)$ to $\mathcal{E}_\le(T,A)$. 
\end{cor}
\begin{proof}
	This follows from \cref{from-H^3_le-to-E(T_A),T-F-inverse=>admissible}.
\end{proof}

\begin{cor}\label{c->ext->cohom-c'}
	The transversals \cref{sigma-for-E-unitary-cover,rho-for-F-inverse} induce a normalized $c'\in Z^3_\le(T^1,A^1)$ cohomologous to the original $c\in Z^3_\le(T^1,A^1)$ in $Z^3_\le(T^1,A^1)$.
\end{cor}
\begin{proof}
	Indeed, and by \cref{adm-ext-induces-str-norm-c,lem:E-cond} the cocycle $c'$ induced by the transversals \cref{sigma-for-E-unitary-cover,rho-for-F-inverse} belongs to $Z^3_\le(T^1,A^1)$, and thanks to \cref{another-choice-of-F,another-choice-of-rho,rho-and-sigma-giving-c} it is cohomologous to $c$ in $Z^3(T^1,A^1)$. Let $d\in C^2(T^1,A^1)$ satisfying \cref{c=(dl^2d)c'}. Since both $c$ and $c'$ are normalized and order-preserving, the coboundary $\dl^2d$ is also normalized and order-preserving, so in view of \cref{3-cocycle-cohom-to-normalized} there is $d'\in C^2_\le(T^1,A^1)$, such that $\dl^2d=\dl^2d'$. But then $c=(\dl^2d')c'$, and thus $c$ and $c'$ are cohomologous in $Z^3_\le(T^1,A^1)$.
\end{proof}

\subsection{From a crossed module extension to \texorpdfstring{$c\in Z^3_\le(T^1,A^1)$}{c in Z³<(T¹,A¹)} and back again}\label{ext->c->ext-sec}

Now we are given a $T$-module $A$ and an admissible crossed module extension $A \xrightarrow{i} N \xrightarrow{\beta} S \xrightarrow{\pi} T$ of $A$ by $T$. We fix order-preserving transversals $\rho$ and $\s$ of $\pi$ and $\bt$, respectively, which respect idempotents. Denote by $c$ the induced (strongly normalized) $3$-cocycle from $Z^3_\le(T^1,A^1)$ and consider the crossed module extension $A \xrightarrow{i'} N' \xrightarrow{\beta'} S' \xrightarrow{\pi'} T$ determined by $c$ as in \cref{c->ext-sec}. We are going to prove that the two extensions are equivalent.

Given $w\in F(T\sqcup T\m)$, $w=[t_1]^{\ve_1}\dots[t_n]^{\ve_n}$, we put
\begin{align}\label{chi(w)=rho(t_1)^ve_1...rho(t_n)^ve_n}
	\chi(w)=\rho(t_1)^{\ve_1}\dots\rho(t_n)^{\ve_n}.
\end{align}
Define $\f_2:S'\to S$ as follows:
\begin{align}\label{f_2(t_w)=rho(r(t))rho(t_1)^ve_1...rho(t_n)^ve_n}
	\f_2(t,w)=
	\begin{cases}
		\rho(\br(t))\chi(\irr w), & w\ne\e,\\
		\rho(t), & w=\e.
	\end{cases}
\end{align}

\begin{lem}\label{f_2-homo}
	The map $\f_2$ is a homomorphism $S'\to S$ such that $\pi\circ\f_2=\pi'$.
\end{lem}
\begin{proof}
	Let $(t,u),(p,v)\in S'$. We are going to prove that $\f_2(tp,uv)=\f_2(t,u)\f_2(p,v)$. This is trivial when $u=\e$, because $t\in E(T)$ in this case, so $\rho(\br(tp))=\rho(t\br(p))=\rho(t)\rho(\br(p))$ by \cref{rho-ord-pres-resp-idemp}. If $v=\e$ and $u\ne\e$, $\irr u=[t_1]^{\ve_1}\dots[t_n]^{\ve_n}$, then applying \cref{rho(x)^ve-commutes-with-rho(e)} several times and using $t\le\nu(u)$, we obtain
	\begin{align*}
		\f_2(t,u)\f_2(p,v)&=\rho(\br(t))\rho(t_1)^{\ve_1}\dots\rho(t_n)^{\ve_n}\rho(p)\\
		&=\rho(\br(t))\rho(t_1)^{\ve_1}\dots\rho(t_{n-1})^{\ve_{n-1}}\rho(\br(t_n^{\ve_n}p))\rho(t_n)^{\ve_n}\\
		&=\dots\\
		&=\rho(\br(t))\rho(\br(t_1^{\ve_1}\dots t_n^{\ve_n}p))\rho(t_1)^{\ve_1}\dots\rho(t_n)^{\ve_n}\\
		&=\rho(\br(t)\br(\nu(u)p))\chi(\irr {uv})=\rho(\br(tp))\chi(\irr {uv})=\f_2(tp,uv).
	\end{align*}
	
	Let now $\e\not\in\{u,v\}$.
	
	\textbf{Case 1.} $\irr u\irr v$ is irreducible. Then $\irr{uv}=\irr u\irr v$. Let $\irr u=[t_1]^{\ve_1}\dots[t_k]^{\ve_k}$ and $\irr v=[p_1]^{\dl_1}\dots[p_l]^{\dl_l}$, so that $\irr{uv}=[t_1]^{\ve_1}\dots[t_k]^{\ve_k}[p_1]^{\dl_1}\dots[p_l]^{\dl_l}$. As in the previous case, we use \cref{rho(x)^ve-commutes-with-rho(e)} and $t\le\nu(u)$:
	\begin{align*}
		\f_2(t,u)\f_2(p,v)&=\rho(\br(t))\rho(t_1)^{\ve_1}\dots\rho(t_k)^{\ve_k}\rho(\br(p))\rho(p_1)^{\dl_1}\dots\rho(p_l)^{\dl_l}\\
		&=\rho(\br(t))\rho(\br(\nu(u)p))\rho(t_1)^{\ve_1}\dots\rho(t_k)^{\ve_k}\rho(p_1)^{\dl_1}\dots\rho(p_l)^{\dl_l}\\
		&=\rho(\br(tp))\chi(\irr {uv})=\f_2(tp,uv).
	\end{align*}
	
	\textbf{Case 2.} $\irr u=u'w$ and $\irr v=w\m v'$ for some non-empty $u'$ and $v'$, where $w$ is the maximal suffix of $\irr u$ such that $w\m$ is a prefix of $\irr v$. Then $\irr{uv}=u'v'$. Let $u'=[t_1]^{\ve_1}\dots[t_k]^{\ve_k}$, $v'=[p_1]^{\dl_1}\dots[p_l]^{\dl_l}$ and $w=[q_1]^{\s_1}\dots[q_m]^{\s_m}$, so that $\irr u=[t_1]^{\ve_1}\dots[t_k]^{\ve_k}[q_1]^{\s_1}\dots[q_m]^{\s_m}$, $\irr v=[q_m]^{-\s_m}\dots[q_1]^{-\s_1}[p_1]^{\dl_1}\dots[p_l]^{\dl_l}$ and $\irr{uv}=[t_1]^{\ve_1}\dots[t_k]^{\ve_k}[p_1]^{\dl_1}\dots[p_l]^{\dl_l}$. Using  \cref{rho(x)^ve-commutes-with-rho(e),rho(t)rho(t)-inv=rho(tt-inv),rho(t)rho(e)rho(t)-inv=rho(tet-inv),nu(u)nu(v)<=nu(uv)} and $t\le\nu(u)$:
	\begin{align*}
	\f_2(t,u)\f_2(p,v)&=\rho(\br(t))\rho(t_1)^{\ve_1}\dots\rho(t_k)^{\ve_k}\rho(q_1)^{\s_1}\dots\rho(q_m)^{\s_m}\\
	&\quad\cdot\rho(\br(p))\rho(q_m)^{-\s_m}\dots\rho(q_1)^{-\s_1}\rho(p_1)^{\dl_1}\dots\rho(p_l)^{\dl_l}\\
	&=\rho(\br(t))\rho(t_1)^{\ve_1}\dots\rho(t_k)^{\ve_k}\rho(\br(\nu(w)p))\rho(p_1)^{\dl_1}\dots\rho(p_l)^{\dl_l}\\
	&=\rho(\br(t))\rho(\br(\nu(u')\nu(w)p))\rho(t_1)^{\ve_1}\dots\rho(t_k)^{\ve_k}\rho(p_1)^{\dl_1}\dots\rho(p_l)^{\dl_l}\\
	&=\rho(\br(t))\rho(\br(\nu(u)p))\chi(\irr {uv})=\rho(\br(tp))\chi(\irr {uv})=\f_2(tp,uv).
	\end{align*}
	
	The remaining $3$ cases $\irr u=u'w$ and $\irr v=w\m$; $\irr u=w$ and $\irr v=w\m v'$; $\irr u=w$ and $\irr v=w\m$ are similar to Case 2.
	
	Let $s=(t,w)\in S$. If $w=\e$, then $(\pi\circ\f_2)(s)=\pi(\rho(t))=t=\pi'(s)$. Otherwise, let $\irr w=[t_1]^{\ve_1}\dots[t_n]^{\ve_n}$, so using $t\le\nu(w)$ we obtain
	\begin{align*}
		(\pi\circ\f_2)(s)=\pi(\rho(\br(t))\rho(t_1)^{\ve_1}\dots\rho(t_n)^{\ve_n})=\br(t)t_1^{\ve_1}\dots t_n^{\ve_n}
		=\br(t)\nu(w)=t=\pi'(s).
	\end{align*}
\end{proof}

Let $f:T^2\to\bt(N)$ be the map determined by \cref{rho(x)rho(y)=f(x_y)rho(xy)} and $F=\s\circ f$. Then $F:T^2\to N$, $F(x,y)\in N_{(i\circ\0)(\br(xy))}$ and \cref{bt(F(x_y))=f(x_y)} holds. Since $\rho$ and $\s$ are order-preserving and respect idempotents, it follows from \cref{rho-ord-pres-resp-idemp,rho(x)rho(y)=f(x_y)rho(xy)} that both $f$ and $F$ satisfy \cref{c(x_1...ex_i...x_n),c(x_1...x_ie...x_n),c(x_1...x_(i-1)_e_x_(i+1)...x_n)-triv} with $\0$ replaced by $\rho$ and $i\circ\0=\af\circ\rho$, respectively. Moreover, by \cref{rho(x)rho(y)=f(x_y)rho(xy),rho(t)rho(t)-inv=rho(tt-inv)} we have $f(x,x\m)=\rho(\br(x))$, so 
\begin{align}\label{F(t-inv_t)=af(rho(t-inv.t))}
F(x,x\m)=\af(\rho(\br(x))).
\end{align}
Let now $\tau=\tau_F:F(T\sqcup T\m)\to N$ as defined in \cref{tau([x][y]-inv.v),tau([x][y]v),tau([x]),tau([x]-inv.u)}. We introduce $\f_1:N'\to N$ as follows
\begin{align}\label{f_1(a_e_w)=tau(irr(w))i(a)}
	\f_1(a,e,w)=
	\begin{cases}
		\tau(\irr w)i(a), & w\ne\e,\\
		i(a), & w=\e.
	\end{cases}
\end{align}

\begin{lem}\label{f_1-homo-N'->N}
	The map $\f_1$ is a homomorphism $N'\to N$ such that $\f_1\circ i'=i$.
\end{lem}
\begin{proof}
	Let $m=(a,e,u)$ and $n=(b,f,v)\in N'$. We first prove $\f_1(mn)=\f_1(m)\f_1(n)$. The case $\e\in\{u,v\}$ is trivial, so let $\e\not\in\{u,v\}$. Since $i(ab)=i(a)i(b)$, it suffices to prove $i(\0(e))\tau(\irr{uv})=i(\0(e))\tau(\irr u)\tau(\irr v)$, which follows from \cref{properties-of-tau}. 
	
	Given $a\in A$, thanks to \cref{i(a)=(a_0(aa-inv)_1),f_1(a_e_w)=tau(irr(w))i(a)} we have $(\f_1\circ i')(a)=\f_1(a,\0\m(aa\m),\e)=i(a)$.
\end{proof}

\begin{lem}\label{bt-circ-tau=chi}
	Let $w\in F(T\sqcup T\m)$ and $e\in E(T)$ be such that $e\le\f(w)$. Then
	\begin{align}\label{rho(e)bt(tau(w))=rho(e)chi(w)}
	\rho(e)\bt(\tau(w))=\rho(e)\chi(w).
	\end{align}
\end{lem}
\begin{proof}
	We first observe that it suffices to consider $w$ of the form $w=[x]w'$. For if $w=[x]\m w'$, then 
    \begin{align*}
	\rho(e)\beta(\tau(w)) & = \rho(e)\beta(\tau([x]^{-1}w')) && \\
	& = \rho(e)\beta(F(x^{-1},x)^{-1})\beta(\tau([x^{-1}]w')) && \text{by } \labelcref{tau([x]-inv.u)} \\
	& = \rho(e)\beta(\alpha(\rho(\br(x^{-1}))))\beta(\tau([x^{-1}]w')) && \text{by } \labelcref{F(t-inv_t)=af(rho(t-inv.t))} \\
	& = \rho(e) \rho(\br(x^{-1})) \beta(\tau([x^{-1}]w')), && \beta \text{ and } \alpha \text{ are inverse on idempotents} \\
	\rho(e)\chi(w) & = \rho(e)\chi([x]^{-1}w') && \\
	& = \rho(e)\rho(x)^{-1} \chi(w') && \text{by } \labelcref{chi(w)=rho(t_1)^ve_1...rho(t_n)^ve_n} \\
	& = \rho(e)\rho(x)^{-1} \rho(x) \rho(x)^{-1} \chi(w') \\
	& = \rho(e)\rho(x^{-1}) \rho(x) \rho(x^{-1}) \chi(w') && \text{by  \Cref{rho(x-inv)=rho(x)-inv}} \\
	& = \rho(e)\rho(x^{-1}) \rho(x) \chi([x^{-1}]w') && \text{by } \labelcref{chi(w)=rho(t_1)^ve_1...rho(t_n)^ve_n} \\
	& = \rho(e) \rho(x^{-1}x) \chi([x^{-1}]w'). && \text{by } \labelcref{rho(t)rho(t)-inv=rho(tt-inv)}
	\end{align*}
	
	We will prove \cref{rho(e)bt(tau(w))=rho(e)chi(w)} by induction on $l(w)$ for all $w$ of the form $[x]w'$. Base of induction: let $w=[x]$. Then by \cref{tau([x]),rho-ord-pres-resp-idemp}
	\begin{align*}
	\rho(e)\bt(\tau(w))&=\rho(e)(\bt\circ \af\circ\rho)(\br(x))=\rho(e)\rho(\br(x))\\
	&=\rho(e\br(x))=\rho(e)=\rho(e)\rho(x)=\rho(e)\chi(w).
	\end{align*}
	
	Let $l(w)>1$. 
	
	\textbf{Case 1.} $w=[x][y]w'$. Then by \cref{tau([x][y]v),rho(x)rho(y)=f(x_y)rho(xy)}
	\begin{align*}
		\bt(\tau(w))&=\bt(F(x,y))\bt(\tau([xy]w'))=f(x,y)\bt(\tau([xy]w')),\\
		\chi(w)&=\rho(x)\rho(y)\chi(w')=f(x,y)\rho(xy)\chi(w')=f(x,y)\chi([xy]w').
	\end{align*}
	Since $l([xy]w')<l(w)$ and $\f([xy]w')=\f(w)$, we may use the induction hypothesis.
	
	\textbf{Case 2.} $w=[x][y]\m w'$. Then by \cref{tau([x][y]-inv.v),rho(x)rho(y)=f(x_y)rho(xy),rho(t)rho(t)-inv=rho(tt-inv),rho-ord-pres-resp-idemp}
	\begin{align*}
	\bt(\tau(w))&=\bt(F(xy\m,y)\m)\bt(\tau([xy\m]w'))=f(xy\m,y)\m\bt(\tau([xy\m]w')),\\
	\chi(w)&=\rho(x)\rho(y)\m\chi(w')=\rho(x)\rho(y)\m \rho(y)\rho(y)\m\chi(w')\\
	&=\rho(x)\rho(y\m y)\rho(y)\m\chi(w')=\rho(xy\m y)\rho(y)\m\chi(w')\\
	&=f(xy\m,y)\m\rho(xy\m)\rho(y)\rho(y)\m\chi(w')=f(xy\m,y)\m\rho(xy\m)\rho(yy\m)\chi(w')\\
	&=f(xy\m,y)\m\rho(xy\m)\chi(w')=f(xy\m,y)\m\chi([xy\m]w').
	\end{align*}
	Since $l([xy\m]w')<l(w)$ and $\f([xy\m]w')=\f(w)$, we may use the induction hypothesis.
\end{proof}

\begin{lem}\label{middle-square-comm}
	We have $\bt\circ\f_1=\f_2\circ\bt'$.
\end{lem}
\begin{proof}
	Let $n=(a,e,w)$. If $w=\e$, then using \cref{CME3,f_1(a_e_w)=tau(irr(w))i(a),0=i^(-1).af.pi^(-1),rho|_E(T)=pi-inv|_E(T),f_2(t_w)=rho(r(t))rho(t_1)^ve_1...rho(t_n)^ve_n} we have
	\begin{align*}
		(\bt\circ\f_1)(n)&=\bt(i(a))=\bt(i(a))\bt(i(a))\m=\bt(i(aa\m))=(\af\m\circ i\circ\0)(e)\\
		&=\pi\m(e)=\rho(e)=\f_2(e,w)=(\f_2\circ\bt')(n).
	\end{align*} 
	
	Let $w\ne\e$. Then in view of $\bt(i(a))=\rho(e)$, \cref{f_1(a_e_w)=tau(irr(w))i(a),bt-circ-tau=chi,f_2(t_w)=rho(r(t))rho(t_1)^ve_1...rho(t_n)^ve_n} we have
	\begin{align*}
		(\bt\circ\f_1)(n)&=\bt(\tau(\irr w)i(a))=\rho(e)\bt(\tau(\irr w))=\rho(e)\chi(\irr w)=(\f_2\circ\bt')(n).
	\end{align*}
\end{proof}

\begin{lem}\label{tau.i-circ-xi_t=lb_rho(t)-circ-tau}
	Let $t\in T$, $\ve\in\{-1,1\}$, $w\in F(T\sqcup T\m)$ and $e\in E(T)$ such that $e\le\f(w)$. Then
	\begin{align}\label{tau([t]w[t]-inv)i(xi_t(w))=lb_rho(t)(tau(w))}
	\af(\rho(\br(t^\ve e)))\tau([t]^\ve w[t]^{-\ve})i(\xi_{t^\ve}(w))=\af(\rho(\br(t^\ve e)))\lb_{\rho(t)^\ve}(\tau(w)).
	\end{align}
\end{lem}
\begin{proof}
	Consider first the case $w=[x]\m w'$ for some $w'\in F(T\sqcup T\m)^1$. Then by \cref{tau([x]-inv.u)}
	\begin{align*}
		\lb_{\rho(t)^\ve}(\tau(w))&=\lb_{\rho(t)^\ve}(F(x\m,x)\m)\lb_{\rho(t)^\ve}(\tau([x\m]w')).
	\end{align*}
	If $\ve=1$, then by \cref{tau([x][y]-inv.v),tau([x][y]v),xi_t([x]-inv.u),i(A)-sst-C(N)}
	\begin{align*}
		\tau([t]^\ve w[t]^{-\ve})i(\xi_{t^\ve}(w))&=F(tx\m,x)\m F(t,x\m)\m i(c(t,x\m,x)\m)\\
		&\quad\cdot\tau([t]^\ve [x\m]w'[t]^{-\ve})i(\xi_{t^\ve}([x\m]w')).
	\end{align*}
	However, by \cref{lb_rho(x)(F(y_z))F(x_yz)=i(c(x_y_z))F(x_y)F(xy_z)}
	\begin{align*}
		\lb_{\rho(t)}(F(x\m,x))F(t,x\m x)=i(c(t,x\m,x))F(t,x\m)F(tx\m,x),
	\end{align*}
	where $F(t,x\m x)=\af(\rho(\br(tx\m)))$ by \cref{c(x_1...x_(i-1)_e_x_(i+1)...x_n)-triv}. If $\ve=-1$, then by \cref{tau([x][y]-inv.v),tau([x][y]v),xi_t([x]-inv.u),i(A)-sst-C(N),tau([x]-inv.u),F(t-inv_t)=af(rho(t-inv.t))}
	\begin{align*}
	\tau([t]^\ve w[t]^{-\ve})i(\xi_{t^\ve}(w))&=F(t\m x\m,x)\m F(t\m,x\m)\m i(c(t\m,x\m,x)\m)\\
	&\quad\cdot\tau([t\m][x\m]w'[t])i(\xi_{t\m}([x\m]w'))\\
	&=F(t\m x\m,x)\m F(t\m,x\m)\m i(c(t\m,x\m,x)\m)\\
	&\quad\cdot\tau([t]^\ve [x\m]w'[t]^{-\ve})i(\xi_{t^\ve}([x\m]w')).
	\end{align*}
	But by \cref{lb_rho(x)(F(y_z))F(x_yz)=i(c(x_y_z))F(x_y)F(xy_z),rho(x-inv)=rho(x)-inv}
	\begin{align*}
	\lb_{\rho(t)\m}(F(x\m,x))F(t\m,x\m x)=i(c(t\m,x\m,x))F(t\m,x\m)F(t\m x\m,x).
	\end{align*}
	Since $\f(w)=\f([x\m]w')$, we see that it suffices to prove \cref{tau([t]w[t]-inv)i(xi_t(w))=lb_rho(t)(tau(w))} for $w$ of the form $[x]w'$.
	
	The proof will be by induction on $l(w)$. Let $w=[x]$. Then by \cref{tau([x]),rho(t)rho(e)rho(t)-inv=rho(tet-inv),TM2,rho(x-inv)=rho(x)-inv}
	\begin{align*}
		\af(\rho(\br(t^\ve e)))\lb_{\rho(t)^\ve}(\tau(w))&=\af(\rho(\br(t^\ve e)))\lb_{\rho(t)^\ve}(\af(\rho(\br(x))))=\af(\rho(\br(t^\ve e)))\af(\rho(\br(t^\ve x)))\\
		&=\af(\rho(\br(t^\ve e))).
	\end{align*}
	If $\ve=1$, then by \cref{tau([x]),tau([x][y]v),tau([x][y]-inv.v),0=i^(-1).af.pi^(-1),c(x_1...ex_i...x_n),c(x_1...x_(i-1)_e_x_(i+1)...x_n)-triv} and $e\le x$
	\begin{align*}
		\af(\rho(\br(t^\ve e)))\tau([t]^\ve w[t]^{-\ve})i(\xi_{t^\ve}(w))&=\af(\rho(\br(te)))F(t,x) F(txt\m,t)\m i(\0(\br(tx)))\\
		&=F(t,ex) F(text\m,t)\m \af(\rho(\br(tx)))\\
		&=F(t,e) F(tet\m,t)\m \af(\rho(\br(tx)))\\
		&=\af(\rho(\br(te)))\af(\rho(\br(tx)))=\af(\rho(\br(te)))=\af(\rho(\br(t^\ve e))).
	\end{align*}
	If $\ve=-1$, then by \cref{tau([x]),tau([x][y]v),tau([x][y]-inv.v),0=i^(-1).af.pi^(-1),tau([x]-inv.u),F(t-inv_t)=af(rho(t-inv.t)),c(x_1...ex_i...x_n)} and $e\le x$
	\begin{align*}
	\af(\rho(\br(t^\ve e)))\tau([t]^\ve w[t]^{-\ve})i(\xi_{t^\ve}(w))&=\af(\rho(\br(t\m e))) F(t\m,x) F(t\m x,t) i(\0(\br(t\m x)))\\
	&=F(t\m,ex) F(t\m ex,t) \af(\rho(\br(t\m x)))\\
	&=F(t\m,e) F(t\m e,t) \af(\rho(\br(t\m x)))\\
	&=\af(\rho(\br(t\m e))) \af(\rho(\br(t\m x)))\\
	&=\af(\rho(\br(t\m e)))=\af(\rho(\br(t^\ve e))).
	\end{align*}
	
	Now let $l(w)>1$.
	
	\textbf{Case 1.} $w=[x][y]w'$. Then by \cref{tau([x][y]v)}
	\begin{align*}
	\lb_{\rho(t)^\ve}(\tau(w))&=\lb_{\rho(t)^\ve}(F(x,y))\lb_{\rho(t)^\ve}(\tau([xy]w')).
	\end{align*}
	If $\ve=1$, then by \cref{tau([x][y]v),xi_t([x][y]v),i(A)-sst-C(N)}
	\begin{align*}
	\tau([t]^\ve w[t]^{-\ve})i(\xi_{t^\ve}(w))&=i(c(t,x,y))F(t,x) F(tx,y) F(t,xy)\m 
	\tau([t]^\ve [xy]w'[t]^{-\ve})i(\xi_{t^\ve}([xy]w')).
	\end{align*}
	But by \cref{lb_rho(x)(F(y_z))F(x_yz)=i(c(x_y_z))F(x_y)F(xy_z)}
	\begin{align*}
	\lb_{\rho(t)}(F(x,y))F(t,xy)=i(c(t,x,y))F(t,x)F(tx,y).
	\end{align*}
	If $\ve=-1$, then by \cref{tau([x][y]v),xi_t([x]-inv.u),xi_t([x][y]v),i(A)-sst-C(N),tau([x]-inv.u),F(t-inv_t)=af(rho(t-inv.t))}
	\begin{align*}
	\tau([t]^\ve w[t]^{-\ve})i(\xi_{t^\ve}(w))&=i(c(t\m,x,y))F(t\m,x) F(t\m x,y) F(t\m,xy)\m\\
	&\quad\cdot\tau([t\m][xy]w'[t])i(\xi_{t\m}([xy]w'))\\
	&=F(t\m,x) F(t\m x,y) F(t\m,xy)\m i(c(t\m,x,y))\\
	&\cdot\tau([t]^\varepsilon[xy]w'[t]^{-\varepsilon})i(\xi_{t^\varepsilon}([xy]w'))
	\end{align*}
    
	But by \cref{lb_rho(x)(F(y_z))F(x_yz)=i(c(x_y_z))F(x_y)F(xy_z),rho(x-inv)=rho(x)-inv}
	\begin{align*}
	\lb_{\rho(t)\m}(F(x,y))F(t\m,xy)=i(c(t\m,x,y))F(t\m,x)F(t\m x,y).
	\end{align*}
	Since $\f(w)=\f([xy]w')$ and $l([xy]w')<l(w)$, we may use the induction hypothesis.
	
	\textbf{Case 2.} $w=[x][y]\m w'$. This case is similar to the previous one.
\end{proof}

\begin{lem}\label{f_1-commutes-with-gm-via-rho}
	For any $t\in T$ and $\e\in\{-1,1\}$ we have $\f_1\circ\gm^\ve_t = \lb_{\rho(t)^\ve}\circ\f_1$.
\end{lem}
\begin{proof}
	Take $n=(a,e,w)\in N'$. If $w=\e$, then by \cref{gm_t(a_e_w)=(z_t(w)eta_t(a)_tet-inv_[t]w[t]-inv),gm-inv_t(a_e_w)=,f_1(a_e_w)=tau(irr(w))i(a),eta_t(a)=i^(-1)(lb_s(i(a))),rho(x-inv)=rho(x)-inv} we have $(\f_1\circ\gm^\ve_t)(n)=\f_1(\eta_{t^\ve}(a),t^\ve et^{-\ve},\e)=i(\eta_{t^\ve}(a))=\lb_{\rho(t^\ve)}(i(a))=\lb_{\rho(t)^\ve}(i(a))=(\lb_{\rho(t)^\ve}\circ\f_1)(n)$. Now let $w\ne\e$.
	
	\textbf{Case 1.} $[t]^\ve\irr w[t]^{-\ve}$ is irreducible. Then $\irr {[t]^\ve w[t]^{-\ve}}=[t]^\ve\irr w[t]^{-\ve}$. Then by \cref{gm_t(a_e_w)=(z_t(w)eta_t(a)_tet-inv_[t]w[t]-inv),gm-inv_t(a_e_w)=,f_1(a_e_w)=tau(irr(w))i(a),tau.i-circ-xi_t=lb_rho(t)-circ-tau,z_t(w)=xi_t(irr(w)),eta_t(a)=i^(-1)(lb_s(i(a))),rho(x-inv)=rho(x)-inv}
	\begin{align}
		(\f_1\circ\gm^\ve_t)(n)&=\f_1(\z_{t^\ve}(w)\eta_{t^\ve}(a),t^\ve et^{-\ve},[t]^\ve w[t]^{-\ve})\notag\\
		&=\tau([t]^\ve\irr w[t]^{-\ve})i(\xi_{t^\ve}(\irr w))i(\eta_{t^\ve}(a))\notag\\
		&=\lb_{\rho(t)^\ve}(\tau(\irr w))\lb_{\rho(t^\ve)}(i(a))\notag\\
		&=\lb_{\rho(t)^\ve}(\tau(\irr w)i(a))=(\lb_{\rho(t)^\ve}\circ\f_1)(n).\label{(f_1-circ-gm^ve_t)(n)=(lb_rho(t)^ve-circ-f_1)(n)}
	\end{align}
	
	\textbf{Case 2.} $\irr w=[t]^{-\ve} w'$, where $w'$ is not empty and the last letter of $w'$ is different from $[t]^\ve$. Then $\irr{[t]^\ve w[t]^{-\ve}}=w'[t]^{-\ve}$ and $\tau([t]^\ve\irr w[t]^{-\ve})=\tau([t]^\ve[t]^{-\ve} w'[t]^{-\ve})$. As in Case 2 of \cref{f_1-commutes-with-gm} we prove that $
	\tau([t]^\ve[t]^{-\ve} w'[t]^{-\ve})=i(\0(\br(t^\ve)))\tau(\irr{[t]^\ve w[t]^{-\ve}})$,
	so the proof of \cref{(f_1-circ-gm^ve_t)(n)=(lb_rho(t)^ve-circ-f_1)(n)} works.
	
	\textbf{Case 3.} $\irr w=w'[t]^\ve$, where $w'$ is not empty and the first letter of $w'$ is different from $[t]^{-\ve}$. Then $\irr{[t]^\ve w[t]^{-\ve}}=[t]^\ve w'$ and $\tau([t]^\ve\irr w[t]^{-\ve})=\tau([t]^\ve w'[t]^\ve[t]^{-\ve})$. As in Case 3 of \cref{f_1-commutes-with-gm} we prove that $\tau([t]^\ve w'[t]^\ve[t]^{-\ve})=i(\0(\br(t^\ve\f(\irr w)t^\ve t^{-\ve})))\tau(\irr{[t]^\ve w[t]^{-\ve}})$, so the proof of \cref{(f_1-circ-gm^ve_t)(n)=(lb_rho(t)^ve-circ-f_1)(n)} again works.
	
	The remaining cases are analogous to Cases 4--6 of \cref{gm^ve_t-circ-gm^dl_u}.
\end{proof}

\begin{lem}\label{f_1-commutes-with-lb'}
	For any $s\in S$ we have $\f_1\circ\lb'_s = \lb_{\f_2(s)}\circ\f_1$.
\end{lem}
\begin{proof}
	Take arbitrary $s=(t,u)\in S$ and $n=(a,e,w)\in N$. Suppose first that $u=\e$, so that $t\in E(T)$. If $w\ne\e$, then thanks to \cref{lb_(t_u)(a_e_w)=,f_2(t_w)=rho(r(t))rho(t_1)^ve_1...rho(t_n)^ve_n,f_1(a_e_w)=tau(irr(w))i(a),eta_t(a)=i^(-1)(lb_s(i(a))),0=i^(-1).af.pi^(-1),TM1} we have $(\f_1\circ\lb'_s)(n)=\f_1(\af(s)n)=\f_1(\0(t)a,te,w)=\tau(\irr w)i(\0(t))i(a)=\af(\rho(t))\tau(\irr w)i(a)=\lb_{\rho(t)}(\tau(\irr w)i(a))=(\lb_{\f_2(s)}\circ\f_1)(n)$. Otherwise, $(\f_1\circ\lb'_s)(n)=i(\0(t))i(a)=\af(\rho(t))i(a)=\lb_{\rho(t)}(i(a))=(\lb_{\f_2(s)}\circ\f_1)(n)$.
	
	Let $u\ne\e$ and $\irr u=[t_1]^{\ve_1}\dots[t_k]^{\ve_k}$. By \cref{lb_(t_u)(a_e_w)=,f_1-commutes-with-gm-via-rho,f_1(a_e_w)=tau(irr(w))i(a),chi(w)=rho(t_1)^ve_1...rho(t_n)^ve_n,f_2(t_w)=rho(r(t))rho(t_1)^ve_1...rho(t_n)^ve_n,0=i^(-1).af.pi^(-1)}
	\begin{align*}
	(\f_1\circ\lb'_s)(n)&=\f_1(\af(\br(s))(\gm_{t_1}^{\ve_1}\circ\dots\circ \gm_{t_k}^{\ve_k})(n))=i(\0(\br(t)))(\f_1\circ\gm_{t_1}^{\ve_1}\circ\dots\circ \gm_{t_k}^{\ve_k})(n)\\
	&=i(\0(\br(t)))(\lb_{\rho(t_1)^{\ve_1}}\circ\dots\circ\lb_{\rho(t_k)^{\ve_k}}\circ\f_1)(n)\\
	&=\af(\rho(\br(t)))(\lb_{\rho(t_1)^{\ve_1}\dots\rho(t_k)^{\ve_k}}\circ\f_1)(n)=(\lb_{\rho(\br(t))}\circ\lb_{\chi(\irr u)}\circ\f_1)(n)\\
	&=(\lb_{\rho(\br(t))\chi(\irr u)}\circ\f_1)(n)=(\lb_{\f_2(s)}\circ\f_1)(n).
	\end{align*}
\end{proof}

\begin{prop}\label{ext->c->equiv-ext}
	For arbitrary inverse semigroup $T$ and $T$-module $A$ the crossed module extensions $A \xrightarrow{i} N \xrightarrow{\beta} S \xrightarrow{\pi} T$ and $A \xrightarrow{i'} N' \xrightarrow{\beta'} S' \xrightarrow{\pi'} T$ are equivalent.
\end{prop}
\begin{proof}
	By \cref{f_2-homo,f_1-homo-N'->N,middle-square-comm,f_1-commutes-with-lb'} the pair $(\f_1,\f_2)$ gives the equivalence of the crossed module extensions.
\end{proof}

\begin{thrm}\label{H^3_le(T_A)<->E(T_A)}
	Let $T$ be an $F$-inverse monoid and $A$ a $T$-module. There is a bijective correspondence between $H^3_\le(T^1,A^1)$ and $\mathcal{E}_\le(T,A)$.
\end{thrm}
\begin{proof}
	There is a map from $\mathcal{E}_\le(T,A)$ to $H^3_\le(T^1,A^1)$ by \cref{E_le(T_A)->H3_le(T^1_A^1)}, and by \cref{H^3_le(T^1_A^1)->E_le(T_A)} there is a map in the opposite direction. The maps are inverse to each other thanks to \cref{c->ext->cohom-c',ext->c->equiv-ext}.
\end{proof}
	
\section*{Acknowledgments}
The first author was  partially supported by FAPESP of Brazil (Proc. 2015/09162-9) and by CNPq of Brazil (Proc. 307873/2017-0). The second author was partially supported by  CNPq of Brazil (Proc. 404649/2018-1) and by Funda\c{c}\~ao para a Ci\^{e}ncia e a Tecnologia (Portuguese Foundation for Science and Technology) through the project PTDC/MAT-PUR/31174/2017. The third author was financed by Coordenação de Aperfeiçoamento de Pessoal de Nível Superior - Brasil (CAPES) - Finance Code 001. We thank the referee for reading the paper and pointing out various misprints.

	\bibliography{bibl-pact}{}
	\bibliographystyle{acm}
	
\end{document}